\documentclass{amsart}
\usepackage[utf8]{inputenc}
\usepackage{amsfonts}
\usepackage{fontenc}

\pdfoutput=1
\usepackage[headings]{fullpage}
\usepackage{amssymb,epic,eepic,epsfig,amsbsy,amsmath,amscd,color,amsthm}
\usepackage[all,cmtip]{xy}
\usepackage{amssymb,amsmath,amscd,color}
\usepackage{graphicx,calc}
\usepackage{texdraw}
\usepackage{url}
\usepackage{mathrsfs}

\newcommand{\pic}[2][0]{\raisebox{-0.5\height + 2.5pt + #1pt}{\includegraphics{#2.pdf}}}

\usepackage{hyperref, bookmark}

\newtheorem{theorem}{Theorem}[section]

\newtheorem{lemma}[theorem]{Lemma}
\newtheorem{coro}[theorem]{Corollary}
\newtheorem{proposition}[theorem]{Proposition}

\newtheorem{definition}[theorem]{Definition}
\newtheorem{example}[theorem]{Example}

\newtheorem{notation}{Notation}

\newtheorem{remark}[theorem]{Remark}

\newtheorem{question}[theorem]{Question}

\usepackage{tikz}
\tikzstyle cross=[preaction={draw=white, -, line width=6pt}]
\tikzstyle normal=[thick]
\usetikzlibrary{arrows,decorations.markings}
\usepackage{relsize,exscale}
\usetikzlibrary{cd}

\newcommand{\calB}{\mathcal{B}}

\newcommand{\calG}{\mathcal{G}}
\newcommand{\calH}{\mathcal{H}}

\newcommand{\bfa}{\boldsymbol{a}}
\newcommand{\bfb}{\boldsymbol{b}}

\newcommand{\C}{\mathbb{C}}
\newcommand{\N}{\mathbb{N}}
\newcommand{\Q}{\mathbb{Q}}
\newcommand{\R}{\mathbb{R}}
\newcommand{\Z}{\mathbb{Z}}

\newcommand{\Homeo}{\mathop{Homeo}}

\newcommand{\Mod}{\operatorname{Mod}}

\newcommand{\bapp}{\left. \begin{array}{ccc}}
\newcommand{\eapp}{\end{array} \right.}
\newcommand{\bfct}{\left\lbrace \begin{array}{ccc}}
\newcommand{\efct}{\end{array} \right.}
\newcommand{\bred}{\begin{center} \color{red}}
\newcommand{\ered}{\end{center}}
\newcommand{\eredfigure}{\vspace{3cm} \end{center}}

\newcommand{\Conf}{\operatorname{Conf}}

\newcommand{\HBM}{\operatorname{H} ^{\mathrm{BM}}}
\newcommand{\Hnot}{\operatorname{H}}

\newcommand{\Aut}{\operatorname{Aut}}

\newcommand{\End}{\operatorname{End}}
\newcommand{\Ker}{\operatorname{Ker}}
\newcommand{\im}{\operatorname{Im}}
\newcommand{\GL}{\operatorname{GL}}
\newcommand{\ab}{\operatorname{ab}}

\newcommand{\slt}{{\mathfrak{sl}(2)}}
\newcommand{\Uq}{{U_q\slt}}

\newcommand{\Heis}{\mathbb{H}}
\newcommand{\Heisg}{\Heis_g}
\newcommand{\Heismod}{\overline{\Heis}}
\newcommand{\Heismodmodr}{\overline{\Heis}_{g,r}}

\newcommand{\phiHeis}{\varphi^{\Heisg}}

\newcommand{\calHmodr}{\overline{\calH}}

\newcommand{\rhoext}{\widetilde{\rho}}

\newcommand{\bfGamma}{\boldsymbol{\Gamma}}
\newcommand{\bfsimp}{\bfGamma}

\newcommand\arxiv[2]{\href{https://arXiv.org/abs/#1}{\texttt{arXiv:\allowbreak #1} #2}}
\newcommand\doi[2]{\href^{https://doi.org/#1}{#2}}

\newcommand{\bfk}{\boldsymbol{k}}

\newcommand{\thread}{\tilde{\boldsymbol{\gamma}}}
\newcommand{\BM}{{\mathrm{BM}}}
\newcommand{\simp}{\Gamma}

\newcommand{\dualH}{\calH^{\dagger}}

\newcommand{\push}{\operatorname{push}}

\newcommand{\Span}{{\mathrm{Span}}}

\newcommand{\kerarc}{{\mathrm{\mathfrak{Ker}}}}

\makeatletter
\def\namedlabel#1#2{\begingroup
    #2%
    \def\@currentlabel{#2}%
    \phantomsection\label{#1}\endgroup
}
\makeatother

\begin{document}

\raggedbottom

\title{On kernels of homological representations of mapping class groups}%[Homological reps. of MCG]{On representations of Mapping class groups from homology of configurations. }

\author[R. Detcherry]{Renaud Detcherry} 
\address{Institut de Mathématiques de Bourgogne, UMR 5584 CNRS, Université Bourgogne Franche-Comté, F-2100 Dijon, France} 
\email{renaud.detcherry@u-bourgogne.fr}

\author[J. Martel]{Jules Martel} 
\address{Université de Cergy-Pontoise} 
\email{jules.martel-tordjman@cyu.fr}

\setcounter{tocdepth}{3}

\begin{abstract}
We study the kernels of representations of mapping class groups of surfaces  on twisted homologies of configuration spaces. We relate them with the kernel of a natural twisted intersection pairing: if the latter kernel is trivial then the representation is faithful. As a main example, we study the representations $\rho_{n}$ of $\mathrm{Mod}(\Sigma_{g,1})$ based on a Heisenberg local system on the $n$ points configuration space of $\Sigma_{g,1}$, introduced in \cite{BPS21}, and some of their specializations. In the one point configuration case, or when the Heisenberg group is quotiented by an element of its center, we find kernel elements in the twisted intersection form. On the other hand, for $n>2$ configuration points and the full Heisenberg local system, we identify subrepresentations of subgroups of $\Mod(\Sigma_{g,1})$ with Lawrence representations. In particular, we find one of these subgroups, which is isomorphic to a pure braid group on $g$ strands, on which the representations $\rho_n$ are faithful.   
\end{abstract}

\maketitle
\setcounter{tocdepth}{3}

\section{Introduction}

\subsection{History of homological representations for mapping class groups and of their kernels}

Let $S$ be an orientable surface, with zero or one boundary component. Its \textit{mapping class group} denoted $\Mod(S)$ is the group of isotopy classes of orientation preserving diffeomorphisms of $S$. If $S$ has \textit{punctures}, diffeomorphisms have to preserve the set of punctures, and isotopies can move them. The seminal books introducing these groups are \cite{Bir,FarbMargalit}, and they showcase how this object is at the interplay between group theory and topology, since it is involved in many topological constructions due to its topological nature while their group-theoretic properties are mysterious and their representation theory very rich. The present paper is interested in \textit{homological representations} of mapping class groups. Since the homology of a manifold is (usually) functorial regarding diffeomorphisms and invariant under isotopies, many representations of mapping class groups can be constructed using homology. The canonical example is the following morphism:
\begin{equation}\label{E:Torelli}
	\Mod(S) \to \Aut_{\Z} (H_1(S)).
\end{equation}
When nothing is indicated in homology groups, we assume that it refers to the standard $\Z$-homology. When $S$ is of genus $g>1$, this morphism has kernel \textit{the Torelli group} of $S$, usually denoted $\mathcal{I}(S)$. When $S$ has zero or one boundary component, its homology is endowed with a perfect symplectic pairing derived from the algebraic intersection between curves in $S$. In this case, there is the following short exact sequence:
\[
1 \to \mathcal{I}(S) \to \Mod(S) \to \mathop{Sp}(2g,\Z) \to 1,
\]
see all the details in \cite{FarbMargalit}. 

Linear representations of $\Mod(S)$ have many possible uses, for instance in topology for building invariants of topological manifolds of low dimensions. This raises the following question:
\begin{question}[{\cite[Problem~30]{Bir},\cite[Question~1.1]{Margalit}}]\label{Q:the_question}
	Let $S$ be a compact oriented surface. Is $\Mod(S)$ linear ? That is, is there a faithful representation $\rho: \Mod(S) \longrightarrow \mathrm{GL}_d(\C)$ for some integer $d ?$
\end{question}
To tackle this question from a homological perspective, one can generalize the representation \eqref{E:Torelli} by letting $\Mod(S)$ act on other manifolds, and by using twisted homologies. The goal of the present work is to study the kernel of such homological representations and some related aspects of their constructions. 

Historically, a first improvement of homological representations arose from the use of \textit{twisted homology}. In the present paper, a twisted homology of a manifold $X$ is built out of the following input:
\begin{equation}\label{E:local_system}
	\varphi : \pi_1(X) \to G,
\end{equation}
where $G$ is a group and $\varphi$ is a surjective morphism. One can then consider the $\varphi$-twisted homology of $X$ denoted $H_{\bullet}(X;\varphi)$, which carries a natural $\Z[G]$-module structure. When $X$ is a surface $S$ with one boundary component and $\varphi$ is the abelianization morphism, the group $\mathcal{I}(S)$ acts on $H_{1}(S;\varphi)$ and the action is $\Z[H_1(S)]$-linear. Note that $\Z[H_1(S)]$ is isomorphic to the ring of Laurent polynomials in $2g$ variables. Those representations are called  \textit{Magnus representations} \cite{Magnus} (see also \cite{Bir}). The action does not extend to a $\Z[H_1(S)]$-linear action on the whole mapping class group but rather to a \textit{twisted action} where the action of $\Mod(S)$ on the group of coefficients $H_1(S)$ enters the picture. The twisted homology construction and its inherent relation with twisted representations of mapping class group in a very broad sense is precisely discussed in Section \ref{S:constructions_of_reps}. It is often convenient to assume that $S$ has one boundary component since then $\pi_1(S)$ is a free group, and $\Mod(S)$ is isomorphic to a subgroup of $\Aut(\pi_1(S))$ (by the Dehn--Nielsen--Baer theorem \cite[Theorem~8.1]{FarbMargalit}). One gets an explicit construction of $H_{1}(S;\varphi)$ from \textit{Fox differential calculus} on automorphisms of free groups. In Section~\ref{S:closed_surfaces} we introduce a construction of \textit{Magnus representations} but \textit{for closed surfaces}, which is a new construction, to the authors' knowledge. 

When $S$ is a disk with $m$ punctures, denoted $D_m$, the group $\Mod(D_m)$ is well known to be the \textit{braid group on $m$ strands} denoted $\calB_m$ and $\mathcal{I}(D_m)$ to be the \textit{the pure braid group} denoted $\mathcal{PB}_m$. In that case the \textit{Magnus representation} of $\mathcal{PB}_m$ is named \textit{the Gassner representation} (see \cite{Bir}) and takes coefficients in $\Z[H_1(D_m)] = \Z[t_1^{\pm 1},\ldots,t_m^{\pm 1}]$. By evaluating variables $t_i$'s to the same variable $t$, the representation extends to $\calB_m$ and recovers the famous (reduced) \textit{Burau representation}. Notice that the Burau representation was originally defined by assigning matrices to the Artin generators of the braid groups, instead of the previous homologic construction. The Burau representation is faithful for the groups $\calB_2$ and $\calB_3$, which can be checked directly from matrices, and the question for higher cases remained open for a very long time. It is thanks to its homological interpretation that kernel elements were discovered by Long and Paton \cite{LongPaton} for $m\ge 6$ and by Bigelow \cite{Big99} for $m=5$, while the question of the faithfulness for $m=4$ is still open. The kernel elements were exhibited thanks to a \textit{twisted homological intersection pairing}:
\begin{equation}\label{E:twisted_intersection_pairing}
	\langle \cdot , \cdot \rangle_G: H_{\bullet}(X;\varphi) \times (H_{\bullet}(X;\varphi))^\dagger \to \Z[G]
\end{equation}
where $(H_{\bullet}(X;\varphi))^\dagger$ is a dual homology space regarding this pairing, and in the Burau case, $X=D_m$, $\varphi$ is the abelianization morphism, $G$ the abelianization of $\pi_1(S)$. We note that it is not known if Gassner representations are faithful, see e.g. \cite{Knu}, however, their kernels are strictly smaller than that of the Burau representations. For the higher genus case, namely when $S= \Sigma_{g,1}$ is the surface of genus $g$ with $1$ boundary component, the Magnus representations were studied from the homological perspective by Suzuki. He has shown, using again the corresponding twisted homological pairing \eqref{E:twisted_intersection_pairing}, that these representations are unfaithful \cite{S02,S05a,S05b}. Roughly speaking, kernel elements can be associated to pairs of simple closed curves in $\Sigma_{g,1}$ with twisted pairing $0$. Magnus representations enjoy \textit{transvections formulas} for actions of Dehn twists (see for instance \cite[Theorem~4.2]{S05b}). \textit{We generalize these transvection formulas to representations based on the Heisenberg local system in Section \ref{S:twisted_transvections}}. 
\\
\\
There is another direction to construct more powerful homological representations: the group $\Mod(S)$ happens to naturally act on homologies of higher dimensional manifolds: \textit{configuration spaces of $S$}. For an integer $n\in \N$ we denote by $\Conf_n(S)$ the unordered configuration space of $n$-points in $S$, and $\Conf_n^{o}(S)$ the ordered configuration space. A rigorous definition is provided in Section~\ref{sec_uncross}. For the surface $\Sigma_{g,1}$ with positive genus, Moriyama introduced in \cite{Mo07} the family of representations:
\begin{equation}\label{E:Moriyama}
	\mathrm{Mor}_n:\Mod(\Sigma_{g,1}) \to H_n(\Conf_n^o(\Sigma_{g,1}))
\end{equation}
Moriyama showed that the kernel of \eqref{E:Moriyama} is the $n$-th term of \textit{the Johnson filtration} of the mapping class group denoted $\lbrace J_n \rbrace_{n\in\N}$ where $J_i \subset J_{i-1} \subset \Mod(\Sigma_{g,1})$, and $J_1 = \mathcal{I}(\Sigma_{g,1})$. 

All of the Johnson subgroups are non trivial, but they have the property that $\bigcap_{n\in\N} J_n = \lbrace 1 \rbrace$. The sequence of representations on ordered configuration spaces is therefore \textit{asymptotically faithful}.% in the sense that considered all together they detect faithfully each mapping class.
%This Johnson filtration was introduced by D. Johnson in 

The most powerful homological representations arise from using a mix of both twisted homologies and configuration spaces. This was first introduced by R. Lawrence in the case of punctured disks and braid groups in \cite{L90}. The input $\varphi_n$ to twist the homology (Eq. \eqref{E:local_system}) is the abelianization morphism of $\pi_1(\Conf_n(D_m))$ composed with a specialization of variables corresponding to the passage from Gassner to Burau in the case $n=1$. The result is known as the Lawrence representations and consists of the following sequence of representations:
\begin{equation}
	\mathcal{L}_{m,n} : \calB_m \to \End_{\Z[q^{\pm 1}, s^{\pm 1}]} \left( H_n(\Conf_n(D_m),\varphi_n) \right),
\end{equation}\label{Lawrence}
We recall precise definitions of the Lawrence representations in Section~\ref{sec:Lawrence}. The term $n=1$ is the Burau representation. The representations $\mathcal{L}_{m,n}$ have coefficients in the ring $\Z[q^{\pm 1}, s^{\pm 1}]$ encoding the twisted structure. Those representations lead to the following spectacular result: 

\begin{theorem}[{Bigelow \cite[Theorem~1.1]{Big00}, Krammer \cite[Theorem~B]{K02}}]
	For all $m \ge 0$, the kernel of $\mathcal{L}_{m,2}$ is trivial. In particular, braid groups are linear.
\end{theorem}

Therefore, braid groups form the first infinite family of positive answers to Question~\ref{Q:the_question}. This raises the questions whether some representations based on the twisted homology of configuration spaces of surfaces with one boundary component are faithful for larger genus. Indeed, Bigelow's proof proceeds by showing that the natural twisted homological pairing (from \eqref{E:twisted_intersection_pairing}) detects geometric intersection between arcs in $D_m$. Using the representations $\mathcal{L}_{m,2}$, it was later shown that mapping class groups of punctured spheres and of the closed genus $2$ surface are also linear \cite{BB01}. Notice that Bigelow's argument is easily generalized to representations $\mathcal{L}_{m,n}$ for $n > 2$ (see e.g. \cite{Jules2}). 

For higher genus surfaces, Question \ref{Q:the_question} is still open, and one would be tempted to imagine extending Bigelow's result to positive genus surfaces by studying representations of $\Mod(\Sigma_{g,1})$ on the twisted homologies of $\Conf_n(\Sigma_{g,1})$. A first attempt would be to twist by the abelianization of $\pi_1(\Conf_n(\Sigma_{g,1})$. In positive genus, contrary to genus zero, the specialization of variables allowing to pass from the Gassner case to the Burau case in punctured disks does not exist, and thus one must restrict representations to the Torelli subgroup $\mathcal{I}(\Sigma_{g,1})$. \textit{It is however possible to extend those representations the whole mapping class group, if one enlarges the ring of coefficients, as we will explain in Section \ref{sec_uncross}}. Even more disappointing: it is expected that homological representations built from $\Conf_2(\Sigma_{g,1})$ and twisted by the abelianization result in unfaithful representations of $\mathcal{I}(\Sigma_{g,1})$ see \cite[Question~1.2, Footnote]{Margalit}. It is not clear to the authors how to exhibit kernel elements of the described representations but \textit{we explain why Bigelow's strategy cannot work for the representation on the homology of $\Conf_2(\Sigma_{g,1})$ twisted by the abelianization in Section \ref{sec:kernel_sev_pts}}, confirming the expectation of \cite{Margalit}. For finding faithful representations and for successfully applying Bigelow's strategy one would need to twist more. 

In general, to construct representations of $\Mod(S)$ on twisted homologies $H_{\bullet}(X;\varphi)$, the morphism $\varphi$ from \eqref{E:local_system} has to satisfy some compatibility conditions. \textit{This is detailed in Section~\ref{S:constructions_of_reps}}. When $X=\Conf_n(\Sigma_{g,1})$ the compatibility is expressed as follows: elements of $\Mod(S)$ seen as automorphisms of $\pi_1(\Conf_n(S))$ need to preserve the kernel of $\varphi$ (Eq. \eqref{E:local_system}). We hence like to think about $\varphi$ being the quotient of $\pi_1(\Conf_n(S))$ by one of its characteristic subgroup. We are interested in quotients by terms of \textit{the lower central series} of $\pi_1(\Conf_n(S))$. The quotient by the first term corresponds to the abelianisation and was already discussed. One learns in \cite{BGG11,DPS} that the quotient by the next term is isomorphic to a \textit{Heisenberg group} denoted by $\mathbb{H}_g$ when $S=\Sigma_{g,1}$ (notice that it does not depend on $n$). Furthermore, that the lower central series stops at this term, and thus $\mathbb{H}_g$ is the largest nilpotent quotient one could obtain of $\pi_1(\Conf_n(\Sigma_{g,1}))$. One can construct the twisted homology of $\Conf_n(\Sigma_{g,1})$ with coefficients in $\Z[\mathbb{H}_g]$ and study how mapping classes act on it. In the end, one gets representations $\rho_n$ of (subgroups) of $\Mod(\Sigma_{g,1})$ from the action on such Heisenberg twisted homology groups. This construction was introduced in \cite{BPS21} and its study initiated in \cite{BPS21,JulesMarco}. We also deal with this very interesting example as we believe that these Heisenberg twisted homologies are the good level of twisting for Bigelow's strategy (\cite{Big00}) to work in arbitrary genus. Namely \textit{we show that with less configuration points or with less twisting, the corresponding twisted homological pairing cannot detect the geometric intersection of arcs in the initial surface for all genera}. One can therefore not apply Bigelow's strategy in those cases, although it does not necessarily mean that the associated mapping class group representation is unfaithful. 

\textit{As a positive sign, for the representations $\rho_n,$ with $n>1,$ we find a subgroup of the mapping class group, isomorphic to a pure braid group on $g$ strands, that acts faithfully (Sec.~\ref{S:Pg_faithful}, Corollary~\ref{C:Pg_is_faithful}).} 

% (we discuss this subtlety). In the Magnus representation case, when Bigelow's strategy gets obstructed, one can exhibit kernel elements thanks to symplectic transvection formulas, but this does not hold in the other contexts. 

Another family of linear representations for mapping class groups of surfaces is provided by \textit{topological quantum field theories (TQFTs)} of dimension $2+1$, we call them \textit{quantum representations}. They are very much related with topological invariants such as ones for closed $3$-manifolds produced by such a TQFT for instance with the well known Witten--Reshetikhin--Turaev (WRT) invariants. These quantum representations arising from the WRT theory are built using as input a semisimple modular category, and the byproduct representations present a systematic non trivial kernel made of powers of \textit{Dehn twists}. TQFTs in the same spirit but built out of non-semisimple category were constructed in \cite{L94} for surfaces with one boundary later extended in \cite{DGGPR20} to closed surfaces and in \cite{BCGP14} for closed surfaces equipped with a cohomology class (and representations restricted to Torelli groups), and finally generalized to surfaces with one boundary (equipped with a cohomology class) in \cite{JulesMarcoBangxin} in the spirit of \cite{L94}. These non-semisimple quantum representations don't have an obvious kernel like the semisimple WRT ones. They form another candidate to tackle the question of linearity of mapping class groups (Question~\ref{Q:the_question}). In \cite{JulesVerma} (resp. \cite{JulesMarco}) it was shown that homological representations studied in the current paper recover non-semisimple quantum representations for the case of $D_m$ (resp. $\Sigma_{g,1}$). For surfaces equipped with cohomology classes, homological representations are built in \cite[Sec.~6.3.1]{JulesMarco} and conjecturally related to quantum representations of \cite{JulesMarcoBangxin}. It is well known that if the image of a Dehn twist by a representation is diagonal, it has finite order. What made non semisimple quantum representations candidates for faithful linear representations was that one finds non diagonal Jordan blocks in the images of Dehn twists, making them of infinite order. \textit{The representations $\rho_n$ studied in this paper also satisfy this property, see Prop.~\ref{P:infinite_order}}.

\subsection{Content of the paper}

This paper is a study of homological representations of mapping class groups of surfaces. We start by presenting a very general framework in which such representations can be constructed. One typically obtains \emph{crossed} representations of $\mathrm{Mod}(S),$ but we discuss how to make them into genuine representations of the whole mapping class group. We study the role of the twisted homological intersection pairing and we express a faithfulness criterion in terms of the kernel of this pairing. We then study the kernel of this pairing for various setups. Doing so we figure out a few homological representations for which the Bigelow's strategy (\cite{Big00}) for proving faithfulness has no chance of being successful. Along the way we develop combinatorial tools for computing the pairing, and express the actions of Dehn twists as twisted transvections related the pairing, in the one point configuration case. To end up with positive result in the direction of faithfulness, we relate representations of braid subgroups in $\Mod(\Sigma_{g,1})$ to homological representations built from punctured disks à la Lawrence. For one of them, we find a pure braid subgroup of $\Mod(\Sigma_{g,1})$ that acts faithfully.%In this genus zero case, Bigelow's strategy has already worked, thanks to the planar geometry possibilities, but in a certain twisted setup only. 

More precisely, Section~\ref{S:constructions_of_reps} is devoted to the construction of mapping class group representations from twisted homologies in general. In Sec.~\ref{sec_uncross}, we discuss the condition on the local system (see \eqref{E:local_system}) for defining them, first as \textit{crossed representations} of $\Mod(S)$. Then we explain how to uncross them by enlarging the ring of coefficients, and how to linearize them (Prop.~\ref{prop_linearRep}). At each step we pay attention to how the kernel of the representation could grow, and give criterions for the kernel to stay the same. In Sec.~\ref{S:twisted_homologies_structure} we recall the general structure of twisted homologies when surfaces have a boundary component, along the way presenting diagrammatic models for homology classes. We are particularly interested in the case where the local system is the Heisenberg group
$$\Heis_g:=\langle a_1,b_1,\ldots, a_g,b_g ,\sigma | [a_i,a_j],[b_i,b_j], [\sigma,a_i],[\sigma,b_i], [a_i,b_j]\sigma^{-2\delta_{ij}}\rangle.$$
where $\delta_{ij}$ stands for the Kronecker symbol.

%A particular

For each $n\geq 2,$ $\pi_1(\Conf_n(\Sigma_{g,1}))$ admits a surjective morphism to $\Heis_g$ by \cite{BGG11}. This gives rise to homological representations $\rho_n^{\Heis_g}$ of (a subgroup of) $\Mod(\Sigma_{g,1})$, see Section \ref{sec_Heisenberg_case_on_surface_conf}. The coefficients of these representations lie in the group ring $\Z[\mathbb{H}_g]$. We are also interested in the local system $\Heis_{g,r}:=\Heis_g/\langle \sigma^r \rangle,$ which gives rise to representations $\rho_{n,r}$ with coefficients in $\Z[\Heis_{g,r}].$ We study linearizations of these group rings, showing:
\begin{proposition}
	\label{prop:linearization} For $r\geq 1,$ let $\Heis_g$ be the Heisenberg group, and $\Heis_{g,r}=\Heis_g/\langle \sigma^r \rangle.$
	\begin{enumerate}
		\item (Corollary \ref{coro:no_embedding}) There is no faithful algebra map $\Z[\Heis_g] \longrightarrow M_n(\C)$ for any integer $n$.
		\item (Proposition \ref{prop_faithfulRepHeisModr}) For any $r\geq 1,$ there is a faithful algebra map $\Z[\Heis_{g,r}] \longrightarrow M_n(\C),$ for some integer $n.$
	\end{enumerate}
\end{proposition}
The significance of this proposition comes from Proposition \ref{prop_linearRep}, which gives conditions under which linearizations of the rings of coefficients do not enlarge the kernel of homological representations.

We also find a linearization of $\Heis_g$ in Prop.~\ref{prop:suprataut} that allows to construct a linear representation of the full mapping class group, using the uncrossing protocol of Section \ref{sec_uncross}. 

In Section \ref{S:closed_surfaces} we discuss how to construct homological representations for the mapping class groups of closed surfaces (Proposition \ref{prop:closedcase} and Remark \ref{remark:Magnus}). We provide such representations for the action on the one point configuration space, but with local system either the Heisenberg local system or the Magnus local system. The authors are unaware of a similar constructions of homological representations in the litterature, even in the case of the Magnus local system. 

In Section~\ref{S:intersection_form} we add to the panorama one very important homological tool : the twisted intersection pairing (that is roughly presented in \eqref{E:twisted_intersection_pairing} in the introduction). We define it and give its classical properties in Sec.~\ref{S:intersection_form_def}. We then (Sec.~\ref{sec:inter_form_def}) present dual families of homology classes, associated to simple arcs in $S$ with ends on $\partial S.$ This allows another version of the homological representations construction as the quotient of a space of arcs by the kernel of a bilinear form (Prop.~\ref{P:Universal_construction_of_homol_reps}), similar to the so called ``universal construction'' used in the context of TQFTs. We then provide a faithfulness criterion for homological representations of mapping class groups fully inspired by Bigelow's strategy to obtain faithfulness of braid group representations in \cite{Big00}. 

\begin{theorem}[Theorem~{\ref{T:definite_implies_faithful}}]\label{theoremark_faithfulness_criterion}
	Let $\kerarc(\langle \cdot , \cdot \rangle_{n,G})$ be the set of pairs of isotopy classes of simple arcs with non-zero geometric intersection in $S=\Sigma_{g,1}$ that are sent to zero by the $G$-twisted intersection pairing $\langle \cdot, \cdot \rangle_{n,G}$ (see \eqref{E:twisted_intersection_pairing}) on $X=\Conf_n(S)$. If $\kerarc(\langle \cdot , \cdot \rangle_{n,G})$ is empty then $\rho_n^G$ is faithful. 
\end{theorem}

In other words, if the twisted intersection pairing detects when simple arcs in the surface have non-zero geometric intersection then the corresponding twisted homological representation is faithful. This raises interest in studying the ``kernel'' of this intersection pairing (notice that this is not a kernel in the traditional sense, just a set of pairs of simple arcs).% and \textit{by kernel we mean pairs of arcs that cannot be disjoint up to isotopy but which have twisted intersection zero}. This is the subject of many of the following sections, beside Section~\ref{sec:Lawrence}.  

Section~\ref{S:single_point} is interested in the $n=1$ case, namely the first level of the representation coming from single point configurations on $\Sigma_{g,1}.$ One of its conclusion is that Bigelow's strategy cannot be performed in this case. Namely, in the Heisenberg case for $n=1$ points of configuration, we show that $\kerarc(\langle \cdot , \cdot \rangle_{1,\Heis_g})$ is not empty, when the genus $g$ is at least $6$ (Theorem~\ref{T:kernel_n=1}). Indeed, we find an explicit pair of non-disjoint simple arcs with vanishing twisted intersection pairing. In Sec.~\ref{S:kernel_n=1_mod} we exhibit simpler elements of $\kerarc(\langle \cdot , \cdot \rangle_{1,\Heis_{g,2k}})$, and in this case also kernel elements of the representations $\rho_{1,2k}$. On that matter, we show how in the Heisenberg case, the usual \textit{transvection formula} for the homological action of a Dehn twist on a curve gets twisted by a power of $\sigma$ and why it does not allow to recover an element in $\Ker(\rho_1^{\Heis_g})$ from a pair in $\kerarc(\langle \cdot , \cdot \rangle_{1,\Heis_g})$ (Prop.~\ref{P:twisted_transvection_formula_1pt}). Other direct consequences of this twisted transvection formula are a criterion for two separating Dehn twists to be distinguished by the homological representation in terms of how they cut the surface into connected components (Prop.~\ref{prop:close_Dehn_twists}) and that the images of separating Dehn twists have infinite order (Prop.~\ref{P:infinite_order}). All this section relies on the nice behavior of the Heisenberg local system regarding separating loops (Prop.~\ref{prop:HeisenbergSimpleCurve}).  

In Section~\ref{sec:several_pts} we tackle the level of several configuration points ($n>1$). In Proposition~\ref{prop:kernel_inter_form_npt} we show that $\kerarc(\langle \cdot , \cdot \rangle_{n,\Heis_{g,2k}})$ is non empty for $g>k$. Along the way we propose a nice combinatorial formula for computing $\langle \cdot , \cdot \rangle_{n,\Heis_g}$ on the homology classes associated to simple arcs embedded in $\Sigma_{g,1}$. The formula depends only on $n$ and some simple combinatorial data regarding how the arcs intersect.  %These elements are simply those Suzuki elements in the usual kernel of Magnus representations and in the single point case, that persist in twisted homology of configuration spaces.  

Here is a summary of our contributions to $\kerarc(\langle \cdot , \cdot \rangle_{n,G})$ motivated by the above Theorem~\ref{theoremark_faithfulness_criterion}.% $\langle \cdot , \cdot \rangle_G$ (Eq.~\eqref{E:twisted_intersection_pairing}) in the case $X= \Conf_n(\Sigma_{g,1})$. 

\begin{theorem}
	%Let $\langle \cdot , \cdot \rangle_G$ be the twisted intersection pairing (Eq.~\eqref{E:twisted_intersection_pairing}) in the case $X= \Conf_n(\Sigma_{g,1})$.
	\begin{itemize}
		\item ($G=\mathbb{H}_g, n=1$) $\kerarc(\langle \cdot , \cdot \rangle_{1,\Heis_g})$ is non empty for $g\ge 6$ (Theorem~\ref{T:kernel_n=1}, we exhibit a pair of arcs).
		\item ($G=\mathbb{H}_{g,2k}, n=1$) $\kerarc(\langle \cdot , \cdot \rangle_{1,\Heis_{g,2k}})$ is non empty for $g>3k$. We exhibit some elements that lead to actual elements of $\Mod(\Sigma_{g,1})$ in $\Ker(\rho_{1,2k})$ (Prop.~\ref{prop:kernel_modsigma2k}).
		\item ($G=\mathbb{H}_g, n=1$) There are \textit{twisted transvection formulas} relating the homological actions of separating Dehn twists and the twisted intersection pairing (Prop.~\ref{P:twisted_transvection_formula_1pt}). They induce a criterion for images of two separating Dehn twists by the homological representation to be different, expressed in terms of how they cut the surface into connected components (Prop.~\ref{prop:close_Dehn_twists}). 
		\item ($G=\mathbb{H}_{g,2k} , n>1$) $\kerarc(\langle \cdot , \cdot \rangle_{n,\Heis_{g,2k}})$ is non empty for $g>k$. (Prop.~\ref{prop:kernel_inter_form_npt}).
	\end{itemize}
\end{theorem}

In \cite{L90}, R. Lawrence has invented representations of braid groups (denoted $\mathcal{L}_{m,n}$ in \eqref{Lawrence}), seen as mapping class groups, on twisted homologies of punctured disks. They have very much inspired the present constructions, even more since Bigelow has proved their faithfulness in \cite{Big00} by analyzing their associated twisted intersection pairings. We note that there are subgroups of $\Mod(\Sigma_{g,1})$ isomorphic to braid groups, on which one can study the restriction of the representations $\rho_n^{\Heisg}$, or subrepresentations of those restrictions. Section~\ref{sec:Lawrence} aims at recognizing such braid subgroups representations as Lawrence representations. Theorem~\ref{thm:pure_braid_group} states that $\rho_n^{\Heisg}$ restricted to some (pure) braid subgroup with $g$ strands of $\Mod(\Sigma_{g,1})$ that is generated by separating Dehn twists  is isomorphic to a specialization of the Lawrence representation $\mathcal{L}_{g,n}$ built from homologies of $\Conf_n(D_g)$, with coefficients in a ring with one variable. 
Theorem \ref{thm:pure_braid_group2} shows that the restriction of $\rho_n^{\Heisg}$ to another pure braid subgroup of $\Mod(\Sigma_{g,1})$ with $g$ strands is isomorphic to a full Lawrence representation $\mathcal{L}_{g,n}$, and Corollary \ref{C:Pg_is_faithful} that $\rho_n^{\Heisg}$ restricted to this subgroup is faithful for $n>1.$ 
Finally, we also find a pure braid subgroup of $\Mod(\Sigma_{g,1})$ this time on $2g$ strands generated by both separating and non separating twists and we show that its action via a restriction of $\rho_n^{\Heisg}$ is isomorphic to an evaluation of Lawrence representation in a ring with $g+1$ variables (Theorem~\ref{thm:pure_braid_group3}). 

%Finally, {\color{red} TO COMPLETE}

\begin{theorem}
	\begin{itemize}
		\item (Sec.~\ref{sec:pure_braid_Lawrence}) There is a subgroup of $\Mod(\Sigma_{g,1})$ isomorphic to a pure braid group on $g$ strands that is generated by separating Dehn twists, and such that restricting $\rho_n^{\Heis_g}$ to it recovers the representation $\mathcal{L}_{g,n}$ evaluated at $s=\sigma^{-2}$. 
		\item (Sec.~\ref{S:Pg_faithful}) There is a pure braid subgroup on $g$ strands of $\Mod(\Sigma_{g,1})$. The restriction of $\rho_n^{\Heis_g}$ to it is fully isomorphic to $\mathcal{L}_{g,m}$. For $n>1$ this restriction is faithful (Coro.~\ref{C:Pg_is_faithful}). 
		\item (Sec.~\ref{sec:P2g}) There is a subgroup of $\Mod(\Sigma_{g,1})$ isomorphic to a pure braid group on $2g$ strands, such that restricting $\rho_n^{\Heis_g}$ to it recovers the representation $\mathcal{L}_{g,n}$ evaluated in a ring of Laurent polynomials with $g+1$ variables (Theorem~\ref{T:P2g=evaluated_Lawrence}). 
	\end{itemize}

\end{theorem}

%{\color{red} TO ADD AFTERWARDS: A word on the subgroup acting faithfully, and a stated theorem.}

\subsection*{Acknowledgments}

Both authors thank Q.~Faes for helpful discussions, and I. Agol for helping with the proof of Lemma \ref{lem_f_not_trivial_makes_arc_intersect}. Most of this work was done during the one year journey of the second author in Dijon. Both authors were supported by the ANR project AlMaRe (ANR-19-CE40-0001-01), and the first author also by the project "CLICQ" of R\'egion Bourgogne-Franche Comt\'e.

\section{Constructions of homological representations}\label{S:constructions_of_reps}

\subsection{Crossed and uncrossed representations on twisted homologies}%Uncross isotopic actions on twisted homologies in general}
\label{sec_uncross}
\begin{definition}[Isotopic actions]\label{def_isotopic_actions}
Let $\mathcal{G}$ be a group, let $X$ be a manifold and let $x\in X$ be a base point. We say $\mathcal{G}$ \emph{isotopically acts} on $X$ if there is a morphism $\mathcal{G}\longmapsto \mathrm{Mod}(X,x),$ where $\mathrm{Mod}(X,x)$ is the group of homeomorphisms of $X$ fixing $x$ up to isotopies fixing $x.$ 
\end{definition}
Although we can state our constructions at this degree of generality, to be more concrete, we will later only consider the case $\mathcal{G}=\mathrm{Mod}(\Sigma)$ where $\Sigma$ is a surface possibly with boundary and punctures. We keep in mind the following examples of isotopic actions.
\begin{example}\label{ex_isotopic_actions}
\begin{enumerate}
\item The action of the mapping class group $\mathrm{Mod}(\Sigma_{g,n})$ of a compact oriented surface $\Sigma_{g,n}$ of genus $g$ with $n\geq 1$ boundary components on $X=\Sigma_{g,n}$ itself, with base point an arbitrary point $x\in \partial \Sigma_{g,n}.$
\item The action of $\mathrm{Mod}(\Sigma)$ (where $\Sigma$ has non-empty boundary) on the space of unordered configurations in $\Sigma$ : 
\begin{equation}
X=\mathrm{Conf}_n(\Sigma) := \left\lbrace \lbrace z_1, \ldots , z_n \rbrace \subset \Sigma | z_i \neq z_j \right\rbrace
\end{equation}\label{def_confspaces_in_ex_isotopic_actions}
with base point $x$ being the set $x=\lbrace x_1,\ldots ,x_n\rbrace$ where each $x_i \in \partial \Sigma$. Here, the action of $\Homeo(\Sigma)$ is extended coordinate by coordinate to $\Conf_n(\Sigma)$, as well as isotopies, which provides the isotopic action of $\Mod(\Sigma)$ on $\Conf_n(\Sigma)$. 
\item The action of $\mathrm{Mod}(\Sigma)$, where $\Sigma$ is a surface with non-empty boundary on either $U(\Sigma),T^*(\Sigma),PU(\Sigma),PT(\Sigma),$ where $U(\Sigma)$, $T^*(\Sigma)$ are respectively the unit tangent bundle and the non-zero tangent bundle of $\Sigma$ (the fiber over any point of $\Sigma$ is the set of non-zero tangent vectors at this point), and $PU(\Sigma), PT(\Sigma)$ are their respective projective versions. For the case of tangent bundles, we consider the smooth version of the mapping class group. Note that to consider the unit tangent bundle, one needs to make an arbitrary choice of a Riemann structure on $\Sigma,$ and the actions corresponding to different choices will be homotopic.
The base point can be taken to be any $\vec{x}=(x,v)\in U(\partial \Sigma),T^*(\partial \Sigma),  PU(\partial \Sigma) $ or $PT(\partial \Sigma)$. One could also replace $\Sigma$ by its configuration spaces in all of the above bundles. 
\end{enumerate}
\end{example}
The isotopic action of $\mathcal{G}$ on $X$ induces a map $\mathcal{G} \longmapsto \mathrm{Aut}(\pi_1(X,x))$. Let $N$ be a normal subgroup of $\pi_1(X,x)$. 
\[
\text{We make the assumption that } N \text{ is stable under the action of } \mathcal{G}. \tag{H}\label{H}
\]
We give examples for such an $N$.
\begin{example}\label{example_stabilized_subgroups_ofpi1}
\begin{enumerate}
\item $X$ is any $CW$-complex on which $\mathcal{G}$ isotopically acts and $N$ is a characteristic subgroup of $\pi_1(X,x)$ (hence stable by any automorphism). % (in particular, it is stable by the automorphisms of $\pi_1(X,x)$ coming from $\mathcal{G}$).
\item A special case of characteristic subgroup is $N=\Gamma_k \pi_1(X,x),$ where $\Gamma_k H=[H,[H,[\ldots,[H,H]]]]$ stands for the $k$-th element of the lower central series of a group $H$.
\item Let $\mathcal{G}=\mathrm{Mod}(\Sigma_{g,1})$ be the mapping class group of a compact oriented surface with genus $g$ and $1$ boundary component and $X= \Conf_n(\Sigma_{g,1})$. In \cite{BGG11,BPS21,DPS,JulesMarco} one finds morphisms $\pi_1(\Conf_n(\Sigma_{g,1})) \longmapsto \mathbb{H}_g$ where $\mathbb{H}_g$ is the Heisenberg group: 
\begin{equation}
\mathbb{H}_g=\langle a_1,b_1,\ldots, a_g,b_g ,\sigma | [a_i,a_j],[b_i,b_j], [\sigma,a_i],[\sigma,b_i], [a_i,b_j]\sigma^{-2\delta_{ij}}\rangle.
\end{equation}\label{eq_def_Heisenberg_group}
 The kernel is proved to be stable under the $\mathrm{Mod}(\Sigma_{g,1})$ action for any $n\geq 2$. This example is central in this paper and rigorously (re)constructed in Sec. \ref{sec_Heisenberg_case_on_surface_conf}.% and that is further studied throughout the paper. 
\end{enumerate}
\end{example}
Let $G:=\pi_1(X,x)/N$ and $X_G$ be the regular cover of $X$ corresponding to the subgroup $N\triangleleft \pi_1(X,x)$. It can be thought as the universal cover $\widetilde{X}$ of $X$ endowed with a $\pi_1(X,x)$-action (by precomposition of paths) that is modded out by the action of $N$. The group of deck transformations for $X_G$ is then isomorphic to $G$. Since $N$ is stabilized by any $f\in \mathcal{G}$ (see \eqref{H}), any homeomorphism $f$ induces a homeomorphism $f_G$ of $X_G$ (it is precisely the lifting property for regular covers). Moreover, any two homotopic $f,g$ induce homotopic homeomorphisms of $X_G.$
Furthermore, looking at actions on $\pi_1$'s, we get a morphism:
$$ \begin{array}{rcl}
\mathcal{G} & \longmapsto  & \mathrm{Aut}(\pi_1(X)/N) \cong \mathrm{Aut}(G)
\\ f & \longmapsto & f_* \end{array}$$
By functoriality of the homology, $f_G$ induces a map on $H_*(X_G,\mathbb{Z})$ that we still denote $f_G$ by abuse of notation. The deck transformations give $H_*(X_G,\mathbb{Z})$ the structure of a $\mathbb{Z}[G]$-module, but the map $f_G$ is not a morphism of $\mathbb{Z}[G]$-modules in general, rather it satisfies:
\begin{equation}\label{eq_the_action_is_crossed}
f_G(g\cdot x)=f_*(g)\cdot f_G(x).
\end{equation}
Let $K_G:=\Ker(\calG \to \Aut(G))$. If $f\in K_G$ then $f_G$ yields a morphism of $\mathbb{Z}[G]$-modules on $H_*(X_G,\mathbb{Z})$ (straightforward from the above Property \eqref{eq_the_action_is_crossed}). Therefore, there is a map:
\begin{equation}
\begin{array}{rcl}
\rho_{G} : K_G \triangleleft \mathcal{G} & \longmapsto  & \mathrm{Aut}_{\mathbb{Z}[G]}(H_*(X_G,\mathbb{Z}))
\\ f & \longmapsto & f_G \end{array}
\end{equation}\label{eq_rep_of_K_G}
If furthermore $H_*(X_G,\mathbb{Z})$ is a free $\mathbb{Z}[G]$-module, up to a choice of basis the automorphism group on the right hand side is identified with $\mathrm{GL}_N(\mathbb{Z}[G])$. For $\mathcal{B}=\lbrace e_1,\ldots,e_N \rbrace$ a $\Z[G]$-basis of $H_*(X_G,\mathbb{Z}),$ the identification is
$$f \longrightarrow \mathrm{Mat}_{\mathcal{B}}(f)$$ 
where $\mathrm{Mat}_{\mathcal{B}}(f)=a_{ij}^*$ where $a_{ij}$ is the coefficient of $f(e_j)$ in $e_i,$ and $*$ is the anti-involution on $\Z[G]$ sending $g\in G$ to $g^{-1}.$  

 The above defined $\rho_G$ is a $\Z[G]$-representation of $K_G$ while the action of $\calG$ on $H_*(X_G,\mathbb{Z})$ is only a \emph{crossed representation}. The following proposition presents a tautological way to  \emph{uncross} these \emph{isotopic representations on twisted homologies} when the homology is a free module. 
\begin{proposition}[Uncross homological representations]\label{prop_uncross_in_general}
Assume $H_*(X_G,\mathbb{Z})$ is a free $\mathbb{Z}[G]$-module with a given basis $\mathcal{B}=\lbrace e_1,\ldots, e_N \rbrace.$ Let also $M_G=\mathrm{Im}(\calG \to \Aut(G)).$ For $f\in \calG$, we define $\widetilde{\rho_G}(f)\in \mathrm{GL}_N(\mathbb{Z}[G\rtimes M_G])$ by the following formula:
$$\widetilde{\rho_G}(f)=\mathrm{Mat}_{\mathcal{B}}(f_G) (f_* I_N).$$
Then $\widetilde{\rho_G}$ is a representation of $\mathcal{G}$ that coincides with $\rho_G$ on $K_G$. Furthermore, $\mathrm{Ker}(\widetilde{\rho_G})=\mathrm{Ker}(\rho_G)$
\end{proposition}
\begin{proof}
From $f_G(g\cdot x)=f_*(g)\cdot f_G(x)$, we have
$$\mathrm{Mat}_{\mathcal{B}}((fg)_G)=\mathrm{Mat}_{\mathcal{B}}(f_G)\mathrm{Mat}_{\mathcal{B}}(g_G)^{f_*}$$
where for $\psi \in \mathrm{Aut}(G)$ and $M\in \mathrm{GL}_N(\mathbb{Z}[G]),$ the matrix $M^{\psi}$ is obtained by applying $\psi$ to all coefficients of $M$ (this multiplication rule justifies the name of \emph{crossed representation}). Note that  
$$\mathrm{Mat}_{\mathcal{B}}(g_G)^{f_*}=(f_* I_N)\mathrm{Mat}_{\mathcal{B}}(g_G) (f_*^{-1} I_N)$$
as a composition of matrices in $\mathrm{GL}_N(\mathbb{Z}[G\rtimes M_G]).$
Hence we get
\begin{align*}
\widetilde{\rho}_G(fg)& = \mathrm{Mat}_{\mathcal{B}}(f_Gg_G)((fg)_* I_N)
\\ & =\mathrm{Mat}_{\mathcal{B}}(f_G)(f_* I_N)\mathrm{Mat}_{\mathcal{B}}(g_G)(f_*^{-1} I_N) ((fg)_* I_N)
\\ & = \widetilde{\rho}_G(f) \widetilde{\rho}_G(g).
\end{align*}
The definition clearly implies that $\rho_G=\widetilde{\rho_G}$ on $K_G,$ by definition of $K_G$.
Finally, to get that $\mathrm{Ker}(\widetilde{\rho_G})=\mathrm{Ker}(\rho_G),$ it suffices to show that $\mathrm{Ker}(\widetilde{\rho_G})$ is a subgroup of $K_G.$ Notice that $\widetilde{\rho_G}(f) \in \mathrm{GL}_N(\mathbb{Z}[G])$ if and only if $f_*=1,$ which proves the claim.
\end{proof}
The kind of representations we get here differ from the classical notion of linear representations of groups, since a linear representation of $\mathcal{G}$ is typically defined to be a representation $\mathcal{G}\to GL_N(\mathbb{K})$ where $\mathbb{K}$ is a field (often $\mathbb{K}=\C$). Here we allow the coefficients of the representation to lie in a group algebra $\Z[G],$ with the typical assumption that $G$ will be a simpler group than $\mathcal{G}.$ It is for example possible to recover linear representations over $\C$ of $\mathcal{G}$ from linear representations over $\C$ of $G$.
\begin{proposition}\label{prop_linearRep} Let $\widetilde{\rho_G}$ be a representation of $\mathcal{G}$ as described in Theorem \ref{prop_uncross_in_general}. Let $\iota: G\rtimes M_G\longmapsto \mathrm{GL}(V)$ be a finite dimensional $\C$-representation of $G\rtimes M_G,$ that is naturally extended to an algebra morphism $\iota: \mathbb{Z}[G\rtimes M_G]\longmapsto \mathrm{End}(V).$ Let us define 
$$\widetilde{\rho_V} : \begin{array}{ccc}
 \mathcal{G} & \longmapsto & GL(V^{\otimes N})
\\ g & \longmapsto & \iota\left( \rho_G(g)\right)\end{array},$$

hence $\rho_V$ is a $\C$-linear representation of $\mathcal{G}.$ Moreover:
\begin{itemize}
\item[(i)] If $\iota(m)\notin \iota(\Z[G])$ for any $m\neq 1 \in M_G,$ then $\mathrm{Ker}(\widetilde{\rho_V})$ is a subgroup of $K_G.$
\item[(ii)] If furthermore $\iota$ is a faithful representation $\Z[G]\mapsto \mathrm{End}(V),$ then $\mathrm{Ker}(\widetilde{\rho_V})=\mathrm{Ker}(\widetilde{\rho_G}).$ 
\end{itemize}
\end{proposition}
\begin{proof}
The fact that $\widetilde{\rho_V}$ is a $\C$ representation of $\mathcal{G}$ follows from the fact that $\widetilde{\rho_G}$ is a representation of $\mathcal{G}$ and $\iota$ is an algebra morphism $\Z[G]\mapsto \mathrm{End}(V).$

Next, assume that for $m\in M_G$, $\iota(m) \in \iota(\Z[G]) \Rightarrow m=1.$ Let $g\in \mathcal{G},$ we have $\widetilde{\rho_V}(g)=\iota(\mathrm{Mat}_{\mathcal{B}}((f_G)_*)(\iota(\psi_G(g))I_N).$ If $g\in \mathrm{Ker}(\widetilde{\rho_V}),$ then the coefficients of $\widetilde{\rho_V}(g)$ are in $\iota(\Z[G])$, so $\psi_G(g)=1$ and $g\in K_G.$

Finally, if furthermore $\iota$ is faithful over $\Z[G],$ then $\widetilde{\rho_V}$ restricted to $K_G$ has the same kernel as $\rho_G,$ so $\mathrm{Ker}(\widetilde{\rho_V})=\mathrm{Ker}(\widetilde{\rho_G})$ since both kernels are included in $K_G.$ 
\end{proof}

\begin{remark}
	(Persistence of the kernel)
\begin{enumerate}
	\item The condition (ii) in Proposition \ref{prop_linearRep} is stronger than asking that $\iota$ is faithful as a representation $G\rightarrow GL(V).$
	\item When (ii) is satisfied, we will say that the representation $\iota$ has the {\em natural persistence of the kernel} property. Note that it is a priori possible that a linearization of the representations $\rho_G$ keeps the same kernel even when this property is not satisfied.
\end{enumerate}
\end{remark}

\subsection{Twisted homologies of configuration spaces and structure}\label{S:twisted_homologies_structure}

Now we focus on the particular case where $X$ is the configuration space of a surface with one boundary. 
Let $\varphi: \pi_1(\Conf_n(\varSigma_{g,1})) \to G$ an onto group morphism. We will denote $\widehat{\Conf_n}(\varSigma_{g,1})$ the associated regular cover to simplify notation. In this section we study the structure of:
\[
\HBM_{\bullet} \left( \Conf_n( \varSigma_{g,1} ) , \Conf_n(\varSigma_{g,1})^- ; \varphi \right)
\]
as a $\Z\left[ G \right]$-module (where $\HBM_{\bullet}$ stands for the {\em twisted Borel--Moore relative homology}). The relative part is defined as follows:
\[
\Conf_n(\varSigma)^- := \lbrace \lbrace z_1, \ldots , z_n \rbrace \in \Conf_n(\varSigma), \exists i, z_i \in \partial^- \varSigma \rbrace
\]
where $\partial^- \varSigma$ is an embedded interval in $\partial \varSigma$.
%is the union of two embedded and disjoint intervals in $\partial \varSigma$. 

%\subsubsection{Twisted homology classes from families of arcs, and structure of the homology}\label{sec_twisted_classes_from_families_of_arcs}
We first define $\varphi$-twisted Borel--Moore homology classes in configuration spaces from graphs supported by a given family of non intersecting arcs of the surface $\varSigma_{g,1}$ with ends in $\partial^- \varSigma_{g,1}$. We define them in general as we will use them extensively in what follows. This section is a summary of \cite[Sec.~2.2]{JulesMarco} and we have adopted their conventions. 

First of all, let us denote by $I = [0,1]$ the unit interval, by $\mathring{I} = I \smallsetminus \partial I$ its interior. Our diagrams consist in three inputs that we define. 

\begin{definition}\label{def:diagrammatic_twisted_class}
A diagrammatic twisted class of $\varSigma_{g,1}$ is made of three elements:
\begin{enumerate}
\item For every integer $m \ge 0$, we define an \textit{$m$-multisimplex} $\bfsimp$ to be an ordered family $(\simp_1,\ldots,\simp_m)$ of disjoint proper embeddings $\simp_1,\ldots,\simp_m : I \to \varSigma_{g,1}$ such that $\simp_\ell$ embeds $\partial I$ into $\partial_- \varSigma_{g,1} $ for $1 \le \ell \le m$. 
\item For all integers $m,n \ge 0$, we say that an ordered family $\bfk=(k_1,\ldots,k_m) \in \N^m$ such that $k_1 + \ldots + k_m = n$ provides an \textit{$m$-partition} of $n$. We label components of the $m$-multi-simplex using this partition.
\item We define a \textit{thread} $\thread$ of a $\bfk$-labeled $m$-multisimplex $\bfsimp$ to be an ordered family $(\tilde{\gamma}_1,\ldots,\tilde{\gamma}_n)$ of disjoint embeddings of $I$ in $\varSigma_{g,1}$ such that $\tilde{x}_i(0) = \xi_i$ and $\tilde{x}_i(1) \in \Gamma_j$ for some $j$ in such a way that there are precisely $k_j$ of the $\gamma_i$'s ending in $\Gamma_j$. % \in \bfsimp^{\times \bfk}(\frac 12,\ldots,\frac 12)$ for all integers $1 \le i \le n$.
\end{enumerate}
We denote $\bfsimp^{(\bfk)}_{\thread}$ such a diagram to show the three inputs. 
\end{definition}

%The following picture is an example of a diagrammatic class. 
Here is an example of such diagram for $n=5$ and $g=1$:
\begin{equation}\label{fig:example_homology_class}
 \pic{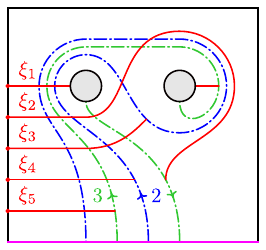}
\end{equation}

Now we explain how these diagrams describe a twisted relative homology class in $\HBM_{n} \left( \Conf_n( \varSigma_{g,1} ) , \Conf_n(\varSigma_{g,1})^- ; \varphi \right)$, for any $\varphi$. If one removes the red threads, the diagram still describes a relative homology class but non twisted, i.e. in $\HBM_{n} \left( \Conf_n( \varSigma_{g,1} ) , \Conf_n(\varSigma_{g,1})^- ; \Z \right)$. %, we begin with this description. 
Indeed, every $m$-multi-simplex $\bfsimp = (\simp_1,\ldots,\simp_m)$, together with an $m$-partition $\bfk = (k_1,\ldots,k_m)$ of $n$, induces an embedding
\begin{equation*}
\begin{array}{cc}
\bfsimp^{\times \bfk} : 
% \mathring{I}^{\times n} & \xrightarrow{\text{ } \bfDelta^{\bfk} \text{ }} 
\Delta^{k_1} \times \ldots \times \Delta^{k_m} & \xrightarrow{\Gamma_1^{k_1} \times \cdots \times \Gamma_m^{k_m}} \Conf_n(\varSigma_{g,1}), % \\*
 \end{array}
\end{equation*}
where for an integer $k$:
\[
 \Delta^k := \{ (t_1,\ldots,t_k) \in \R^k \mid 0 < t_1 < \ldots < t_k < 1 \}
\]
is the standard open $k$-dimensional simplex in $\R^k$, %The above embedding is first a basic homeomorphism $\bfDelta^{\bfk}$ (see Rem.~\ref{R:parametrization_of_embeddings}) of a hypercube to a product of simplices. 
and if $\Gamma$ is an embedded arc with ends in $\partial^- \varSigma_{g,1}$, there is a natural embedding for any integer $k$:
\[
\begin{array}{rcl}
\Gamma^k : \Delta^k & \to & \Conf_k(\varSigma_{g,1}) \\
(t_1, \ldots , t_k) & \mapsto & \lbrace \Gamma(t_1) , \ldots , \Gamma(t_n) \rbrace. 
\end{array}
\]
If one thinks about $\Delta^k$ being the (ordered) {\em configuration space of $k$ points in $I$} then the image of $\Gamma^k$ is that of $k$ points in the embedded interval $\Gamma(I)$ (disordered afterwards). The definition of an $m$-multi simplex ensures that $\Gamma_1^{k_1} \times \cdots \times \Gamma_m^{k_m}$ is well defined. As the faces of the simplices are either sent to infinity or to $\Conf_n(\varSigma_{g,1})^-$, then this embedding of a product of simplices defines a Borel--Moore cycle relative to $\Conf_n(\varSigma_{g,1})^-$. Indeed, in Borel--Moore homology, closed submanifold are cycles (even those going to infinity).

%
%\begin{remark}\label{R:parametrization_of_embeddings}
%The homeomorphism between the hypercube and a product of simplices is defined as follows:
%\begin{align*}
% \bfDelta^k : \mathring{I}^k &\to \Delta^k \\*
% (t_1,\ldots,t_k) &\mapsto (\Delta^k_1(t_1),\ldots,\Delta^k_k(t_k)),
%\end{align*}
%where $\Delta^k_1(t_1) := t_1$ and, recursively, $\Delta^k_i(t_i) := t_i(1-\Delta^k_{i-1}(t_{i-1})) + \Delta^k_{i-1}(t_{i-1})$ for all integers $1 < i \le k$.
%
%It might seems superfluous at first glance while the orientation of the product of simplices matters. Hence the order of the product actually matters and precomposing by this homeomorphism avoid more markings. 
%\end{remark} 
 
A labeled multi-simplex $\bfsimp^{\times \bfk}$ yields a homology class in $H^\BM_n(\Conf_n(\varSigma_{g,1}),\Conf_n(\varSigma_{g,1})^-)$, it remains to use the thread to choose a lift of it to $\widehat{\Conf}_n(\varSigma_{g,1})$ where $\hat{p} : \widehat{\Conf}_n(\varSigma_{g,1}) \to \Conf_n(\varSigma_{g,1})$ is the regular cover associated with $\Ker \varphi$. A thread is a path $\thread : I \to X_{n,g}$ from $\underline{\xi}$ to an element $x$ in the image of $\bfsimp^{\times \bfk}$. It therefore selects a point denoted $\hat{\gamma} \in \widehat{\Conf}_n(\varSigma_{g,1})$ in the fiber $\hat{p}^{-1}(x)$. 

The diagrammatic twisted class denoted $\hat{\bfsimp}^{(\bfk)}_{\thread}$ associated with the $\bfk$-labeled $m$-multisimplex $\bfsimp^{\times \bfk}$ threaded by $\thread$ is the homology class of the unique lift:  %, where $\tilde{p} : \widetilde{X}_{n,g} \to X_{n,g}$ denotes the regular cover corresponding to the kernel of the ring homomorphism $\varphi_{n,g}^\dHeis : \Z[\pi_{n,g}] \to \Z[\dHeis_g]$ given by Lemma~\ref{L:phi};
\[
\hat{\bfsimp}^{\times \bfk} : \mathring{I}^n \to \hat{\Conf}_n(\varSigma_{g,1})
\]
of ${\bfsimp}^{\times \bfk}$ that contains $\hat{\gamma}$. We refer the reader to Sec.~2.2.1 of \cite{JulesMarco} for more details on these diagrammatic notations for homology classes. 

%\begin{remark}\label{R:thread_and_orientation}
%The thread also sets a permutation $\thread \in \mathfrak{S}_n$, together with an associated map 
%  \begin{align*}
%   \thread : \mathring{I}^n &\to \mathring{I}^n \\*
%   (t_1,\ldots,t_n) &\mapsto (t_{\thread(1)},\ldots,t_{\thread(n)}),
%  \end{align*}
%  where $\thread(i)$ is the unique integer satisfying $\tilde{\gamma}_i(1) = \simp^{\bfk}_{\thread(i)}(\Delta^{\bfk}_{\thread(i)}(\frac 12))$. We use it extensively for tracking the orientation of the multisimplex (more precisely the order on simplices in the product) as:
%  \[
% \tilde{\bfsimp}^{\times \bfk} \circ \thread = \sgn(\thread) \tilde{\bfsimp}^{\times \bfk}.
%\]
%holds in $H^\BM_n(\hat{X}_{n,g},\hat{p}^{-1}(Y_{n,g}))$, where $\sgn$ computes the signature of a permutation.
%\end{remark}

%\subsubsection{Structure of the twisted homology}\label{sec_structure_of_twisted_homologies}

Using previous diagrammatic notation, we define the following particular classes in $\HBM_{\bullet} \left( \Conf_n( \varSigma_{g,1} ) , \Conf_n(\varSigma_{g,1})^- ; \varphi \right)$:
\[
\hat{\bfsimp}(\bfa,\bfb):=\pic{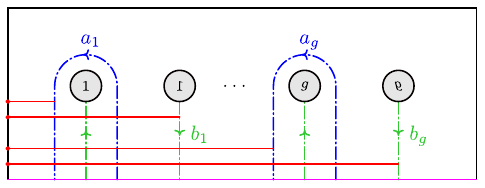}
\]
where $\bfa=(a_1,\ldots,a_g),\bfb=(b_1,\ldots,b_g) \in \N^{\times g}$ are such that $a_1+b_1+\ldots+a_g+b_g=n$. We can now derive the structure of twisted homologies. 

\begin{theorem}[Structure of twisted homologies, {\cite{Big04,JulesVerma,BPS21,JulesMarco}}]\label{thm_structure_twisted_homology}
The Borel--Moore twisted homology modules are free, concentrated in middle dimension. Namely $\HBM_{\bullet} \left( \Conf_n( \varSigma_{g,1} ) , \Conf_n(\varSigma_{g,1})^- ; \varphi \right)$ is free as a $\Z\left[G \right]$-module, and trivial in other dimensions than $n$. In dim. $n$, it admits the following set as basis
 \[
  \left\lbrace \hat{\bfsimp}(\bfa,\bfb) \Biggm| 
  \begin{array}{l}
   \bfa=(a_1,\ldots,a_g),\bfb=(b_1,\ldots,b_g) \in \N^{\times g} \\
   a_1+b_1+\ldots+a_g+b_g=n
  \end{array} \right\rbrace.
 \]
\end{theorem}

Notice that the above theorem works for any such $\varphi$ without further assumption. 

%\subsection{A general action of mapping class groups}

\subsection{The Heisenberg twisted case}\label{sec_Heisenberg_case_on_surface_conf}

\subsubsection{Heisenberg local systems}

 %Note that we have an embedding $\pi_1(\Sigma_{g,1}) \longmapsto B_2(\Sigma_{g,1})$ where $\gamma$ is mapped to $\lbrace \gamma, x_2 \rbrace,$ where $x_2$ is one of the points in our base point of $\mathrm{Conf}_2(\Sigma_{g,1}).$ It is clear that the image of $\pi_1(\Sigma_{g,1})$ in $B_2(\Sigma_{g,1})$ is stable under $\mathrm{Mod}(\Sigma)$ action, hence we have a morphism
 
Now we focus on Example \ref{example_stabilized_subgroups_ofpi1} (3) where the morphism used to twist the homology involves the Heisenberg groups associated with a surface and explicitly defined in Equation \eqref{eq_def_Heisenberg_group} by its presentation. We detail this case and its quotients of particular interest in this paper. Notice that this particular case of local system is also considered in \cite{BPS21,JulesMarco}. 

\begin{notation}[Heisenberg local systems]\label{not_Heisenberg_local_systems}
We fix the following notation for the fundamental group of configuration spaces of positive genus surfaces, for $n,g \in \N$:
\[
\pi_{n,g} := \pi_1\left( \Conf_n(\varSigma_{g,1}) , \lbrace \xi_1 , \ldots , \xi_n \rbrace \right)
\]
where $\varSigma_{g,1}$ stands for the surface of genus $g$ with one boundary component, and $\lbrace \xi_1 , \ldots , \xi_n \rbrace$ is a base point of $\Conf_n(\varSigma_{g,1})$ chosen so that coordinates $\xi_i$'s lie in the boundary of the surface (see the picture below). This group is usually designated as the \emph{braid group with $n$ strands, of the surface of genus $g$}. We refer the reader to \cite{BGG11} for a proof that $\pi_{n,g}$ is generated by the following three types of braids, for $1\le k \le g$ and $1\le k < n$:
\begin{center}
	\def\svgwidth{0.8 \columnwidth}
	%% Creator: Inkscape 1.3 (0e150ed6c4, 2023-07-21), www.inkscape.org
%% Accompanies image file 'generateurs.pdf' (pdf, eps, ps)
\begingroup%
  \makeatletter%
  \providecommand\color[2][]{%
    \errmessage{(Inkscape) Color is used for the text in Inkscape, but the package 'color.sty' is not loaded}%
    \renewcommand\color[2][]{}%
  }%
  \providecommand\transparent[1]{%
    \errmessage{(Inkscape) Transparency is used (non-zero) for the text in Inkscape, but the package 'transparent.sty' is not loaded}%
    \renewcommand\transparent[1]{}%
  }%
  \providecommand\rotatebox[2]{#2}%
  \newcommand*\fsize{\dimexpr\f@size pt\relax}%
  \newcommand*\lineheight[1]{\fontsize{\fsize}{#1\fsize}\selectfont}%
  \ifx\svgwidth\undefined%
    \setlength{\unitlength}{350.21193857bp}%
    \ifx\svgscale\undefined%
      \relax%
    \else%
      \setlength{\unitlength}{\unitlength * \real{\svgscale}}%
    \fi%
  \else%
    \setlength{\unitlength}{\svgwidth}%
  \fi%
  \global\let\svgwidth\undefined%
  \global\let\svgscale\undefined%
  \makeatother%
  \begin{picture}(1,0.37125078)%
    \lineheight{1}%
    \setlength\tabcolsep{0pt}%
    \put(0,0){\includegraphics[width=\unitlength,page=1]{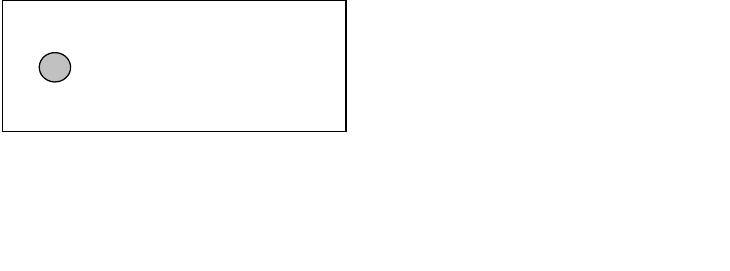}}%
    \put(0.06623689,0.27450726){\color[rgb]{0,0,0}\makebox(0,0)[lt]{\lineheight{1.25}\smash{\begin{tabular}[t]{l}$1$\end{tabular}}}}%
    \put(0,0){\includegraphics[width=\unitlength,page=2]{generateurs.pdf}}%
    \put(0.37221613,0.27555144){\color[rgb]{0,0,0}\makebox(0,0)[lt]{\lineheight{1.25}\smash{\begin{tabular}[t]{l}$g$\end{tabular}}}}%
    \put(0,0){\includegraphics[width=\unitlength,page=3]{generateurs.pdf}}%
    \put(0.2222737,0.27675155){\color[rgb]{0,0,0}\makebox(0,0)[lt]{\lineheight{1.25}\smash{\begin{tabular}[t]{l}$i$\end{tabular}}}}%
    \put(0,0){\includegraphics[width=\unitlength,page=4]{generateurs.pdf}}%
    \put(0.15227406,0.27485882){\color[rgb]{0,0,0}\makebox(0,0)[lt]{\lineheight{1.25}\smash{\begin{tabular}[t]{l}$\ldots$\end{tabular}}}}%
    \put(0.30894681,0.2766759){\color[rgb]{0,0,0}\makebox(0,0)[lt]{\lineheight{1.25}\smash{\begin{tabular}[t]{l}$\ldots$\end{tabular}}}}%
    \put(0,0){\includegraphics[width=\unitlength,page=5]{generateurs.pdf}}%
    \put(0.22524368,0.3118541){\color[rgb]{1,0.11372549,0.11372549}\makebox(0,0)[lt]{\lineheight{1.25}\smash{\begin{tabular}[t]{l}$a_i$\end{tabular}}}}%
    \put(0,0){\includegraphics[width=\unitlength,page=6]{generateurs.pdf}}%
    \put(0.54689437,0.27459612){\color[rgb]{0,0,0}\makebox(0,0)[lt]{\lineheight{1.25}\smash{\begin{tabular}[t]{l}$1$\end{tabular}}}}%
    \put(0,0){\includegraphics[width=\unitlength,page=7]{generateurs.pdf}}%
    \put(0.70973333,0.27664052){\color[rgb]{0,0,0}\makebox(0,0)[lt]{\lineheight{1.25}\smash{\begin{tabular}[t]{l}$i$\end{tabular}}}}%
    \put(0,0){\includegraphics[width=\unitlength,page=8]{generateurs.pdf}}%
    \put(0.63703408,0.27607943){\color[rgb]{0,0,0}\makebox(0,0)[lt]{\lineheight{1.25}\smash{\begin{tabular}[t]{l}$\ldots$\end{tabular}}}}%
    \put(0.79720674,0.2767647){\color[rgb]{0,0,0}\makebox(0,0)[lt]{\lineheight{1.25}\smash{\begin{tabular}[t]{l}$\ldots$\end{tabular}}}}%
    \put(0,0){\includegraphics[width=\unitlength,page=9]{generateurs.pdf}}%
    \put(0.61918309,0.24856632){\color[rgb]{1,0.11372549,0.11372549}\makebox(0,0)[lt]{\lineheight{1.25}\smash{\begin{tabular}[t]{l}$b_i$\end{tabular}}}}%
    \put(0,0){\includegraphics[width=\unitlength,page=10]{generateurs.pdf}}%
    \put(0.387608,0.08598036){\color[rgb]{0,0,0}\makebox(0,0)[lt]{\lineheight{1.25}\smash{\begin{tabular}[t]{l}$\ldots$\end{tabular}}}}%
    \put(0.55078815,0.08694864){\color[rgb]{0,0,0}\makebox(0,0)[lt]{\lineheight{1.25}\smash{\begin{tabular}[t]{l}$\ldots$\end{tabular}}}}%
    \put(0,0){\includegraphics[width=\unitlength,page=11]{generateurs.pdf}}%
    \put(0.25841861,0.05371844){\color[rgb]{1,0.11372549,0.11372549}\makebox(0,0)[lt]{\lineheight{1.25}\smash{\begin{tabular}[t]{l}$\sigma_i$\end{tabular}}}}%
    \put(0.11811039,0.27531036){\color[rgb]{0,0,0}\makebox(0,0)[lt]{\lineheight{1.25}\smash{\begin{tabular}[t]{l}$\reflectbox{1}$\end{tabular}}}}%
    \put(0.28254698,0.27729028){\color[rgb]{0,0,0}\makebox(0,0)[lt]{\lineheight{1.25}\smash{\begin{tabular}[t]{l}$\reflectbox{i}$\end{tabular}}}}%
    \put(0.42077556,0.27493385){\color[rgb]{0,0,0}\makebox(0,0)[lt]{\lineheight{1.25}\smash{\begin{tabular}[t]{l}$\reflectbox{g}$\end{tabular}}}}%
    \put(0,0){\includegraphics[width=\unitlength,page=12]{generateurs.pdf}}%
    \put(0.60181148,0.27594857){\color[rgb]{0,0,0}\makebox(0,0)[lt]{\lineheight{1.25}\smash{\begin{tabular}[t]{l}$\reflectbox{1}$\end{tabular}}}}%
    \put(0,0){\includegraphics[width=\unitlength,page=13]{generateurs.pdf}}%
    \put(0.4484626,0.08657961){\color[rgb]{0,0,0}\makebox(0,0)[lt]{\lineheight{1.25}\smash{\begin{tabular}[t]{l}$i$\end{tabular}}}}%
    \put(0,0){\includegraphics[width=\unitlength,page=14]{generateurs.pdf}}%
    \put(0.50672344,0.08701503){\color[rgb]{0,0,0}\makebox(0,0)[lt]{\lineheight{1.25}\smash{\begin{tabular}[t]{l}$\reflectbox{i}$\end{tabular}}}}%
    \put(0,0){\includegraphics[width=\unitlength,page=15]{generateurs.pdf}}%
    \put(0.30266802,0.08526124){\color[rgb]{0,0,0}\makebox(0,0)[lt]{\lineheight{1.25}\smash{\begin{tabular}[t]{l}$1$\end{tabular}}}}%
    \put(0,0){\includegraphics[width=\unitlength,page=16]{generateurs.pdf}}%
    \put(0.35560464,0.08499892){\color[rgb]{0,0,0}\makebox(0,0)[lt]{\lineheight{1.25}\smash{\begin{tabular}[t]{l}$\reflectbox{1}$\end{tabular}}}}%
    \put(0,0){\includegraphics[width=\unitlength,page=17]{generateurs.pdf}}%
    \put(0.85512621,0.27670869){\color[rgb]{0,0,0}\makebox(0,0)[lt]{\lineheight{1.25}\smash{\begin{tabular}[t]{l}$g$\end{tabular}}}}%
    \put(0,0){\includegraphics[width=\unitlength,page=18]{generateurs.pdf}}%
    \put(0.90368561,0.27609116){\color[rgb]{0,0,0}\makebox(0,0)[lt]{\lineheight{1.25}\smash{\begin{tabular}[t]{l}$\reflectbox{g}$\end{tabular}}}}%
    \put(0.76705276,0.27733431){\color[rgb]{0,0,0}\makebox(0,0)[lt]{\lineheight{1.25}\smash{\begin{tabular}[t]{l}$\reflectbox{i}$\end{tabular}}}}%
    \put(0,0){\includegraphics[width=\unitlength,page=19]{generateurs.pdf}}%
    \put(0.61035754,0.08492){\color[rgb]{0,0,0}\makebox(0,0)[lt]{\lineheight{1.25}\smash{\begin{tabular}[t]{l}$g$\end{tabular}}}}%
    \put(0,0){\includegraphics[width=\unitlength,page=20]{generateurs.pdf}}%
    \put(0.65571597,0.08630296){\color[rgb]{0,0,0}\makebox(0,0)[lt]{\lineheight{1.25}\smash{\begin{tabular}[t]{l}$\reflectbox{g}$\end{tabular}}}}%
    \put(0,0){\includegraphics[width=\unitlength,page=21]{generateurs.pdf}}%
  \end{picture}%
\endgroup%

\end{center}
where red dots mean trivial paths, namely a point staying at its place at all time of the path, and red paths display the movie of the path. We refer the reader to the more detailed Sec.~2.1.1 of \cite{JulesMarco} that use same pictures. 
There exists the following morphism \cite{BGG11,BPS21,JulesMarco}:
\begin{equation}\label{eq_localsystem_of_Heisenberg}
    \varphi^{\mathbb{H}_g}_{n} : \begin{array}{ccc} 
     \pi_{n,g} & \xrightarrow{}  & \mathbb{H}_g = \langle a_1,b_1,\ldots, a_g,b_g ,\sigma | [a_i,a_j],[b_i,b_j], [\sigma,a_i],[\sigma,b_i], [a_i,b_j]\sigma^{-2\delta_{ij}}\rangle \\
    (\alpha_i, \beta_i) & \mapsto & (a_i, b_i), \text{ for } i=1,\ldots, g \\
    \sigma_i & \mapsto & \sigma \text{ for } i= 1 ,\ldots, n-1.
%\\ \alpha_i,\beta_i & \mapsto & a_i, b_i
\end{array}
\end{equation}
For an integer $r \in \N$, we fix notation for the following quotients of the Heisenberg group:
\begin{align}\label{eq_quotients_of_Heisenberg}
\Heis_{g,r} := & \Heisg \big/ \langle \sigma^r \rangle\\
\Heismod_{g,r} := & \Heisg \big/ \langle \sigma^r,a_i^r, b_i^r, i= 1, \ldots ,g  \rangle,
\end{align}
We designate them respectively the {\em Heisenberg group modulo $r$} and the {\em finite Heisenberg group} (modulo $r$). We denote by $\varphi_{n,r}$ and $\overline{\varphi}_{n,r}$ the composition of $\phiHeis_n$ with projection onto resp. $\Heis_{g,r}$ and $\Heismodmodr$. We introduce the following notations for the corresponding relative twisted homology modules of configuration spaces:
\begin{align}
    \calH_n & := \HBM_n\left( \Conf_n(\varSigma_{g,1}), \Conf_n(\varSigma_{g,1})^-; \phiHeis_n\right)\\
    \calH_{n,r} & := \HBM_n\left( \Conf_n(\varSigma_{g,1}), \Conf_n(\varSigma_{g,1})^-; \varphi_{n,r}\right) \\
    \calHmodr_{n,r} & := \HBM_n\left( \Conf_n(\varSigma_{g,1}), \Conf_n(\varSigma_{g,1})^-; \overline{\varphi}_{n,r}\right). 
    \end{align}\label{eq_def_twisted_homology_modules}
%where, for a surface $\varSigma$ with one boundary component.
\end{notation}
In \cite{BGG11,BPS21,JulesMarco} it is shown that the isotopic action of $\Mod(\varSigma)$ on $\Conf_n(\varSigma)$ stabilizes the kernel of the morphism $\phiHeis_n$ (and hence of $\varphi_{n,r}$ and $\overline{\varphi}_{n,r}$) . Namely $\phiHeis_n$ satisfies \eqref{H}.

The case $n=1$ in the above notation deserves a particular treatment as $\pi_{1,g}$ is simply $\pi_1(\varSigma_{g,1}, \xi_1)$ and does not have any $\sigma_i$ as a generator. Nevertheless, note that we have an embedding $\pi_1(\varSigma_{g,1}) \to \pi_{2,g}$ where $\gamma$ is mapped to $\lbrace \gamma, \xi_2 \rbrace$. It is clear that the image of $\pi_1(\varSigma_{g,1})$ is stable under the $\mathrm{Mod}(\varSigma_{g,1})$ action. Alternatively, we note that the Heisenberg group can be defined as the following natural quotient of $\pi_1(\Sigma_{g,1})\times \Z:$ 
\begin{proposition}
We have the following isomorphism:
$$\Heisg \simeq \pi_1(\Sigma_{g,1})\times\langle \sigma \rangle /\langle [\alpha,\beta]=\sigma^{2\omega(\alpha,\beta)}, \ \forall \alpha,\beta \in \pi_1(\Sigma_{g,1})\rangle$$
where $\omega$ is the intersection form on $H_1(\Sigma_{g,1}).$

Moreover, the kernel of the map 
$$\phiHeis_1:\pi_1(\Sigma_{g,1})\longrightarrow \Heis_g$$
is stable under $\Mod(\Sigma_{g,1}).$
\end{proposition}
The group ring $\Z[\Heis_g]$ is hence a quantum torus in the variable $\sigma$ associated with the group $H_1(\Sigma_{g,1})$ endowed with the natural bilinear intersection form. 
\begin{proof}
Let $G=\pi_1(\Sigma_{g,1})\times\langle \sigma \rangle /\langle [\alpha,\beta]=\sigma^{2\omega(\alpha,\beta)}, \ \forall \alpha,\beta \in \pi_1(\Sigma_{g,1})\rangle.$ It is clear from the presentation of $G$ that $[G,G]$ is cyclic, generated by $\sigma.$ If $\alpha_i,\beta_i$ is a system of standard generators of $\pi_1(\Sigma_{g,1}),$ let $a_i=\phiHeis_1(\alpha_i)$ and $b_i=\phiHeis_1(\beta_i).$ It is easy to see that the relations between the $a_i$ and $b_i$ in $\Heisg$ are satisfied, since $\omega(\alpha_i,\beta_i)=1$ and all other generators are represented by disjoint curves. Hence $G$ is a quotient of $\Heisg.$

To see that $G$ is actually isomorphic to $\Heisg,$ we notice that for words $\alpha$ and $\beta$ in the generators $a_i,b_i$ of $\Heisg,$ one has again that $[\alpha,\beta]=\sigma^{2\omega(\alpha,\beta)},$ as can be shown by a simple induction on the lengths of the words $\alpha$ and $\beta.$ Hence the relations in $G$ are already relations in $\Heisg.$

Finally, it is clear that the kernel of $\phiHeis_1$ is stable under the $\Mod(\varSigma)$ action, since $\Mod(\varSigma)$ respects the intersection form $\omega.$
\end{proof}

For the following notation we step back to the following general case:
\[
\varphi : \pi_{n,g} \to G,
\]
where $G$ could be thought of as either $\Heisg$, $\Heis_{g,r}$ or $\Heismodmodr$. We recall the notation: $K_G:= \Ker(\Mod \to \Aut(G)) \subset \Mod(\varSigma_{g,1})$ and: 
\[
\rho^G_n : K_G \to \Aut_{\Z\left[G\right]} (\calH_n)
\]
the non-crossed representation of $K_G$, see Sec.~\ref{sec_uncross}. Recall also: $M_G := \im(\Mod \to \Aut(G)) \subset \Aut(G)$, and:
\[
\rhoext^G_n : \Mod(\varSigma_{g,1}) \to \GL_{N_{g,n}}\left(\Z\left[G \rtimes M_G \right] \right), % (\calH_n)
\]
is the uncrossed representation of $\Mod(\varSigma_{g,1})$ arising from Prop.~\ref{prop_uncross_in_general}. The definition of this extension requires $\calH_n$ (resp. $\calH_{n,r}$ or $\calHmodr_{n,r}$) to be free on $\Z\left[G \right]$, see Theorem~\ref{thm_structure_twisted_homology}, with a given dimension $N_{g,n}:=\binom{2g+n-1}{n}$ also prescribed by the theorem. 

%%\begin{definition}[Heisenberg representations]\label{def_Heisenberg_representations}
%%Let $G$ designate the group $\Heisg$ (or resp. $\Heis_{g,r},\Heismodmodr$), and $\Mod:= \Mod(\varSigma_{g,1})$. Let $K_G:= \Ker(\Mod \to \Aut(G)) \subset \Mod$, then homology modules are endowed with a representation of $K_G$:
%\[
%\rho^G_n : K_G \to \Aut_{\Z\left[G\right]} (\calH_n)
%\]
%(replace $\calH_n$ by resp. $\calH_{n,r}$ or $\calHmodmodr_{n,g}$), see Sec.~\ref{sec_uncross}. Let $M_G := \im(\Mod \to \Aut(G)) \subset \Aut(G)$, then we can extend the homological representations to the entire mapping class group:
%\[
%\rhoext^G_n : \Mod \to \GL_{N_{g,n}}\left(\Z\left[G \rtimes M_G \right] \right), % (\calH_n)
%\]
%arising from Prop.~\ref{prop_uncross_in_general}. The definition this extension requires $\calH_n$ (resp. $\calH_{n,r}$ or $\calHmodmodr_{n,g}$) to be free on $\Z\left[G \right]$, see Theorem~\ref{thm_structure_twisted_homology}, with a given dimension $N_{g,n}:=\binom{2g+n-1}{n}$ also prescribed by the theorem. 
%\end{definition}

\subsubsection{Linearization of $\Heisg$ and related quotients}

In this subsection we study the possibilities for linearizing Heisenberg representations of mapping class groups of surfaces. We begin with a negative result.

\begin{proposition}\label{lem:noEmbedding} 
Let $G$ be a group that contains a non-abelian torsion free nilpotent subgroup. Then there is no faithful algebra embedding $\varphi : \Z[G] \longmapsto M_N(\C)$ for any $N\geq 1.$
\end{proposition}
\begin{proof}
Let $N$ be a torsion free nilpotent subgroup of $G.$ Since $N$ is nilpotent and non-abelian, let $k\geq 2$ so that $\Gamma^k N=1$ and $\Gamma^{k-1}N \neq 1.$ Since $\Gamma^{k-1}N \neq 1,$ there exists $a\in N,$ $b\in \Gamma^{k-2}N$ so that $[a,b]=\sigma \neq 1.$ Note that $\sigma \in \Gamma^{k-1}N,$ so $\sigma$ commutes with $a$ and with $b$.

Choose a representation $\varphi : \Z[G] \longmapsto M_N(\C).$ We have that $\varphi(\sigma)$ commutes with $\varphi(a)$ and $\varphi(b),$ so any eigenspace $V_{\lambda}$ of $\varphi(\sigma)$ is stable by $\varphi(a)$ and $\varphi(b).$ This implies that $\varphi(\sigma)|_{V_{\lambda}}=[\varphi(a)|_{V_{\lambda}},\varphi(b)|_{V_{\lambda}}]$ is a commutator, and therefore $\varphi(\sigma)|_{V_{\lambda}}$ has determinant $1,$ hence $\lambda$ is a root of unity. Hence all eigenvalues of $\varphi(\sigma)$ are roots of unity. Consider $N$ such that $\lambda^N=1$ for any eigenvalue $\lambda$ of $\sigma,$ and $k$ such that each eigenvalue of $\sigma$ has multiplicity at most $k.$ Then the characteristic polynomial $\chi_{\varphi(\sigma)}$ is a divisor of $(X^N-1)^k,$ and by Cayley-Hamilton theorem we have that $(X^N-1)^k$ is an annihilating polynomial for $\varphi(\sigma).$ However $\sigma$ is of infinite order in $G.$ Hence $\varphi$ is not faithful over $\Z[G].$
\end{proof}
%\section{Some representation theory of the Heisenberg group and related groups}

\begin{coro}\label{coro:no_embedding}
Any representation $\Z[\mathbb{H}_g]\to M_d(\C)$ contains $(\sigma^{2N}-1)^k$ in its kernel for some $N,k\geq 1.$
\end{coro}
\begin{proof}
It is a consequence of Lemma \ref{lem:noEmbedding} applied to $G=\mathbb{H}_g$ since the Heisenberg group $\mathbb{H}_g$ is a non-abelian torsion free $2$-nilpotent group. In the proof of Lemma \ref{lem:noEmbedding} take $\sigma$ to be $\sigma^2 \in \Heisg$, it implies the second part of this corollary. Hence there does not exist any faithful representation of $\Z\left[\Heisg \right].$
\end{proof}

\begin{remark}[On faithfulness of Heisenberg representations]\label{rem_on_faithfulness_of_Heis_rep}	
This corollary clarifies the fact that even if one proves the representation of $K_{\Heisg}$ on $\calH_n$ to be faithful, it would not necessarily imply the linearity of mapping class group, since the natural persistence of kernel Prop.~\ref{prop_uncross_in_general}~(ii) is never satisfied. . 

We suggest that an alternative strategy might be to study the representation $\rho_n^{\Heis_{g,r}}$ which has coefficients in $\Z[\Heis_{g,r}],$ for some integer $r>>g.$ We will see in Proposition \ref{prop_faithfulRepHeisModr} that $\Z[\Heis_{g,r}]$ has a faithful representation in $M_d(\C)$ for some integer $d.$
\end{remark}

Next we describe a faithful representation of $\Heisg \rtimes \Aut^+(\Heisg)$. Applying this to the coefficients of the representation $\tilde{\rho}_n^{\Heis_g}$ provides a complex representation of $\mathrm{Mod}(\Sigma_{g,1}).$ We remark that the extension to $\Z[\Heisg \rtimes \Aut^+(\Heisg)]$ of the representation of $\Heisg \rtimes \Aut^+(\Heisg)$ will be unfaithful, due to the previous argument.  

We recall that any element in $\mathbb{H}_g$ has a normal form: any $x\in \mathbb{H}_g$ can be written in a unique way $x=\left(\underset{i=1}{\overset{g}{\prod}}a_i^{m_i}\right)\left(\underset{i=1}{\overset{g}{\prod}}b_i^{n_i}\right)\sigma^l$ where $m_i, n_i, l\in \Z.$ This fact (which is easily derived from the presentation) implies that the group $\mathbb{H}_g$ admits a faithful representation in $\mathrm{GL}_{g+2}(\Q),$ obtained by mapping the element $\left(\underset{i=1}{\overset{g}{\prod}}a_i^{m_i}\right)\left(\underset{i=1}{\overset{g}{\prod}}b_i^{n_i}\right)\sigma^l$ to the matrix $$\begin{pmatrix}
1 & m_1 & \ldots & m_g & \frac{l}{2}
\\ 0 & 1 & \ldots & 0 & n_g
\\ 0 & 0 & \ddots & 0 & \vdots
\\ 0 & & & 1 & n_1
\\ 0 & & & & 1
\end{pmatrix}.$$
This representation is called {\em the tautological representation of $\mathbb{H}_g.$} We recall the representation $\rhoext^{\Heisg}_n$ with coefficients in $\Z[\mathbb{H}_g \rtimes \mathrm{Aut}^+(\mathbb{H}_g)]$ (recalling $M_{\Heisg} \subset \Aut^+(G)$). We would also like to study linear representations of $\mathbb{H}_g \rtimes \mathrm{Aut}^+(\mathbb{H}_g).$ It turns out that $\mathbb{H}_g \rtimes \mathrm{Aut}^+(\mathbb{H}_g)$ is also a linear group. We recall that $\mathrm{Aut}^+(\mathbb{H}_g)=H_1(\Sigma_{g,1},\Z) \rtimes \mathrm{Sp}_{2g}(\Z)$ \cite[Lemma 15]{BPS21}. Let $J$ be the map $\Z^{2g} \longmapsto \Z^{2g}$ defined by $J(X,Y)=(-Y,X)$ for $X,Y \in \Z^g,$ so that $\mathrm{Sp}_{2g}(\Z)$ is the group of invertible integer matrices $M$ that satisfy $M^{-1}=J M^T J^{-1}.$ 
\begin{proposition}[Supra-tautological representation of the Heisenberg group]\label{prop:suprataut}
There is a unique injective morphism $\iota^{\Heisg} : \mathbb{H}_g \rtimes \mathrm{Aut}^+(\mathbb{H}_g) \longmapsto \mathrm{GL}_{2+2g}(\Z)$ which satisfies:
\begin{itemize}
\item[-] For $x=\left(\underset{i=1}{\overset{g}{\prod}}a_i^{m_i}\right)\left(\underset{i=1}{\overset{g}{\prod}}b_i^{n_i}\right)\sigma^l \in \mathbb{H}_g,$ 
$$\iota^{\Heisg}(x)=\begin{pmatrix}
1 & (JX)^T & l
\\ 0 & I_{2g} & X
\\ 0 & 0 & 1\end{pmatrix},$$
where $X=(m_1 \ldots m_g n_1 \ldots n_g).$

\item[-]For $Y \in H_1(\Sigma_{g,1},\Z)\simeq \Z^{2g},$ 
$$\iota^{\Heisg}(Y)=\begin{pmatrix}
1 & 0 & 0 
\\ 0 & I_{2g} & Y 
\\ 0 & 0 & 1
\end{pmatrix}$$
\item[-]For $M\in \mathrm{Sp}_{2g}(\Z),$
$$\iota^{\Heisg}(M)=\begin{pmatrix}
1 & 0 & 0 \\
0 & M & 0 \\
0 & 0 & 1
\end{pmatrix}$$
\end{itemize}
We will call this representation the {\em supra-tautological representation} of $\Heisg \rtimes \Aut^+(\Heisg)$.
\end{proposition}
\begin{proof}
We first check that the restriction of $\iota^{\Heisg}$ to $\Aut^+(\Heisg)\simeq \Z^{2g} \rtimes \mathrm{Sp}_{2g}(\Z)$ is a representation. It is clear that the restriction to $\mathrm{Sp}_{2g}(\Z)$ and to $\Z^{2g}$ are representations, so this follows from the fact that for $M\in $
$\mathrm{Sp}_{2g}(\Z)$ and $Y \in \Z^{2g},$ one has
$$\iota^{\Heisg}(M)\iota^{\Heisg}(Y)\iota^{\Heisg}(M^{-1})=\begin{pmatrix}1 & 0 & 0 \\ 0 & I_{2g} & MY \\ 0 & 0 & 1
\end{pmatrix}=\iota^{\Heisg}((M Y))$$
We also check that the restriction of $\iota^{\Heisg}$ to $\Heisg$ is a representation. Indeed, it is clear that the image of $\sigma$ is central in the image of the representation, and for $X,Y \in \Z^{2g}$ a direct computation shows that
$$\iota^{\Heisg}(X)\iota^{\Heisg}(Y)=\iota^{\Heisg}(Y)\iota^{\Heisg}(X)\iota^{\Heisg}(\sigma)^{2(JX)^TY}$$
which corresponds to the relations in $\Heisg$ since the intersection form is expressed as $\omega(X,Y)=(JX)^T Y.$

Then, to see that $\iota^{\Heisg}$ is a well defined representation of $\Heisg \rtimes \Aut^+(\Heisg)$, one has to check:
\[
\iota^{\Heisg}(f) \iota^{\Heisg}(g) \iota^{\Heisg}(f^{-1}) = \iota^{\Heisg}(f(g))
\]
for $f \in \Aut^+(\Heisg)$ and $g \in \Heisg$.
Since $\mathrm{Sp}_{2g}(\Z)$ and $\Z^{2g}$ generate $\Aut^+(\Heisg)$ (\cite[Lemma~16]{BPS21}), we only need to check it for $f\in \mathrm{Sp}_{2g}(\Z)$ or $f\in \Z^{2g}.$

For $M\in \mathrm{Sp}_{2g}(\Z)$ and $x=\left(\underset{i=1}{\overset{g}{\prod}}a_i^{m_i}\right)\left(\underset{i=1}{\overset{g}{\prod}}b_i^{n_i}\right)\sigma^l \in \Heisg,$ we have that 
$$\iota^{\Heisg}(M)\iota^{\Heisg}(x) \iota^{\Heisg}(M)^{-1}=\begin{pmatrix}1 & (JX)^TM^{-1} & l \\ 0 & I_{2g} & MX \\ 0 & 0 & 1 \end{pmatrix}=\begin{pmatrix}1 & (JMX)^T & l \\ 0 & I_{2g} & MX \\ 0 & 0 & 1 \end{pmatrix}.$$
For $Y \in \Z^{2g}$ and $x$ as above, we have
$$\iota^{\Heisg}(Y)\iota^{\Heisg}(x) \iota^{\Heisg}(Y)^{-1}=\begin{pmatrix}1 & (JX)^T & l-(JX)^T Y \\ 0 & I_{2g} & X \\ 0 & 0 & 1  \end{pmatrix} $$
This matches the action of $\Z^{2g}$ on $\Heisg,$ as $-(JX)^TY=(JY)^TX=\omega(Y,X).$ %(we refer to \cite[Lemma 16]{BPS21} for the description of $\Aut^+(\Heisg)$).

Finally, we check that this is a faithful representation. Let $z \in \Heisg\rtimes \Aut^+(\Heisg)$ which is in the kernel of $\iota^{\Heisg}.$ Looking at the middle diagonal block of $\iota^{\Heisg}(z),$ we get that the automorphism part of $z$ is in $\Z^{2g}.$ Next one can get that $z\in \Heisg$ as the two blocks just above the diagonal are zero, and finally it is clear that $\iota^{\Heisg}$ is faithful on $\Heisg.$

\end{proof}

\begin{remark}
The above supra-tautological representation of $\Heisg \rtimes \Aut^+(\Heisg)$ endows $\calH_n \otimes_{\iota^{\Heisg}} \C^{2+2g}$ with a linear action of $\Mod(\varSigma_{g,1})$. It is a natural extension of Torelli group representations constructed in the first version of \cite{BPS21} to the entire mapping class group. 

We note that while the authors were in the process of writing this paper, a new version of \cite{BPS21} appeared. In \cite[Section 4.2]{BPS21} a very closely related representation was defined, but the authors had discovered this construction independently. 
\end{remark}

Proposition \ref{prop:suprataut} provides us with a linearization of the representations $\rho_n^{\Heis_g}$ that extend to the whole mapping class group, however the kernel does not naturally persist. However, we show that the group rings of the groups $\Heis_{g,r}$ are linear:

\begin{proposition}\label{prop_faithfulRepHeisModr} Let $r\geq 2,$ let $V$ be a $r$-dimensional $\C$-vector space with basis $e_0,\ldots,e_{r-1}.$ 

Let $s_1,\ldots ,s_g,t_1,\ldots ,t_g$ be algebraically independent complex numbers. 

There is a unique faithful morphism of algebras $\iota^{\Heis_{g,r}}: \Z[\Heis_{g,r}] \to \mathrm{End}(V^{\otimes g+1}),$ defined by:
\begin{itemize}
\item[-]$\iota^{\Heis_{g,r}}(\sigma)(e_{j_1}\otimes \ldots \otimes e_{j_g}\otimes e_{j_{g+1}})=e_{j_1}\otimes \ldots \otimes e_{j_g}\otimes e_{j_{g+1}+1}$ where by $j_{g+1}+1$ we mean $j_{g+1}+1 \ \mathrm{mod} \ r.$
\item[-]$\iota^{\Heis_{g,r}}(a_i)(e_{j_1}\otimes \ldots \otimes e_{j_g}\otimes e_{j_{g+1}})=s_i (e_{j_1}\otimes \ldots \otimes e_{j_g}\otimes e_{j_{g+1}+2j_i}).$
\item[-]$\iota^{\Heis_{g,r}}(b_i)(e_{j_1}\otimes \ldots \otimes e_{j_g}\otimes e_{j_{g+1}})=t_i (e_{j_1}\otimes \ldots e_{j_i+1} \otimes \ldots \otimes e_{j_g}\otimes e_{j_{g+1}}).$ 
\end{itemize}
Hence the vector spaces $\calH_{n,r} \otimes_{\iota^{\Heis_{g,r}}} V^{\otimes g+1}$ are endowed with a linear action of $K_{\Heis_{g,r}}$ with the same kernel as the representation $\rho_n^{\Heis_{g,r}}$ of $K_{\Heis_{g,r}}$ upon $\calH_{n,r}$.
\end{proposition}
\begin{proof}

It is easily seen that $\iota^{\Heis_{g,r}}(\sigma^{r})=id_{V^{\otimes g+1}},$ that for any $1\leq i \leq g,$ one has $$\iota^{\Heis_{g,r}}(a_i)\iota^{\Heis_{g,r}}(b_i)=\iota^{\Heis_{g,r}}(\sigma^{2}) \iota^{\Heis_{g,r}}(b_i)\iota^{\Heis_{g,r}}(a_i),$$ and that any other pair of generators of $\Heis_{g,r}$ have commuting images, which shows that $\iota^{\Heis_{g,r}}$ is indeed a representation of $\Heis_{g,r}.$ Hence we focus on showing injectivity.

Assume that $\iota^{\Heis_{g,r}}(x)=0$ for some $x=\underset{h\in \Heis_{g,r}}{\sum}a_h h\in \Z[\Heis_{g,r}].$ We recall that any element $h$ of $\Heis_{g,r}$ can be written in canonical form $h=\left(\prod a_i^{n_i} \right) \left(\prod b_i^{m_i} \right) \sigma^k$ where $n_i,m_i \in \Z$ and $0\leq k <r.$  Let $h'=\left(\prod a_i^{n_i'} \right) \left(\prod b_i^{m_i'} \right) \sigma^{k'}$ be another element of $\Heis_{g,r}.$  Now from the definition of $\iota^{\Heis_{g,r}}$ notice that if for all $1\leq i\leq g,$ one has $k=k',$ $n_i'=n_i \ \mathrm{mod} \ r$ and $m_i'=m_i \ \mathrm{mod} \ r,$ then $\iota^{\Heis_{g,r}}(h)$ and $\iota^{\Heis_{g,r}}(h')$ are colinear. On the other hand, if $h_1,\ldots,h_t$ are elements of $\Heis_{g,r}$ such that the reduction of their canonical coordinates $(n_i \ \mathrm{mod} \ r,m_i \ \mathrm{mod} \ r,k)$ are all different, then $\iota^{\Heis_{g,r}}(h_1),\ldots ,\iota^{\Heis_{g,r}}(h_t)$ are linearily independent over $\C$ in $\mathrm{End}(V^{\otimes g+1}).$ Hence, it suffices to consider the case where $x=x_1 h_1 + \ldots +x_n h_n$ where all the elements $h_1,\ldots,h_n$ have the same canonical coordinates $(n_i,m_i)$ modulo $r$ and same coordinates $k.$ Up to multiplying by an element of $\Heis_{g,r},$ it suffices to consider the case where $x$ is a $\Z$ linear combination of elements of the form $\left(\prod a_i^{r n_i} \right) \left(\prod b_i^{r m_i} \right)$ where $n_i,m_i \in \Z$.% and $0\leq k <r.$

Now, notice that $$\iota^{\Heis_{g,r}}\left(\left(\prod a_i^{r n_i} \right) \left(\prod b_i^{r m_i} \right)  \right)=\left(\prod s_i^{r n_i} \right) \left( \prod t_i^{r m_i} \right) \mathrm{id}_{V^{\otimes g+1}}$$
As the $s_i,t_i$ are algebraically independent, it shows that $\iota^{\Heis_{g,r}}(x)=0$ implies that $x=0.$
\end{proof}

For the above representation of $\Heis_{g,r}$ to extend to $\Heis_{g,r} \rtimes M_{\Heis_{g,r}}$ and hence to obtain linear representations of the entire mapping class group, a solution is studied in \cite{JulesMarco}. The result is a projective representation of the Torelli group (that can be unprojectivized using the adjoint action). In \cite{JulesMarco} it is shown that the obtained representations arise from {\em non semi-simple topological quantum field theories}. In the more particular case of the finite Heisenberg group, the extension is even more natural. However they use the unique irreducible representation of $\Heisg$ in dimension $r^g$ while the above has a greater dimension and is not irreducible. 

\begin{remark}
	For $G$ a finite group, the regular representation is a faithful representation of $\Z[G],$ hence $\Heismod_{g,r}$ and $\Heismod_{g,r}\rtimes \mathrm{Aut}^+(\Heismod_{g,r})$ admit faithful representations of their group rings.
	Alternatively, a faithful representation of $\Heismod_{g,r}$ can be obtained by the previous construction setting $s_i=t_i=1.$ To get a faithful representation of $\Heismod_{g,r}\rtimes \mathrm{Aut}^+(\Heismod_{g,r}),$ one can simply construct the induced representation.
	
	A closely related linearization of $\Heismod_{g,r}$ was considered in \cite{JulesMarco}.
\end{remark}

The conclusion of this section is that if one shows the representation of $K_{\mathbb{H}_g}$ on $\calH_n$ is faithful, a further argument is required to deduce the linearity of $\Mod(\Sigma_{g,1})$, even though the supratautological linearization provides linear representations built out of a faithful linearization of $\mathbb{H}_g$. However, obtaining the faithfulness of the action on $\calH_{n,r}$ would imply linearity of $K_{\Heis_{g,r}}.$ 

\subsection{An extension of $\rho^{\mathbb{H}_g}_1$ to closed surfaces}\label{S:closed_surfaces}
In this section, we will construct a version of the representations $\rho_1$ from previous section which will be a representation of (a subgroup of) the mapping class group of a \emph{closed} surface  $\Sigma_g$ of genus $g.$ The idea is to make use of the Birman exact sequence \cite{Bir,FarbMargalit}:
\begin{equation}\label{E:Birman_exact_sequence}
1 \longrightarrow \pi_1(\Sigma_g,*) \overset{\push}{\longrightarrow}  \mathrm{Mod}(\Sigma_{g},*) \overset{r}{\longrightarrow} \mathrm{Mod}(\Sigma_g) \longrightarrow 1
\end{equation}
Here, $\mathrm{Mod}(\Sigma_{g},*)$ is the mapping class group of $\Sigma_g$ with a base point $*$, namely defined from homeomorphisms fixing $*$.
We recall that the map $r$ is the forgetful map which sends isotopy classes of homeomorphisms fixing the base point to isotopy classes of homeomorphisms. The map $\push$ is called the \emph{point-pushing} morphism and it sends an element $x\in \pi_1(\Sigma_g,*)$ to a mapping class in $\mathrm{Mod}(\Sigma_{g},*)$ for which the induced action on $\pi_1(\Sigma_g)$ is the conjugation by $x.$

We also need to restrict $\rho^{\Heisg}_1$ to the absolute homology. Let $\widehat{\Sigma}_{g,1}$ be the cover of $\Sigma_{g,1}$ associated to the morphism $\varphi_1^{\Heis_g}.$ Let
$$\calH_1'=H_1(\Sigma_{g,1},\varphi_1^{\Heis_g})$$ be the absolute homology (as opposed to $\calH_1$ which is the homology relative to $\partial^- \Sigma_{g,1}$). %It may be identified with the subspace of $\calH_1$ consisting of relative chains $[\gamma]$ with $\partial \gamma=0$ while regarded as a twisted chain in $C_0(\Sigma_{g,1},\partial^- \Sigma_{g,1},\varphi_1^{\Heis_g}).$ 
Notice that the absolute homology is embedded in the relative one since in the long exact sequence of the pair, the term $H_1(\partial^- \Sigma_{g,1})$ is trivial. We note that this subspace $\calH_1'$ is fixed by any element in $\mathrm{Mod}(\Sigma_{g,1}),$ and has a natural structure of $\Z[\Heis_g]$-module.  This defines, by restriction of $\rho_1^{\Heis_g}$ to $\calH_1'$:
\[
\rho'^{\Heis_g}_1 : K_{\Heis_g} \to \mathrm{Aut}_{\Z[\Heis_g]} (\calH_1').
\]

We also recall that the mapping class group of $\Sigma_g$ with a puncture $*$ that is involved in \eqref{E:Birman_exact_sequence} is related to $\Mod(\Sigma_{g,1})$ by the following:
\begin{equation}\label{E:capping}
\mathrm{Mod}(\Sigma_g,*)=\mathrm{Mod}(\Sigma_{g,1})/\langle \tau_{\delta }\rangle,
\end{equation}
where $\tau_{\delta}$ denotes the Dehn twist along a simple closed curve $\delta$ which is parallel to the boundary component of $\Sigma_{g,1}.$ We let $\overline{K}_{\Heisg} \subset \mathrm{Mod}(\Sigma_g,*) $ be the image of $K_{\Heisg} \subset \Mod(\Sigma_{g,1})$ under the quotient \eqref{E:capping}, and $K_{\Heis_g}' := r(\overline{K}_{\Heis_g})$ be the corresponding subgroup of $\Mod(\Sigma_g)$, where $r$ is involved in \eqref{E:Birman_exact_sequence}. Let us define
\begin{equation}\label{E:PAut}
\mathrm{PAut}_{Z[\Heis_g]}(\calH_1'):=\mathrm{Aut}_{Z[\Heis_g]}(\calH_1')/\lbrace \sigma^k \mathrm{id}_{\calH_1'}, k\in \Z \rbrace.
\end{equation}
Notice that since any $[\gamma]\in H_1(\Sigma_{g,1},\Heisg)$ can be isotoped to be disjoint and in the interior of $\delta,$ we have:
$$ \rho^{\Heisg}_1(\tau_{\delta})=\sigma^{-2g} \mathrm{id}_{\calH_1'}$$
as a consequence of Proposition \ref{P:twisted_transvection_formula_1pt} that will be established later on. It shows that the representation $
\rho'^{\Heis_g}_1$ descends to:
\begin{equation}\label{E:rep_for_punctured_mcg}
\rho'^{\Heis_g}_1 : \overline{K}_{\Heis_g} \to \mathrm{PAut}_{Z[\Heis_g]}(\calH_1') .
\end{equation}
Next we show that it also descends to a representation of $K'_{\Heis_g}$. 

Let $\push(x) \in p(\pi_1(\Sigma_g))$ be an element of $\Mod(\Sigma_g,*)$. Notice that it has a homeomorphism representative fixing a small neighborhood of $*$ and hence can be regarded as an element of $\Mod(\Sigma_{g,1})$ still denoted $\push(x)$ by abuse of notation. It is an element of $\Mod(\Sigma_{g,1})$ in the preimage of $\push(x)$ by \eqref{E:capping}. We look at the action of $\push(x)$ on $\calH_1'.$ Let $[\gamma] \in \calH_1',$ where $\gamma$ represents a closed loop in $\widehat{\Sigma}_g$, the regular cover defined from $\varphi_1^{\Heisg}$. By standard Fox calculus, we get:
$$\rho'^{\Heis_g}_1(\push(x))([\gamma])=[x\gamma x^{-1}]=[x]+\varphi_1^{\Heis_g}(x) [\gamma]-\varphi_1^{\Heis_g}(x\gamma x^{-1})[x]= \varphi_1^{\Heis_g}(x) [\gamma]$$
since $\gamma \in \mathrm{Ker}\varphi_1^{\Heis_g}$ and $x\gamma x^{-1} \in \langle \sigma^k, k\in\Z \rangle.$ So $\push(x)$ is sent to $\varphi_1^{\Heisg}(x) \mathrm{id}_{\calH_1'}$ by $\rho'^{\Heis_g}_1$. 

Now let us take $f\in K'_{\Heis_g}.$ Choose a lift of $f$ in $\mathrm{Mod}(\Sigma_g,*).$ Note that two different choices of lift differ by an element of $\push(\pi_1(\Sigma_g))\cap \overline{K}_{\Heis_g}.$ Since the automorphism of $\Heis_g$ induced by $\push(x)$ is the conjugation by $\varphi_1^{\Heis_g}(x)$, we have that $\push(x)\in \overline{K}_{\Heis_g}$ if and only if $\varphi_1^{\Heis_g}(x)\in Z(\Heis_g)=\langle \sigma \rangle.$
Therefore we get the following proposition.

\begin{proposition}
	\label{prop:closedcase} The action of $\mathrm{Mod}(\Sigma_{g,1})$ on $\calH_1'$ induces a representation of $K_{\Heis_g}'\subset \mathrm{Mod}(\Sigma_g):$
	
$$\rho_1'^{\Heis_g}:K_{\Heis_g}' \longrightarrow \mathrm{PAut}_{Z[\Heis_g]}(\calH_1').$$
\end{proposition}

\begin{remark}
	\label{remark:Magnus}
	\begin{enumerate}
		\item We note that the above construction may also be used to construct a version of the Magnus representation for the Torelli group of closed surfaces. Let us write
		$$\calH_1^{ab}:=H_1(\widehat{\Sigma}_{g,1},\Z),$$ where $\widehat{\Sigma}_{g,1}$ is the maximal abelian cover of $\Sigma_{g,1}.$ Cover transformations give $\calH_1^{ab}$ the structure of a $\Z[H_1(\Sigma_{g,1},\Z)]$-module. The (absolute) Magnus representation is the natural representation 
		$$\rho_1^{ab}: \mathcal{I}(\Sigma_{g,1}) \longrightarrow \Aut_{\Z[H_1(\Sigma_{g,1},\Z)]}(\calH_1^{ab}).$$
		Note that the twist along the boundary component is in the kernel of this representation, hence we can think of it as a representation of $\mathcal{I}(\Sigma_g,*)$ instead. Moreover, any element of $\pi_1(\Sigma_g)$ is sent by the point-pushing map to an element of  $\mathcal{I}(\Sigma_g,*).$ Those elements are sent by $\rho_1^{ab}$ to maps of the form $h \mathrm{id}_{\mathcal{H}_1^{ab}},$ where $h\in H_1(\Sigma_g,\Z).$ We write $\mathrm{PAut}_{\Z[H_1(\Sigma_{g,1},\Z)]}(\calH_1^{ab})$ to be the quotient of $\mathrm{PAut}_{\Z[H_1(\Sigma_{g,1},\Z)]}(\calH_1^{ab})$ by such elements. Similarly to the discussion before, we get a representation 
		\[
		\rho_1'^{\ab} : \mathcal{I}(\Sigma_g) \to \mathrm{PAut}_{\Z[H_1(\Sigma_{g,1})]}(H_1(\Sigma_{g,1}; \ab))
		\]
		where $\ab: \pi_1(\Sigma_{g,1}) \to H_1(\Sigma_{g,1})$ is the Magnus local system. The authors do not know if this representation has appeared anywhere else in the litterature.
		\item We note that Suzuki has found elements in the kernel of Magnus representations for $g\geq 2,$ as commutators of Dehn twists along separating curves in $\Sigma_{g,1}$ for which the geometric intersection is non zero but the twisted intersection form vanishes.
		We note that the elements found in \cite{S05b} for $g=2$ would not yield elements in the kernel of our representation $\rho_1'^{ab}:$ indeed, if we fill the boundary component, then the two curves found by Suzuki have now zero geometric intersection. Therefore we ask:
		
		\begin{question}
			Is the representation $\rho_1'^{ab}$ faithful for $g=2 ?$
		\end{question} 
		A positive answer would give an independent proof of the linearity of the mapping class group of a closed surface of genus $2,$ which is known by Bigelow and Budney \cite{BB01}, using Lawrence representations.
		\item Compared to the representation $\rho_1^{\Heis_g}$ the added issue is that the $\Z[\Heis_g]$-module $\calH_1'$ may not be free. For example, it is known that the first homology of the maximal abelian cover of a surface $\Sigma_g$ of genus $g\geq 2$ is not free over $\Z[H_1(\Sigma_g)].$ However, there might still be some stable submodule $M\subset \calH_1',$ such that either $M$ is free or $\calH_1'/M$ is free. This is the case for the homology of the maximal abelian cover, see \cite[Theorem E and F]{Put}. Moreover in general, notice that the Poincaré duality for manifolds with boundary shows:
		\[
		\calH_1' \simeq H^1(\Sigma_{g,1},\partial \Sigma_{g,1}, \varphi_1^{\Heisg}),
		\]
		while finding a $CW$-complex adapted to twisted homologies is doable for the relative to the boundary case. In configuration spaces and Borel--Moore homology, it was done in \cite{Bianchietal} for the non twisted case, and in \cite{Pierre} for the case of punctured disks. This could help finding a basis of the free part of $\calH_1'$. 
	\end{enumerate} 
\end{remark}

\section{Intersection form and faithfulness criterion}\label{S:intersection_form}

\subsection{Twisted intersection form}\label{S:intersection_form_def}

In this section, for $g>0$, we construct an intersection pairing for the twisted homology with coefficients in $G$ for an arbitrary local system on $\Conf_n(\Sigma_{g,1})$:
\[
\varphi: \pi_{n,g} \to G.
\]
Of course $G$ will later be $\Heis_g$ (or $\Heis_{g,r}$ and $\Heismodmodr$). We use the notation $\calH_n$ for the twisted homology constructed in this case, even though we usually use this notation for the precise case $G=\mathbb{H}_g$. This abusive notation is adopted for all this section. We want to define a bilinear (\textit{intersection}) form on spaces $\calH_n$. Poincaré duality for manifolds with boundary does not provide self dualities. We thus define the space:
\[
\dualH_n := \Hnot_n\left( \Conf_n(\varSigma_{g,1}), \Conf^+_n(\varSigma_{g,1}); \varphi \right),
\]
where:
\[
\Conf^+_n(\varSigma_{g,1}):= \lbrace \lbrace z_1 , \ldots , z_n \rbrace  \in \Conf_n(\varSigma_{g,1}) | \exists i, z_i \in \partial^+ \varSigma_{g,1} \rbrace
\]
where $\partial^+ \varSigma_{g,1}$ is the complement of $\partial^- \varSigma_{g,1}$ in $\partial \varSigma_{g,1}$. %We similarly define spaces $\dualHmodr$ and $\dualHmodmodr$ simply replacing $\varphi_n^{\Heis_g}$ by $\varphi_n^{\Heis_{g,r}}$ and $\varphi_n^{\Heismodmodr}$ respectively. 
We highlight the fact that this new notation involves \textit{standard} singular homologies rather than Borel--Moore ones. 

Now for $(c,d) \in \calH_n \times \dualH_n$ we can assume they can be represented by a pair of chains still denoted $c$ and $d$ where $c$ is a chain of the pair $\left(\widehat{\Conf}_n(\varSigma_{g,1}) , \hat{p}^{-1}({\Conf}^-_n(\varSigma_{g,1}) \right)$ and $d$ one of the pair $\left(\widehat{\Conf}_n(\varSigma_{g,1}) , \hat{p}^{-1}({\Conf}^+_n(\varSigma_{g,1}) \right)$ that are in a generic transversal intersection position so that they have a finite number of intersection points (recalling both are middle dimension chains of $\Conf_n(\varSigma_{g,1})$). This is thanks to the fact that $d$ must be compactly supported, $c$ is locally finite for the ambient topology (natural consequence of being a Borel--Moore chain) and that $c$ and $d$ are disjoint in a neighborhood of the boundary of $\partial \Conf_n(\Sigma_{g,1}).$ We define:
\[
\begin{array}{rcl}
\langle \cdot, \cdot \rangle_{\varphi} : \calH_n \times \dualH_n & \to & \Z[G] \\
(c,d) & \mapsto & \sum_{h \in G} \langle c,h\cdot d \rangle h
\end{array}
\]
where:
\[
\begin{array}{rcl}
\langle \cdot, \cdot \rangle : \calH_n \times \dualH_n & \to & \Z
\end{array}
\]
is the usual $\Z$-bilinear algebraic intersection form, existing since an orientation of $\Conf_n(\varSigma_{g,1})$ determines one on any of its regular cover. We omit the dependence in $n$ in the notation when no confusion arises. For the above sum to be finite, one of the two classes involved must be compactly supported which motivates the choice of the standard homology for the dual homology $\dualH_n$.

\begin{proposition}\label{P:properties_twisted_intersection}
	%Let $\bbG_g$ designate either $\Heis_g,\Heis_{g,r}$ or $\Heismodmodr$ for $g>0$. 
	The intersection form $\langle \cdot , \cdot \rangle_{\varphi}$ has the following properties:
	\begin{enumerate}
		\item It is preserved by the action of the group $K_{G}:= \Ker(\Mod(\varSigma_{g,1}) \to \Aut(G))$, i.e. for $f \in K_{G}$:
		\[
		\langle \rho_n^G(f)(c) , \rho_n^G(f)(d) \rangle_{\varphi} = \langle c,d \rangle_{\varphi} .
		\]
		%%%%%%%%%%%%%%%%%%%%%%%%%% Faux et inutile dans leur version vraie. 
		%    \item It is preserved by the (left regular) action of $G$.
		%    \item Moreover it is $\Z[G]$-sesquilinear, namely for $h_1,h_2 \in \Z[G]$, it satisfies:
		%    \[
		%    \langle h_1 \cdot c , h_2 \cdot c \rangle_{\varphi} =   \langle c,d \rangle_{\varphi} \overline{h_2} h_1
		%    \]
		%    where we recall that $\overline{\cdot}$ denotes the group ring conjugation. 
		\item One has:
		\[
		\varepsilon\left( \langle c,d \rangle_{\varphi} \right) = \langle \hat{p}(c) , \hat{p}(d) \rangle
		\]
		where $\varepsilon: \Z[G] \to \Z$ is the augmentation morphism and $\hat{p}$ is the regular cover map associated with $\pi_{n,g} \to G$. 
	\end{enumerate}
\end{proposition}
\begin{proof}
	Let $\hat{f}$ be a homeomorphism of $\Conf_n(\Sigma_{g,1})$ that projects to an element $f$ with isotopy class in $K_{G}$. Then:
	\[
	\langle \rho_n(f)(c) , \rho_n(f)(d) \rangle_{\varphi} = \sum_{h \in G} \langle \hat{f}(c),h\cdot \hat{f}(d) \rangle h = \sum_{h \in G} \langle \hat{f}(c), \hat{f}(h\cdot d) \rangle h = \langle c,d \rangle_{\varphi}
	\]
	the second equality relying on the fact that $\hat{f}$ commutes with that of the deck transformations group $G$ which is by definition of $K_G$, the third one is the algebraic intersection number that is typically preserved by orientation preserving homeomorphisms. It proves the first item. 
	%%%%%%%%% Une preuve fausse
	%For the third item, let $h_1,h_2 \in G$ (and conclude by linearity):
	%\[
	%\langle h_1 \cdot c , h_2 \cdot d \rangle_{\varphi} = \langle  c , h_1^{-1} h_2 \cdot d \rangle_{\varphi} = \sum_{g \in G} \langle c, g h_1^{-1} h_2 \cdot d \rangle g = \sum_{k \in G} \langle c, k \cdot d \rangle k h_2^{-1} h_1 = \langle c,d \rangle_{\varphi}  h_2^{-1} h_1
	%\]

	As for the last item, let us note that intersection points of $\hat{p}(c)$ and $\hat{p}(d)$ lift to $G$-orbits of intersection points between $c$ and  translates of $d$ by cover transformations. Moreover, the augmentation morphism transform the contributions of those intersection points to $\langle c,d\rangle$ (belonging to $\pm G \subset \Z[G]$) into signs, thus recovering the algebraic intersection.
\end{proof}

This pairing could be interpreted in terms of twisted Poincaré duality, and it is non degenerate, which is the purpose of next section.

\subsection{Dual families and non degeneracy}
\label{sec:inter_form_def}
In Sec.~\ref{S:twisted_homologies_structure} we have defined twisted diagrammatic classes in $\calH_n$ from three inputs: a family of arcs called an $m$-multisimplex, a partition of $m$ and a thread. In this section we define another kind of twisted diagrammatic classes lying in $\dualH_n$. Let $\bfk=(k_1,\ldots,k_m)$ be a partition of $n$ and $\bfsimp = (\simp_1,\ldots,\simp_m)$ a family of $m$ disjoint arcs with ends in $\partial^+\varSigma_{g,1}$. Let's assume for commodity that $\Gamma_1$ has an end in $\xi_1$, and every $\Gamma_i$ has an end in $\xi_{k_1+\ldots + k_{i_1} +1}$. we define the $\bfk$-parallelized $\bfsimp$ to be:
\[
\bfsimp^{[\bfk]} := \left(\simp_1(1) , \ldots , \simp_1(k_1) , \simp_2(1) , \ldots , \simp_2 (k_2) , \ldots , \simp_m (1) , \ldots, \simp_m(k_m) \right)
\]
where $\Gamma_i(1):= \Gamma_i$ and for $1\le j\le k_i$, $\Gamma_i(j)$ is a left-parallel of $\Gamma_i(j-1)$ with starting end in $\xi_{k_1 + \ldots + k_{i-1} + j}$. By \textit{left-parallel} we mean that the arc is slightly pushed to the left (according to its orientation) staying relative to $\partial^+ \varSigma_{g,1}$. Here is an example showing a family of $3$ disjoint arcs relative to $\partial^+ \varSigma$ and its $(2,3,1)$-parallelization:
\begin{center}
	\def\svgwidth{0.8 \columnwidth}
	%% Creator: Inkscape 1.3 (0e150ed6c4, 2023-07-21), www.inkscape.org
%% Accompanies image file 'multiarc.pdf' (pdf, eps, ps)
\begingroup%
  \makeatletter%
  \providecommand\color[2][]{%
    \errmessage{(Inkscape) Color is used for the text in Inkscape, but the package 'color.sty' is not loaded}%
    \renewcommand\color[2][]{}%
  }%
  \providecommand\transparent[1]{%
    \errmessage{(Inkscape) Transparency is used (non-zero) for the text in Inkscape, but the package 'transparent.sty' is not loaded}%
    \renewcommand\transparent[1]{}%
  }%
  \providecommand\rotatebox[2]{#2}%
  \newcommand*\fsize{\dimexpr\f@size pt\relax}%
  \newcommand*\lineheight[1]{\fontsize{\fsize}{#1\fsize}\selectfont}%
  \ifx\svgwidth\undefined%
    \setlength{\unitlength}{302.09939071bp}%
    \ifx\svgscale\undefined%
      \relax%
    \else%
      \setlength{\unitlength}{\unitlength * \real{\svgscale}}%
    \fi%
  \else%
    \setlength{\unitlength}{\svgwidth}%
  \fi%
  \global\let\svgwidth\undefined%
  \global\let\svgscale\undefined%
  \makeatother%
  \begin{picture}(1,0.21410173)%
    \lineheight{1}%
    \setlength\tabcolsep{0pt}%
    \put(0,0){\includegraphics[width=\unitlength,page=1]{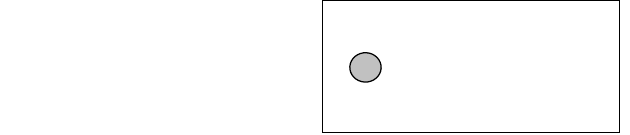}}%
    \put(0.57021862,0.1017713){\color[rgb]{0,0,0}\makebox(0,0)[lt]{\lineheight{1.25}\smash{\begin{tabular}[t]{l}$1$\end{tabular}}}}%
    \put(0,0){\includegraphics[width=\unitlength,page=2]{multiarc.pdf}}%
    \put(0.80803865,0.10066255){\color[rgb]{0,0,0}\makebox(0,0)[lt]{\lineheight{1.25}\smash{\begin{tabular}[t]{l}$g$\end{tabular}}}}%
    \put(0,0){\includegraphics[width=\unitlength,page=3]{multiarc.pdf}}%
    \put(0.88834656,0.10069089){\color[rgb]{0,0,0}\makebox(0,0)[lt]{\lineheight{1.25}\smash{\begin{tabular}[t]{l}$\reflectbox{g}$\end{tabular}}}}%
    \put(0.70521051,0.10217885){\color[rgb]{0,0,0}\makebox(0,0)[lt]{\lineheight{1.25}\smash{\begin{tabular}[t]{l}$\ldots$\end{tabular}}}}%
    \put(0.64751583,0.10038307){\color[rgb]{0,0,0}\makebox(0,0)[lt]{\lineheight{1.25}\smash{\begin{tabular}[t]{l}$\reflectbox{1}$\end{tabular}}}}%
    \put(0,0){\includegraphics[width=\unitlength,page=4]{multiarc.pdf}}%
    \put(0.05817979,0.10205378){\color[rgb]{0,0,0}\makebox(0,0)[lt]{\lineheight{1.25}\smash{\begin{tabular}[t]{l}$1$\end{tabular}}}}%
    \put(0,0){\includegraphics[width=\unitlength,page=5]{multiarc.pdf}}%
    \put(0.29599984,0.10094503){\color[rgb]{0,0,0}\makebox(0,0)[lt]{\lineheight{1.25}\smash{\begin{tabular}[t]{l}$g$\end{tabular}}}}%
    \put(0,0){\includegraphics[width=\unitlength,page=6]{multiarc.pdf}}%
    \put(0.37630775,0.10097338){\color[rgb]{0,0,0}\makebox(0,0)[lt]{\lineheight{1.25}\smash{\begin{tabular}[t]{l}$\reflectbox{g}$\end{tabular}}}}%
    \put(0.19317166,0.10246133){\color[rgb]{0,0,0}\makebox(0,0)[lt]{\lineheight{1.25}\smash{\begin{tabular}[t]{l}$\ldots$\end{tabular}}}}%
    \put(0.13547698,0.10066555){\color[rgb]{0,0,0}\makebox(0,0)[lt]{\lineheight{1.25}\smash{\begin{tabular}[t]{l}$\reflectbox{1}$\end{tabular}}}}%
    \put(0,0){\includegraphics[width=\unitlength,page=7]{multiarc.pdf}}%
  \end{picture}%
\endgroup%

\end{center}

Such $\bfsimp^{[\bfk]}$ is then a family of $n$ disjoint arcs, each of them having an end in a different coordinate of the base point. We define $\hat{\bfsimp}^{[\bfk]}$ to be the unique lift to $\widehat{\Conf}_n(\varSigma_{g,1})$ of the hypercube embedding:
\[
\bfsimp^{[\bfk]} : I^n \to \Conf_n(\varSigma_{g,1}),
\]
that contains $\hat{\underline{\xi}}$. Notice that no thread is used to choose this lift, since $\bfsimp^{[\bfk]}$ has been chosen so to contain the base point $\underline{\xi}$. Since the faces of such a hypercube lies in (the lifted) $\Conf_n(\varSigma_{g,1})^+$, then $\hat{\bfsimp}^{[\bfk]}$ defines a class in $\dualH_n$. The reader should pay attention that it does not need the Borel--Moore homology, but simply the standard singular one.  

\begin{definition} 
	Let $\bfa=(a_1,\ldots,a_g),\bfb=(b_1,\ldots,b_g) \in \N^g$ such that $a_1+\ldots+a_g+b_1+\ldots+b_g=n.$ We define the classes $\tilde{\bfsimp}^\dagger(\bfa,\bfb)$ in $\dualH_n$ as the $(\bfa,\bfb)$-parallelization of the multiarc $A_1\cup \ldots \cup A_g \cup B_1 \cup \ldots B_g$ represented below:
	\begin{center}
		\def\svgwidth{0.4 \columnwidth}
		%% Creator: Inkscape 1.3 (0e150ed6c4, 2023-07-21), www.inkscape.org
%% Accompanies image file 'dualbasis.pdf' (pdf, eps, ps)
\begingroup%
  \makeatletter%
  \providecommand\color[2][]{%
    \errmessage{(Inkscape) Color is used for the text in Inkscape, but the package 'color.sty' is not loaded}%
    \renewcommand\color[2][]{}%
  }%
  \providecommand\transparent[1]{%
    \errmessage{(Inkscape) Transparency is used (non-zero) for the text in Inkscape, but the package 'transparent.sty' is not loaded}%
    \renewcommand\transparent[1]{}%
  }%
  \providecommand\rotatebox[2]{#2}%
  \newcommand*\fsize{\dimexpr\f@size pt\relax}%
  \newcommand*\lineheight[1]{\fontsize{\fsize}{#1\fsize}\selectfont}%
  \ifx\svgwidth\undefined%
    \setlength{\unitlength}{147.41278821bp}%
    \ifx\svgscale\undefined%
      \relax%
    \else%
      \setlength{\unitlength}{\unitlength * \real{\svgscale}}%
    \fi%
  \else%
    \setlength{\unitlength}{\svgwidth}%
  \fi%
  \global\let\svgwidth\undefined%
  \global\let\svgscale\undefined%
  \makeatother%
  \begin{picture}(1,0.43355299)%
    \lineheight{1}%
    \setlength\tabcolsep{0pt}%
    \put(0,0){\includegraphics[width=\unitlength,page=1]{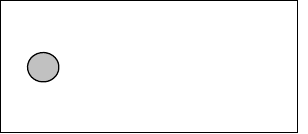}}%
    \put(0.11923035,0.20392831){\color[rgb]{0,0,0}\makebox(0,0)[lt]{\lineheight{1.25}\smash{\begin{tabular}[t]{l}$1$\end{tabular}}}}%
    \put(0,0){\includegraphics[width=\unitlength,page=2]{dualbasis.pdf}}%
    \put(0.60660525,0.20165609){\color[rgb]{0,0,0}\makebox(0,0)[lt]{\lineheight{1.25}\smash{\begin{tabular}[t]{l}$g$\end{tabular}}}}%
    \put(0,0){\includegraphics[width=\unitlength,page=3]{dualbasis.pdf}}%
    \put(0.77118372,0.20171419){\color[rgb]{0,0,0}\makebox(0,0)[lt]{\lineheight{1.25}\smash{\begin{tabular}[t]{l}$\reflectbox{g}$\end{tabular}}}}%
    \put(0.39587503,0.20476352){\color[rgb]{0,0,0}\makebox(0,0)[lt]{\lineheight{1.25}\smash{\begin{tabular}[t]{l}$\ldots$\end{tabular}}}}%
    \put(0.27763883,0.20108335){\color[rgb]{0,0,0}\makebox(0,0)[lt]{\lineheight{1.25}\smash{\begin{tabular}[t]{l}$\reflectbox{1}$\end{tabular}}}}%
    \put(0,0){\includegraphics[width=\unitlength,page=4]{dualbasis.pdf}}%
    \put(0.04982157,0.09838997){\color[rgb]{0,0.0627451,0.97647059}\makebox(0,0)[lt]{\lineheight{1.25}\smash{\begin{tabular}[t]{l}$A_1$\end{tabular}}}}%
    \put(0.59363069,0.10050946){\color[rgb]{0,0.0627451,0.97647059}\makebox(0,0)[lt]{\lineheight{1.25}\smash{\begin{tabular}[t]{l}$A_g$\end{tabular}}}}%
    \put(0.31313313,0.32367824){\color[rgb]{0.18823529,0.78823529,0.22745098}\makebox(0,0)[lt]{\lineheight{1.25}\smash{\begin{tabular}[t]{l}$B_1$\end{tabular}}}}%
    \put(0.81380869,0.30396817){\color[rgb]{0.18823529,0.78823529,0.22745098}\makebox(0,0)[lt]{\lineheight{1.25}\smash{\begin{tabular}[t]{l}$B_g$\end{tabular}}}}%
  \end{picture}%
\endgroup%

	\end{center}
	
\end{definition}

We want to show that elements $\hat{\bfsimp}(\bfa,\bfb)$ and $\hat{\bfsimp}^{\dagger}(\bfa,\bfb)$ are dual for the pairing $\langle \cdot, \cdot \rangle_{\varphi_n^{\Heis_g}}$. But first we describe a natural protocole to compute the pairing between a twisted diagrammatic class and a parallelized family of arcs. 

Let $\bfsimp = (\simp_1 , \ldots , \simp_m)$ be an $m$-multisimplex, $\bfk = (k_1 , \ldots , k_m)$ be a partition of $n$, and let $\hat{\bfsimp}^{(\bfk)}_{\thread}$ be the associated class in $\calH_n$ after the choice of an adapted thread $\thread$. On the other side, let $\bfsimp_+^{[\bfk_+]}$ be the $\bfk_+ = (k_{+,1}, \ldots , k_{+,l})$-parallelized family of adapted arcs $\bfsimp_+ = (\simp_{+,1} , \ldots , \simp_{+,l})$, and $\hat{\bfsimp}_+^{[\bfk_+]}$ the associated class in $\dualH_n$. We explain how to compute $\langle \hat{\bfsimp}^{(\bfk)}_{\thread}, \hat{\bfsimp}_+^{[\bfk_+]} \rangle_{\varphi}$. An \textit{intersecting configuration} is a configuration $\lbrace x_1 , \ldots , x_n \rbrace \in \Conf_n(\varSigma_{g,1})$ that lies in the intersection of diagrams $\bfsimp$ and $\bfsimp_+^{[\bfk_+]}$, with exactly one $x_i$ on each of the $n$ arcs of $\bfsimp_+^{[\bfk_+]}$ and $k_1$ of them on $\simp_1$, $k_2$ on $\simp_2$ and so on until $k_m$ of them on $\simp_m$. It is an intersection point between ${\bfsimp}^{(\bfk)}$ and $\bfsimp_+^{[\bfk_+]}$. Let $P$ be the set of all such intersecting points. Each intersecting configuration $\underline{x} \in P$ contributes to one term of the pairing, since there exists only one point in the fiber of $\underline{x}$ where $\hat{\bfsimp}^{(\bfk)}_{\thread}$ and the pre image by $\hat{p}$ of $\bfsimp_+^{[\bfk_+]}$ intersect. Hence:
\[
\langle \hat{\bfsimp}^{(\bfk)}_{\thread}, \hat{\bfsimp}_+^{[\bfk_+]} \rangle_{\varphi} = \sum_{\underline{x} \in P} \varepsilon_{\underline{x}} h_{\underline{x}},
\]
where for a fixed $\underline{x}$, one has $\varepsilon_{\underline{x}} = \pm 1$ is the sign of the corresponding intersection in fibers, and $h_{\underline{x}} \in G$ is the unique element such that the intersection is between $\hat{\bfsimp}^{(\bfk)}_{\thread}$ and $h_{\underline{x}} \cdot \hat{\bfsimp}_+^{[\bfk_+]}$. Thus 
\[
\varepsilon_{\underline{x}}= \langle \hat{\bfsimp}^{(\bfk)}_{\thread} , h_{\underline{x}} \cdot \hat{\bfsimp}_+^{[\bfk_+]} \rangle.
\]

For a given $\underline{x} \in P$ we define a loop $\delta_{\underline{x}}$ of $\Conf_n(\varSigma_{g,1})$ as the following composition of paths:
\begin{itemize}
	\item First relating $\underline{\xi}$ and $\underline{x}$ following the thread $\thread$ then part of $\bfsimp$ (it is unique up to isotopy fixing ends),
	\item Second relating $\underline{x}$ and $\underline{\xi}$ along $\bfsimp_+^{[\bfk_+]}$.
\end{itemize}
\begin{lemma}\label{L:monomials_in_pairing}
	One has:
	\begin{equation}
	h_{\underline{x}} = \varphi (\delta_{\underline{x}})
	\end{equation}
\end{lemma}
\begin{proof}
	There exists $h_{\underline{x}} \in G$ such that $\hat{\bfsimp}^{(\bfk)}_{\thread}$ and $h_{\underline{x}} \cdot \hat{\bfsimp}_+^{[\bfk_+]}$ intersect in the fiber of $\underline{x}$ by the lifting property of regular cover. Let's lift the path $\delta_{\underline{x}}$ to the cover $\widehat{\Conf}_n(\varSigma_{g,1})$ starting at $\hat{\underline{\xi}}$, it is unique and denoted $\hat{\delta}_{\underline{x}}$. The first part of the path is the lift of the thread so that the path reaches a point in the fiber of $\underline{x}$ lying precisely in $\hat{\bfsimp}^{(\bfk)}_{\thread}$, by definition of the thread. Then the second part of the loop lies in a lift of ${\bfsimp}_+^{[\bfk_+]}$, and more particularly in $h_{\underline{x}} \cdot \hat{\bfsimp}_+^{[\bfk_+]}$ for the $h_{\underline{x}}$ we are seeking, by connectedness. Hence $\hat{\delta}_{\underline{x}}$ ends at $h_{\underline{x}} \cdot \hat{\underline{\xi}}$ which proves the lemma. 
\end{proof}

%{\color{red} LE LEMME SUIVANT DOIT ETRE ECRIT DANS LE CAS $G$néral, et pas dans le cas $G=\mathbb{H}_g$. A MON AVIS CA RESTE UNE CONSEQUENCE IMMEDIATE DE [BPS, APPENDIX B] mais il faut une involution qui envoie les générateurs tresses sur $-1$, pas seulement $\sigma$.}

\begin{proposition}\label{P:dual_families}
	Let $(\bfa=(a_1,\ldots,a_g),\bfb=(b_1,\ldots,b_g))$ and $(\bfa'=(a'_1,\ldots,a'_g),\bfb=(b'_1,\ldots,b'_g))$ be two partitions of $n$. One has:
	\[
	\langle \hat{\bfsimp}(\bfa,\bfb), \hat{\bfsimp}^{\dagger}(\bfa',\bfb') \rangle_{\varphi} = \delta_{(\bfa,\bfb),(\bfa',\bfb')}
	\]
	where the right term is a Kronecker symbol for lists. 
\end{proposition}
\begin{proof}
	The reader would easily check that if partitions $(\bfa,\bfb)$ and $(\bfa',\bfb')$ are different, there does not exist any intersecting configuration and hence the pairing is zero. It remains to compute $\langle \hat{\bfsimp}(\bfa,\bfb), \hat{\bfsimp}^{\dagger}(\bfa,\bfb) \rangle_{\varphi}$ for a given partition $(\bfa,\bfb)$. There exists a unique intersecting configuration $\underline{x}$, for which the loop $\delta_{\underline{x}}$ is trivial so that $h_{\underline{x}} = 1$. All intersections of arcs involved are positive, and $\delta_{\underline{x}}$ never permutes the base point so that:
	\[
	\langle \hat{\bfsimp}(\bfa,\bfb), \hat{\bfsimp}^{\dagger}(\bfa,\bfb) \rangle_{\varphi}  = 1. 
	\]
	which is a direct consequence of the formula provided by \cite[Appendix~B]{BPS21}. 
\end{proof}

In other words, the above proposition means the basis:
\[
\left\lbrace \hat{\bfsimp}(\bfa,\bfb) \Biggm| 
\begin{array}{l}
\bfa=(a_1,\ldots,a_g),\bfb=(b_1,\ldots,b_g) \in \N^{\times g} \\
a_1+b_1+\ldots+a_g+b_g=n
\end{array} \right\rbrace \subset \calH_n 
\]
admits
\[
\left\lbrace \hat{\bfsimp}^{\dagger}(\bfa,\bfb) \Biggm| 
\begin{array}{l}
\bfa=(a_1,\ldots,a_g),\bfb=(b_1,\ldots,b_g) \in \N^{\times g} \\
a_1+b_1+\ldots+a_g+b_g=n
\end{array} \right\rbrace \subset \dualH_n
\]
as a dual family. Thus the following is a straightforward consequence of the fact that the $\Z[G]$-module $\calH_n$ is free with a basis admitting a dual family. 

\begin{coro}\label{C:pairing_non_degenerate}
	The pairing $\langle \cdot, \cdot \rangle_{\varphi}$ is non degenerate on the left. 
\end{coro}

%The reader should be careful as $\dualH_n$ is not a free module in general. The proof of Theorem~\ref{thm_structure_twisted_homology} that states such a free structure relies heavily on the utilisation of the Borel--Moore homology. %The Prop. \ref{P:dual_families} and its Corollary \ref{C:pairing_non_degenerate} are here proved in the case $G=\mathbb{H}_g$ but could be proved for any $G$. Only the utilization of $\eta$ in the proof of Prop. 

Using the fact that the Borel--Moore homology is fully concentrated in its middle dimension, and by application of the universal coefficients spectral sequence, one can show that actually the pairing is perfect, however we won't use it in the present paper. We conclude by giving a formula to compute signs of intersections in the precise case where $G= \mathbb{H}_g$. Hence, we replace $\varphi$ by:
\[
\varphi_n^{\Heis_g} : \pi_{n,g} \to \mathbb{H}_g 
\]
and the signs $\varepsilon_{\underline{x}}$ appearing in the following lemma are only involved in the pairing $\langle \cdot , \cdot \rangle_{\varphi_n^{\mathbb{H}_g}}$.

\begin{lemma}
	One has:
	\[
	\varepsilon_{\underline{x}} = \varepsilon_{x_1} \cdots \varepsilon_{x_n} \eta(h_{\underline{x}})
	\]
	where $\varepsilon_{x_i}$ is the sign of intersection in $\varSigma_{g,1}$ between the arcs supporting $\bfsimp$ and $\bfsimp_+^{[\bfk_+]}$ over which $x_i$ is lying, and:
	\[
	\begin{array}{rcl}
	\eta : \Heis_g & \to & \lbrace \pm 1 \rbrace \\
	g \in \Heis_g & \mapsto & 1 \text{ if } g \text{ is a generator that is not } \sigma \\
	\sigma & \mapsto & -1
	\end{array} 
	\]
\end{lemma}
\begin{proof}
	This is an immediate generalization of the formula given in \cite[Appendix B]{BPS21}
\end{proof}

Finally this last lemma and Lemma \ref{L:monomials_in_pairing} imply:
\begin{equation}\label{E:formula_for_the_pairing}
\langle \hat{\bfsimp}^{(\bfk)}_{\thread}, \hat{\bfsimp}_+^{[\bfk_+]} \rangle_{\varphi_n^{\Heis_g}} = \sum_{\underline{x} \in P} \varepsilon_{x_1} \cdots \varepsilon_{x_n} \varphi_n^{\Heis_g}(\delta_{\underline{x}})_{|\sigma = -\sigma}, % h_{\underline{x}},
\end{equation}

where $\varphi_n^{\Heis_g}(\delta_{\underline{x}})_{|\sigma = -\sigma}$ is the element $\varphi^{\Heis_g}(\delta_{\underline{x}}) \in \Heis_g$ to which we apply the involution of $\Z[\Heis_g]$ that fixes all generators of $\Heis_g$ except $\sigma$ that is sent to $-\sigma$. %This formula is a protocole to compute the pairing in general, and in particular it allows to prove the following. 

\subsection{Alternative construction of the representations}\label{S:alternative_construction}
In this paragraph, we will describe another way of defining the homological representations. It has some similarities with the way quantum representations of mapping class groups are constructed. Indeed, we will rebuild the representations from the data of the twisted intersection form defined in the previous section, using the spirit of the \textit{universal construction} similar to the way Witten-Reshetikhin-Turaev representations are defined in \cite{BHMV}. This can be used to show that some homological representations are isomorphic by showing that the involved pairing are isomorphic, which is typically easier to compute (see Theorem \ref{thm:pure_braid_group}, Theorem \ref{thm:pure_braid_group2} and Theorem \ref{thm:pure_braid_group3} for examples). 

Let us fix a group $G$ and a morphism $\varphi^G: \pi_{n,g} \longrightarrow G$ such that the kernel is preserved by $\Mod(\Sigma_{g,1})$-action, and as in Section \ref{sec_uncross}, let $K_G$ be the kernel of the map $\Mod(\Sigma_{g,1})\longrightarrow \Aut(G).$

Let $V_n$ be the free $\Z[G]$-module spanned by $n$-diagrammatic twisted classes defined in Def.~\ref{def:diagrammatic_twisted_class}. Let $V_n^{\dagger}$ be the free $\Z[G]$-module spanned by dual $n$-diagrammatic classes : they are a disjoint unions of $n\ge 1$ simple arcs on $\varSigma_{g,1}$ with boundary on $\partial^+ \varSigma_{g,1}$. We assume each arc of such a dual twisted class has an end in one of the $\xi_i$'s. We can define a $\Z[G]$-linear form $\langle \cdot , \cdot \rangle_G$ on $V\times V^{\dagger}$ by the formula from Equation~\eqref{E:formula_for_the_pairing}. That formula can be directly adapted to the present case by similarly defining the paths and signs involved. In \eqref{E:formula_for_the_pairing}, it is given for a particular twisted class and a particular dual one but the definition persists. 

%, for $\gamma$ a colored $n$-multiarc and $\delta$ a dual $n$-multiarc:
\begin{equation} \langle \Gamma,\Gamma^{\dagger} \rangle= \underset{\underline{x} \in I(\Gamma,\Gamma^{\dagger})}{\sum}  \varepsilon_{\underline{x}} \varphi^{G}(\delta_{\underline{x}}) \label{eq:twisted_intersection} %\varepsilon_{x_1} \cdots \varepsilon_{x_n}
\end{equation}
where $I(\gamma,\delta)$ stands for the $n$-uple of intersection points of $\Gamma$ and $\Gamma^{\dagger}$.% such that $n_i$ points are on the $i$-th component

The definition of $\langle \cdot , \cdot \rangle_G$ on general elements of $V_n$ and $V_n^{\dagger}$ follows by multilinearity. Notice $\Mod(\Sigma_{g,1})$ has a natural action on isotopy classes of colored (resp. dual) $n$-twisted classes, sending an arc to its image by a representative homeomorphism, as the boundary is pointwise fixed by such a homeomorphism we stay in the same class of arcs.

\begin{proposition}\label{P:Universal_construction_of_homol_reps}
	%We have the following:
	%\begin{itemize}
	%\item[(i)] For $\gamma$ and $\delta$ colored $n$-multiarc and dual $n$-multiarc, $\langle \gamma,\delta \rangle$ depends only on the isotopy classes of $\gamma$ and $\delta.$
	%\item[(ii)] The form $\langle ,\rangle$ is equivariant under $\Mod(\Sigma_{g,1})$-action, and thus invariant under $K_G$-action.
	%\item[(iii)] 
	The quotient $H_{n,g}=V_n /\Ker(\langle \cdot, \cdot \rangle_G)$ of $V_n$ by the left kernel of $\langle \cdot, \cdot \rangle_G$ is a free $\Z[G]$-module isomorphic to $\calH_n^G,$ and the natural action of $\Mod(\Sigma_{g,1})$ on $H_{n,g}$ induces a representation $ K_G \longrightarrow \Aut_{\Z[G]}(H_{n,g})$ which is isomorphic to $\rho_n^G.$
	%\end{itemize}
\end{proposition}
\begin{proof}
	%(i) 
	The formula $\langle \cdot , \cdot \rangle_G$ actually computes the twisted intersection of the homology classes naturally associated with an $n$-diagrammatic twisted class on the one hand and with a dual one on the other. The dual one consists in assigning an embedding of a hypercube in $(\Conf_n(\varSigma_{g,1}) , \Conf_n^+(\varSigma_{g,1}) )$ with the product of arcs, and choosing the only lift of it containing $\hat{\underline{\xi}}$.
	
	%$\hat{\gamma}^{[n]}\in \mathcal{H}_{n,g} $ and $\hat{\gamma}^{\dagger}\in \mathcal{H}_{n,g}^{\dagger},$ hence 
	The pairing is hence invariant under ambient isotopies of diagrams. Note that one could also check this from a direct computation from Equation \ref{eq:twisted_intersection}, showing that removing a bigon between the two $n$-classes does not change the form $\langle \cdot , \cdot \rangle_G.$ (We will not attempt to do this computation). A simple restatement of Proposition~\ref{P:properties_twisted_intersection}-(1) is that the current $\langle \cdot , \cdot \rangle_G$ is equivariant under the action of $\Mod(\varSigma_{g,1})$, and again alternatively, one can see it directly from Equation \ref{eq:twisted_intersection}. Since the pairing $\langle \cdot, \cdot \rangle_G$ is non degenerate on the left (Coro.~\ref{C:pairing_non_degenerate}), and since the homology classes of $n$-twisted classes span $\mathcal{H}_n$ (Theorem~\ref{thm_structure_twisted_homology}), the map that sends a twisted $n$-class to the corresponding homology class in $\mathcal{H}_n$ induces an isomorphism $V_n/\Ker \langle \cdot,\cdot \rangle \simeq \mathcal{H}_n.$ The fact that the $\Mod(\Sigma_{g,1})$-action on diagrammatic twisted $n$-classes induces a representation $K_G \longrightarrow \Aut_{\Z[G]}(\calH_n)$ follows from the fact that $\langle \cdot , \cdot\rangle_G$ is $K_G$-invariant, and this representation is isomorphic to $\rho_n^G$ since the matrix coefficients of $\rho_n^G(f)$ for any $f\in \mathrm{Mod}(\Sigma_{g,1})$ are given by the intersections of the classes $f_*(\mathbf{\hat{\Gamma}(a,b)})$ and $\mathbf{\Gamma^{\dagger}(a',b')}.$  
\end{proof}

\subsection{Intersection of arcs in a surface and a faithfulness criterion}
\label{sec:criterion}
%\subsubsection{Heisenberg intersection of arcs in a surface}

Let $\Gamma_-$ and $\Gamma_+$ be two oriented arcs of the surface $\varSigma_{g,1}$, the first one being relative to $\partial^- \varSigma_{g,1}$ the second one relative to $\partial^+ \varSigma_{g,1}$. Namely they are proper embeddings of $I$ into $\varSigma_{g,1}$ with ends in $\partial^- \varSigma_{g,1}$ resp. $\partial^+ \varSigma_{g,1}$. We furthermore assume $\Gamma_+$ has its starting end in $\xi_1$ the first coordinate of the base point standing in $\partial^+ \varSigma_{g,1}$. 

For $n\in \N$ , we are going to assign them $\Gamma_-^{(n)}$ and $\Gamma_+^{[n]}$ respectively which are classes in $\calH_n$ and $\dualH_n$ respectively. We define the \textit{$G$-twisted intersection pairing} between two such arcs as follows:
\[
\langle \Gamma_-, \Gamma_+ \rangle_{G,n} := \langle  \Gamma_-^{(n)} , \Gamma_+^{[n]} \rangle_{\varphi}.
\]
We define the classes of interest:
\begin{itemize}
	\item The class $\Gamma_-^{(n)}$ is the diagrammatic twisted class (Def.~\ref{def:diagrammatic_twisted_class}) made of the $1$-multisimplex $\Gamma_-$ labeled by $n$, to which we add a thread $\tilde{x}_{\Gamma_-}$ that is given by the unique $n$-multipath going from $\lbrace \xi_1 , \ldots , \xi_n \rbrace$ to the arc $\Gamma_-$ following the boundary counterclockwise.
	\item The class $\Gamma_+^{[n]}$ is the $n$-parallelized $\Gamma_+$. 
	% we need slightly different protocole. We define $\Gamma_+(1):= \Gamma_+$ and for $1\le i\le n$, $\Gamma_+(i)$ is a left-parallel of $\Gamma_+(i-1)$ with starting end in $\xi_i$. By \textit{left-parallel} we mean that the arc is slightly pushed to the left staying relative to $\partial^+ \varSigma_{g,1}$. Here is an example:
	%\begin{center}
	%{\color{red} TO DRAW: A HYPERCUBE, CONVINCING THAT IT IS WELL DEFINED}
	%\end{center}
	%The family of embedded intervals given by $(\Gamma_+(1) , \ldots , \Gamma_+(n))$ naturally defines an embedded hypercube:
	%\[
	%I^n \to (\Conf_n(\varSigma_{g,1}) , \Conf_n^+(\varSigma_{g,1}) ).
	%\]
	%The class $\Gamma_+^{[n]}$ is that of its unique lift to the pair $\left(\widehat{\Conf}_n(\varSigma_{g,1}) , \hat{p}^{-1}\left( \Conf_n^+(\varSigma_{g,1}) \right) \right)$, namely to the regular cover, that contains $\hat{\boldsymbol{\xi}}$ our chosen lift of $\boldsymbol{\xi}$.
\end{itemize}

\begin{definition}\label{D:definite_intersection}
	For $n \in \N$, we say that the pairing of arcs $\langle \cdot , \cdot \rangle_{G,n}$ is \textit{definite} if we have:
	\[
	\forall \Gamma_-,\Gamma_+, \ \langle \Gamma_- , \Gamma_+ \rangle_{G,n} = 0 \iff \Gamma_- \pitchfork \Gamma_+ = 0
	\]
	where $\Gamma_-,\Gamma_+$ are relative arcs as described at the beginning of this section (for which the pairing is well defined), and $\Gamma_- \pitchfork \Gamma_+ = 0$ means that both arcs have a null geometric intersection, namely they can be isotoped off one of each other (by an isotopy fixing ends). 
	
\end{definition}
Roughly speaking the $G$-twisted intersection pairing is definite if and only if it detects the geometric intersection between arcs.
In the above definition the converse implication is easy and true for any intersection pairing. If the intersection pairing turns out to be definite, then Theorem \ref{T:definite_implies_faithful} below will show that the corresponding homological representation of the mapping class group of the surface is faithful. First we need the following lemma: 

%\subsubsection{Definite intersection implies faithful group actions}

\begin{lemma}\label{lem_f_not_trivial_makes_arc_intersect}
	Let $\Sigma$ be a compact connected oriented surface of genus $g\geq 2.$
	\begin{itemize}
		\item Let $f\in \Mod(\Sigma)$ be a non central mapping class on $\Sigma$. Then there exists a simple closed curve $c$ such that $c \pitchfork f(c) \neq 0$. Moreover $c$ can be chosen non-separating or separating of a given genus. 
		\item If $\Sigma=\Sigma_{g,1},$ and $f \in \Mod(\Sigma)$ is not trivial, then there exists an arc (rel. to boundary) $\Gamma$ such that $\Gamma \pitchfork f(\Gamma) \neq 0$. Again, $\Gamma$ can be chosen non separating, or separating of any given genus.
	\end{itemize}
\end{lemma}
\begin{proof}
	The proof of the first point relies on curve complex machinery. Let $\mathcal{C}(\Sigma)$ be the curve graph of $\Sigma,$ (see \cite{Harvey} for a definition) and let $d$ be the distance in the curve graph. Note that any simple closed curve $c$ is at distance at most $1$ from a non separating (resp. separating of genus $k$) one $c'.$ If $d(c,f(c))\geq 4,$ then $d(c',f(c'))\geq 2$ by triangular inequality and the fact that $f$ induces an isometry of $\mathcal{C}(\Sigma).$ Hence, it is enough to show that for $f$ non central, then $d(c,f(c))\geq 4$ for some simple closed curve $c.$ However, $\sup \lbrace d(c,f(c)) \ | \ c\in \mathcal{C}(\Sigma) \rbrace=+\infty$ for any non central $f\in \mathrm{Mod}(\Sigma),$ see \cite[Proof of Corollary 1.1]{RafiSchleimer}.
	
	For the second point, identify $\Sigma_{g,0}$ with the quotient space $\Sigma_{g,1}/\partial \Sigma_{g,1}.$ If $f$ is not in the group generated by the Dehn twist along the boundary (and the elliptic involution if $g=2$) then $f\in \Mod (\Sigma_{g,1})$ induces a non central element of $\Mod (\Sigma_g).$ By the previous argument, there is a non separating curve $c$ in $\Sigma_g$ (resp. separating curve of genus $k$, where $1\leq k\leq g-1$ is fixed), such that the geometric intersection number $i(c,f(c))$ is larger than $2.$ (indeed, it can be picked arbitrarily large). Let $\gamma$ be an arc in $\Sigma_{g,1}$ that maps to $c$ under contraction of $\partial \Sigma_{g,1}.$ Then $i(\gamma,f(\gamma))\geq i(c,f(c))-1 >0$ (where the $-1$ accounts for the new point of intersection that might be created by contracting $\partial \Sigma_{g,1}.$)
	
	Finally, if $f$ is in the group generated by the Dehn twist along the boundary (and the elliptic involution if $g=2$), it is easy to produce arcs that satisfy the condition.
\end{proof}

\begin{theorem}\label{T:definite_implies_faithful}
	If $\langle \cdot , \cdot \rangle_{G,n}$ is definite, then the action of $K_{G} \subset \Mod(\varSigma_{g,1})$ on $\calH_n$ is faithful and its extension to an action of $\Mod(\varSigma_{g,1})$ on $\calH_n \otimes_{\Z[G]} \Z[G \rtimes M_{G}]$ is too. 
\end{theorem}
\begin{proof}
	Let $f \in K_{G}$ be non trivial. Then there exists $\Gamma_+$ an arc relative to $\partial^+ \varSigma_{g,1}$ such that $\Gamma_+ \pitchfork f(\Gamma_+) \neq 0$, from Lemma~\ref{lem_f_not_trivial_makes_arc_intersect}. One can slide $\Gamma_+$ slightly along a small regular tubular neighborhood of itself in such a way that in the end we get an arc $\Gamma_-$ that is a \textit{parallel} of $\Gamma_+$ having the same intersection points with $f(\Gamma_+)$ and such that $\Gamma_- \pitchfork \Gamma_+ = 0$. Suppose $f$ acts trivially on $\calH_n$, then one has:
	\begin{align*}
	0= \langle \Gamma_- , \Gamma_+ \rangle_{G,n} & = \langle f(\Gamma_-) , f(\Gamma_+) \rangle_{G,n} \\
	& = \langle \Gamma_- , f(\Gamma_+) \rangle_{G,n},
	\end{align*}
	the first equality is Prop.~\ref{P:properties_twisted_intersection}, the second line is the triviality of the action of $f$ on $\calH_n$. Then $$\langle \Gamma_- , f(\Gamma_+) \rangle_{G,n} \neq 0$$
	if the pairing is definite. It is a contradiction. 
	
	The second part of the theorem is the persistence of the kernel established in Prop.~\ref{prop_uncross_in_general}.
\end{proof}

The next two sections are devoted to finding elements in the kernel of the pairing, namely pairs of intersecting arcs with $G$-twisted pairing zero when the surface is $\Sigma_{g,1}$.

%\subsection{Some properties in the Heisenberg case}

\section{The single-point case}\label{S:single_point}

In this section, we denote by $\rho_1$ the representation arising from one point of configuration and the local system determined by $\varphi_1^{\Heis_g}:\pi_1(\Sigma_{g,1}) \longrightarrow \Heis_g.$ Similarly, we denote by $\rho_{1,k}$ the representation coming from the local system $\varphi_1^{\Heis_{g,k}}:\pi_1(\Sigma_{g,1})\longrightarrow \Heis_{g,k}=\Heis_g/\langle \sigma^{2k}\rangle.$ We recall that:
\[
\begin{array}{rcl}
\rho_1 : K_{\Heis_g} & \to & \Aut_{\Z[\Heis_g]}(\calH_1),
\end{array}
\]

where $\calH_1$ can be defined by the standard homology rather than Borel--Moore, since in the case of one point, it is the homology of the surface which is compact so both homologies coincide. 

The following proposition is important since many computations rely on it.

\begin{proposition}\label{prop:HeisenbergSimpleCurve} Let $\gamma \in \pi_1(\Sigma_{g,1})$ representing a separating simple closed curve of genus $k$ on the surface $\Sigma_{g,1}.$ Set $\varepsilon=1$ if the orientation of $\gamma$ is compatible with the orientation of $\partial \Sigma_{g,1}$ and $\varepsilon=-1$ otherwise. Then $\varphi_{\mathbb{H}_g}(\gamma)=\sigma^{2\varepsilon k}$.
\end{proposition}
In the above, the orientation of $\gamma$ and $\partial \Sigma_{g,1}$ are compatible if they are both induced by the same orientation on the connected component of $\Sigma_{g,1}\setminus \gamma$ that contains them both.
\begin{proof}
	Recall that two separating simple closed curves of a given genus on $\Sigma_{g,1}$ are in the same orbit under mapping class group action; moreover two oriented such curves are in the same mapping class group orbit if and only if the orientation matches.
	
	Moreover, for $f\in \mathrm{Mod}(\Sigma_{g,1})$ we have $\varphi_1^{\Heisg}(f(\gamma))=f_*(\varphi_1^{\Heisg}(\gamma))$ where $f_*$ is the induced automorphism of $\Heisg.$ Note that $\Aut^+(\Heisg)$ is the identity on $\langle \sigma \rangle:$ therefore to prove the formula it is sufficient to verify it for a single element in each mapping class group orbit of oriented separating simple closed curve.
	
	We conclude noting that for each $k,$ and for each $\varepsilon \in \lbrace \pm 1 \rbrace,$ the element $\left([\alpha_1^{-1},\beta_1^{-1}]\ldots[\alpha_k^{-1},\beta_k^{-1}]\right)^{\varepsilon}$ represents a separating simple closed curve of genus $k.$ Moreover, its orientation is compatible with that of $\partial \Sigma_{g,1}$ if and only if $\varepsilon=1.$ Then we compute
	$$\varphi_{\Heisg}\left([\alpha_1^{-1},\beta_1^{-1}]\ldots[\alpha_k^{-1},\beta_k^{-1}]\right)=[a_1^{-1},b_1^{-1}]\ldots[a_k^{-1},b_k^{-1}]=\sigma^{2k},$$
	by the presentation of $\Heisg$ given in \ref{eq_localsystem_of_Heisenberg}.
 
\end{proof}

\subsection{Action of Dehn twists: twisted-transvection formula and consequences}\label{S:twisted_transvections}

We first give a formula for an action of a Dehn twist along a separating close curve on the homology class of a simple relative arc, similar to Suzuki's transvection formula for the Magnus representation of the Torelli group  \cite{S05a,S05b}. Let $\alpha$ be a separating simple closed curve, we denote by $\Sigma^{int}(\alpha)$ the surface with one boundary component separated by $\alpha$ and the genus of $\Sigma^{int}(\alpha)$ by $g(\alpha).$ We say then that $\alpha$ is separating of genus $g(\alpha).$

Let us consider $\alpha$ a separating simple closed curve of genus $k,$ and $\beta$ an arc in $\Sigma$ with endpoints in $\partial^- \Sigma_{g,1}.$ The arc $\beta$ gives rise to a homology class $[\beta] \in \calH_1.$ Note that up to isotopy, we can assume that all intersection points of $\beta$ and $\alpha$ happen at a single point $p\in \Sigma_{g,1}.$ We call $\hat{p}$ the lift of $p$ to $\hat{\Sigma}_{g,1}$ corresponding to following the lift starting from basepoint of the first subarc of $\beta \setminus \lbrace p \rbrace.$

We can decompose $\beta$ into a collection of arcs, lying in the internal or external connected component of $\Sigma_{g,1}\setminus \alpha$ and with endpoints on $\partial^{-} \Sigma_{g,1}$ or on $\alpha$. Let $[\beta]=[\beta]_{ext}+[\beta]_{int}$ be the associated decomposition in $C_1(\Sigma_{g,1},\partial \Sigma_{g,1},\varphi_1^{\Heis_g})$. 

We can also make the simple closed curve $\alpha$ into an arc with endpoints in $\partial^+ \Sigma$ by opening $\alpha$ along an arc left parallel to the first subarc of $\beta \setminus \lbrace p \rbrace.$ We call the resulting arc $\alpha'.$  Finally, let $[\alpha] \in C_1(\Sigma_{g,1},\partial\Sigma_{g,1},\varphi_1^{\Heisg})$ be the chain corresponding to the lift of $\alpha$ starting at $\hat{p}.$

\begin{proposition}\label{P:twisted_transvection_formula_1pt}
With the notations above, one has:
$$\rho_1(\tau_{\alpha})([\beta])=[\beta]_{ext} +\sigma^{-2k}[\beta]_{int} + \langle \alpha',\beta \rangle_{\varphi_1^{\Heis_1}} [\alpha].$$
\end{proposition}
\begin{proof}
 Note that while $[\beta]$ is a cycle, $[\beta]_{ext}$ and $[\beta]_{int}$ are just chains; however one has $\partial [\beta]_{ext}=\langle \alpha',\beta \rangle \hat{p}.$ Indeed the boundary of the first arc of $[\beta]_{ext}$ contributes $+\hat{p},$ and subsequent arcs contribute exactly the monomials involved in formula \eqref{E:formula_for_the_pairing}. Notice that in latter formula, $\varepsilon_{x_i}$ are here the usual sign of arc intersections while the involution $\sigma \to -\sigma$ does not change anything in the single point case since only squared $\sigma$ are involved. We also note that $[\alpha]$ is just a chain, but one has $\partial [\alpha]=(\sigma^{-2k}-1)\hat{p},$ since $\alpha$ is mapped to $\sigma^{-2k}$ by $\varphi_1^{\Heis_1}$ (Prop. \ref{prop:HeisenbergSimpleCurve}). 
Note that the preimage of the simple closed curve $\alpha$ in $\hat{\Sigma}_{g,1}$ is homeomorphic to $\R,$ and obtained by concatenating different lifts of $\alpha$ (hence it is made of concatenations of translated $[\alpha]$). Moreover, a diffeomorphism representing the isotopy class of $\tau_{\alpha}$ is obtained as the identity on $\Sigma_{g,1}$ minus an annulus neighborhood of $\alpha,$ and on that annulus homeomorphic to $S^1 \times [0,1]$ it acts by $t$-translation on $S^1 \times \lbrace t \rbrace.$ The lift of the latter is the identity on the preimage of $\Sigma^{ext}(\alpha),$ it translates by the cover transformation $\sigma^{-2k}$ on the preimage of $\Sigma^{int}(\alpha)$. It remains to study its action on the preimage of the annulus (homeomorphic to $\R \times [0,1]$): it acts by $t$-translation on $\R\times \lbrace t \rbrace.$ Combining this description with the formula for $\partial [\beta]_{ext},$ we get exactly the formula of Proposition \ref{P:twisted_transvection_formula_1pt}.
\end{proof}

\begin{coro}\label{C:transvection_mod2k_formula} Let $\alpha$ be an oriented separating simple closed curve on $\Sigma_{g,1}$ of genus $k.$ Let $\tau_{\alpha}$ be the associated Dehn twist. Then for any $[\beta] \in H_1(\Sigma_{g,1},\partial\Sigma_{g,1},\mathbb{H}_g)$ we have 
$$\rho_1(\tau_{\alpha})([\beta])=[\beta]+\langle \alpha,\beta\rangle_{\varphi_1^{\Heis_1}} [\alpha] \ \mathrm{mod} \ \sigma^{2k}-1 .
$$
In other words, as $\calH_{1,2k}$ designates the twisted homology by $\Heis_{g,2k},$ one has:
$$\rho^{\Heis_{g,2k}}_1(\tau_{\alpha})([\beta])=[\beta]+\langle \alpha,\beta\rangle_{\varphi_1^{\Heis_{g,2k}}} [\alpha]  .
$$
\end{coro}
\begin{proof}
It is a reduction mod $\sigma^{2k}-1$ of Prop.~\ref{P:twisted_transvection_formula_1pt}. We remark that modulo $\sigma^{2k}-1,$ the simple closed curve $\alpha$ is a cycle, and equal to the $\alpha'$ of Proposition \ref{P:twisted_transvection_formula_1pt}. Indeed, a generator of $\pi_1(\alpha)$ is sent to $\sigma^{\pm 2k}$ via $\varphi_1^{\Heis_g}$ thanks to Prop. \ref{prop:HeisenbergSimpleCurve}. It should depend on a choice of relation to the base point, but for two different choices, the results are conjugated, and $\sigma$ is central. Hence modulo $\sigma^{2k}-1$ the fundamental group of $\alpha$ is trivialized which is the condition to lift manifolds to the associated cover. 
\end{proof}
Let us now consider two separating curves $\alpha_1$ and $\alpha_2,$ of genera $k_1$ and $k_2.$ We will write $\Sigma^{int,int}(\alpha_1,\alpha_2)$ for the connected components of $\Sigma \setminus (\alpha_1 \cup \alpha_2)$ which lies in the interior of $\alpha_1$ and of $\alpha_2.$ Similarly we define $\Sigma^{int,ext}(\alpha_1,\alpha_2),\Sigma^{ext,int}(\alpha_1,\alpha_2)$ and $\Sigma^{ext,ext}(\alpha_1,\alpha_2).$
\begin{proposition}\label{prop:close_Dehn_twists}
Assume that $\rho_1(\tau_{\alpha_1})=\rho_1(\tau_{\alpha_2}).$ Then $\Sigma^{int,ext}(\alpha_1,\alpha_2)$ and $\Sigma^{ext,int}(\alpha_1,\alpha_2)$ are unions of disks.
    
\end{proposition}
\begin{proof}
Assume to the countrary that $\Sigma^{int,ext}(\alpha_1,\alpha_2)$ contains a component of genus at least one. Note that Proposition \ref{P:twisted_transvection_formula_1pt} implies that $\rho_1(\tau_{\alpha_1})$ acts by multiplication by $\sigma^{-2k_1}$ on $H_1(\Sigma^{int}(\alpha_1),\Sigma \setminus \Sigma^{int}(\alpha_1); \varphi_1^{\Heis_g}),$ while $\rho_1(\tau_{\alpha_2})$ acts as the identity on $H_1(\Sigma^{ext}(\alpha_2),\Sigma \setminus \Sigma^{ext}(\alpha_2); \varphi_1^{\Heis_g}).$ So their action differ on $H_1(\Sigma^{int,ext}(\alpha_1,\alpha_2),\Sigma \setminus \Sigma^{int,ext}(\alpha_1,\alpha_2); \varphi_1^{\Heis_g}),$ as soon as this space is not the trivial module, which is the case if $\Sigma^{int,ext}(\alpha_1,\alpha_2)$ contains a positive genus surface (just take an arc representing a non zero homology class in that component).
\end{proof}

\begin{remark}
We recall that $\rho_1$ is faithful if and only if it distinguishes all separating Dehn twists. The above proposition gives a strong restriction on separating simple closed curves  satisfying $\rho_1(\tau_{\alpha_1})=\rho_1(\tau_{\alpha_2})$
\end{remark}

\begin{proposition}\label{P:infinite_order}
Images of separating Dehn twists by $\rho_1$ have infinite order
\end{proposition}
\begin{proof}
It is enough to show this for the representation $\rho_1^{\Heis_{g,1}},$ with coefficients in $\Heis_g/\langle \sigma \rangle \simeq H_1(\Sigma_{g,1}).$ In fact, this representation is exactly the Magnus representation.

It is a well-known fact that the image of separating Dehn twists by the Magnus representation have infinite order. For the sake of completedness, let us give an argument using the transvection formula:

Let $\alpha$ be a separating closed curve in $\Sigma_{g,1},$ and $\gamma$ is an arc in $\Sigma_{g,1}$ with endpoints in $\partial\Sigma_{g,1}$ 
then 
$$\rho_1^{\Heis_{g,1}}(\tau_{\alpha})([\gamma])=[\gamma]+\langle \gamma, \alpha \rangle_{\Heis_{g,1}}[\alpha],$$
and thus
$$\rho_1^{\Heis_{g,1}}(\tau_{\alpha}^n)([\gamma])=[\gamma]+ n\langle \gamma, \alpha \rangle_{\Heis_{g,1}}[\alpha].$$
Therefore it suffices to show that there is an arc $\gamma$ such that $\langle \gamma, \alpha \rangle\neq 0.$ One can realize this with an arc that intersects $\alpha$ in exactly $2$ points, and such that a loop formed by a subarc of $\gamma$ and a subarc of $\alpha$ joining those $2$ points represents a non zero homology class.
\end{proof}

\subsection{Discussions on the kernel}
 \label{sec:kernel_1pt}
\subsubsection{Intersection form has kernel}
\label{sec:inter_form_kernel}
As explained in Theorem \ref{T:definite_implies_faithful}, a strategy to prove faithfulness of one of the representations $\rho_n^{\Heis_g}$ for $n\geq 1$ would be to show that the associated intersection form is definite. Unfortunately, for $n=1$ we prove the contrary. 
\begin{theorem}\label{T:kernel_n=1} The intersection form $\langle \cdot , \cdot \rangle_{\Heis_g}$ on $\mathcal{H}_1 \times \mathcal{H}_1^{\dagger}$ has kernel for any $g\geq 6.$
\end{theorem}
\begin{proof}
An example of a pair of arcs $\alpha$ and $\beta$ with non-zero geometric intersection but such that $\langle \alpha, \beta \rangle_{1,\Heis_g}=0$ is shown in Figure \ref{fig:example_kernel_intersection}.
\begin{figure}[h]
  \centering
  \def \svgwidth{0.5\columnwidth}
  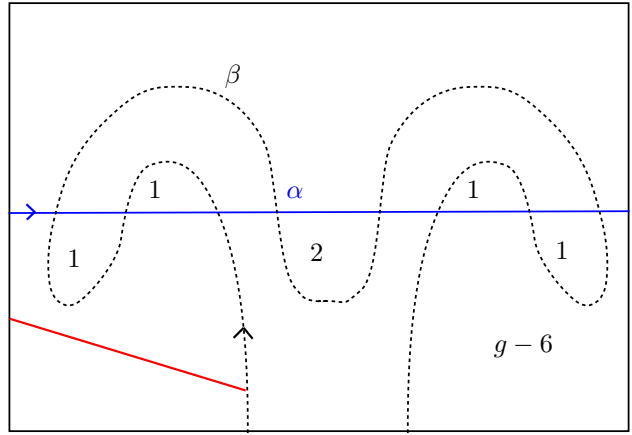
  \caption{Two disjoint arcs on which the intersection form $\langle \cdot , \cdot \rangle_{\Heis_g}$ vanishes. The surface $\Sigma_{g,1}$ is represented by a square with handles added. If a region is colored by an integer $k,$ it means that one needs to add $k$ handles in that region. In red, the thread for $\beta$ is represented.}
  \label{fig:example_kernel_intersection}
  \end{figure}
 Note that the two arcs are in minimal intersection positions by the bigon lemma since there is no disk in $\Sigma_{g,1}$ bounded by two subarcs of $\alpha$ and $\beta.$ Therefore their geometric intersection number is $8.$ Computing the intersection using Formula \ref{E:formula_for_the_pairing}, we get
$$ \langle \alpha, \beta \rangle_{\Heis_g}= \sigma^{-2}-\sigma^{-4}+\sigma^{-2}-1 +\sigma^{-4}-\sigma^{-2}+1-\sigma^{-2}=0.$$
In the above computation, we have ordered the contributions of intersection points in $\alpha \cap \beta$ following the orientation of $\alpha.$
\end{proof}
We note that Long and Paton \cite{LongPaton} and Bigelow \cite{Big99} have used the kernel of the twisted intersection pairing associated to the Burau representations to give kernel elements of the Burau representation of $B_n$ for $n\geq 6.$ The same kind of argument appears also in Suzuki's proof \cite{S02, S05b} that the Magnus representation of the Torelli group $\mathcal{I}(\Sigma_{g,1})$ has kernel for $g\geq 2$. The argument of Suzuki uses curves $\alpha$ and $\beta$ that are orthogonal for the twisted intersection pairing, and a transvection formula for the action of Dehn twists to show that $[\tau_{\alpha},\tau_{\beta}]$ is in the kernel of the Magnus representation.

However, in the proof of Theorem \ref{T:kernel_n=1}, we found non-disjoint \emph{arcs} $\alpha$ and $\beta$ with zero twisted intersection. It is not clear how to associate kernel elements of $\rho_1$ to this pair of arcs. One might be tempted to study the commutator $[\tau_{\alpha'},\tau_{\beta'}]$ of the Dehn twists along $\alpha',\beta',$ which simple closed curves obtained by closing $\alpha$ and $\beta$ using sub-arcs of $\partial\Sigma_{g,1}.$ However, we have computed that $[\tau_{\alpha'},\tau_{\beta'}]$ is not in the kernel of $\rho_1.$ One issue is that Proposition \ref{P:twisted_transvection_formula_1pt} only gives a \emph{twisted} transvection formula for Dehn-twist, and the factors $\sigma^{2k}$ in the formula prevents $\rho_1([\tau_{\alpha'},\tau_{\beta'}])$ to vanish, even when the twisted intersection of $\alpha$ and $\beta$ vanishes.

However, the above proposition shows that one can not apply Bigelow's strategy (Theorem \ref{T:definite_implies_faithful}) in the single point case to prove the faithfulness of the representation $\rho_1.$  

\subsubsection{Kernel elements for the Heisenberg $ \mod \sigma^{2k}-1$ representation}\label{S:kernel_n=1_mod}
\label{sec:kernelmodsigma2K}
We recall that $\rho_{1,k}$ stands for the representation obtained from $\rho_1$ by reducing coefficients mod $\sigma^{2k}-1$ which is equivalent to the homological representations associated with the $\Heis_{g,2k}$-twisted representation. We find kernel elements of this representation, provided that the genus is large enough:
\begin{proposition}\label{prop:kernel_modsigma2k}
 Let $g\geq 3k,$ then the representation $\rho_{1,k}$ has kernel.   
\end{proposition}
\begin{proof}
\begin{figure}[h]
  \centering
  \def \svgwidth{0.4\columnwidth}
  %% Creator: Inkscape 1.0 (4035a4fb49, 2020-05-01), www.inkscape.org
%% Accompanies image file 'noyau1point.pdf' (pdf, eps, ps)
\begingroup%
  \makeatletter%
  \providecommand\color[2][]{%
    \errmessage{(Inkscape) Color is used for the text in Inkscape, but the package 'color.sty' is not loaded}%
    \renewcommand\color[2][]{}%
  }%
  \providecommand\transparent[1]{%
    \errmessage{(Inkscape) Transparency is used (non-zero) for the text in Inkscape, but the package 'transparent.sty' is not loaded}%
    \renewcommand\transparent[1]{}%
  }%
  \providecommand\rotatebox[2]{#2}%
  \newcommand*\fsize{\dimexpr\f@size pt\relax}%
  \newcommand*\lineheight[1]{\fontsize{\fsize}{#1\fsize}\selectfont}%
  \ifx\svgwidth\undefined%
    \setlength{\unitlength}{210.6230488bp}%
    \ifx\svgscale\undefined%
      \relax%
    \else%
      \setlength{\unitlength}{\unitlength * \real{\svgscale}}%
    \fi%
  \else%
    \setlength{\unitlength}{\svgwidth}%
  \fi%
  \global\let\svgwidth\undefined%
  \global\let\svgscale\undefined%
  \makeatother%
  \begin{picture}(1,0.69267517)%
    \lineheight{1}%
    \setlength\tabcolsep{0pt}%
    \put(0,0){\includegraphics[width=\unitlength,page=1]{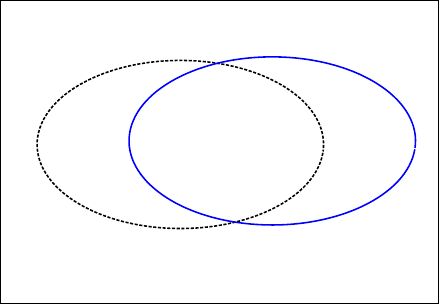}}%
    \put(0.43165685,0.33734507){\makebox(0,0)[lt]{\lineheight{1.25}\smash{\begin{tabular}[t]{l}$k$\end{tabular}}}}%
    \put(0.17267142,0.35353164){\makebox(0,0)[lt]{\lineheight{1.25}\smash{\begin{tabular}[t]{l}$k$\end{tabular}}}}%
    \put(0.79135804,0.34453909){\makebox(0,0)[lt]{\lineheight{1.25}\smash{\begin{tabular}[t]{l}$k$\end{tabular}}}}%
    \put(0.1671245,0.56977571){\color[rgb]{0,0,0}\makebox(0,0)[lt]{\lineheight{1.25}\smash{\begin{tabular}[t]{l}$\beta$\end{tabular}}}}%
    \put(0.68621807,0.57876824){\color[rgb]{0,0,1}\makebox(0,0)[lt]{\lineheight{1.25}\smash{\begin{tabular}[t]{l}$\alpha$\end{tabular}}}}%
    \put(0.17908273,0.07613033){\makebox(0,0)[lt]{\lineheight{1.25}\smash{\begin{tabular}[t]{l}$g-3k$\end{tabular}}}}%
  \end{picture}%
\endgroup%

  \caption{The images of the Dehn twists $\tau_{\alpha}$ and $\tau_{\beta}$ by $\rho_{1,k}$ commute}
  \label{fig:example_kernel_modk}
  \end{figure}
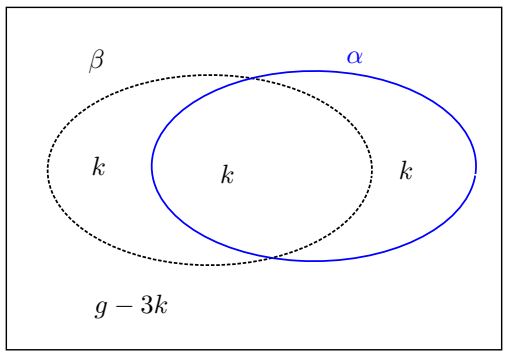
Figure \ref{fig:example_kernel_modk} shows two separating curves $\alpha$ and $\beta$ of genus $2k$ such that $\langle \alpha,\beta \rangle_{\Heis_{g,k}} =0.$ Indeed, the two curves have only two intersection points, which form a bigon of genus $k.$ Applying Formula \eqref{E:formula_for_the_pairing} easily shows that the two contributions are $+1$ and $-\sigma^{2k}=-1,$ and therefore $\langle \alpha,\beta \rangle_{\Heis_{g,k}} =0.$ (Here, let us note that we make an arbitrary choice of lift and orientations of $\alpha$ and $\beta$, which has no effect on the vanishing of $\langle \alpha,\beta \rangle_{\Heis_{g,k}}$).

Now let $[\gamma] \in H_1(\Sigma_{g,1},\partial\Sigma_{g,1},\Heis_{g,k}).$ By Corollary \ref{C:transvection_mod2k_formula} we have
$$\rho_{1,k}(\tau_{\alpha}\tau_{\beta})([\gamma])=[\gamma]+\langle \beta, \gamma \rangle_{\Heis_{g,k}} [\beta] +\langle \alpha, \gamma \rangle_{\Heis_{g,k}} [\alpha]+ \langle \beta, \gamma \rangle_{\Heis_{g,k}} \langle \alpha, \beta \rangle_{\Heis_{g,k}}[\alpha].$$
Since $\langle \alpha, \beta \rangle_{\Heis_{g,k}}=0,$ and since the only non symmetric term in $\alpha, \beta$ in the above formula factors by $\langle \alpha, \beta \rangle_{\Heis_{g,k}}$, we get that $\rho_{1,k}(\tau_{\alpha}\tau_{\beta})=\rho_{1,k}(\tau_{\beta}\tau_{\alpha}).$ Thus $[\tau_{\alpha},\tau_{\beta}]\in \mathrm{Ker} \rho_{1,k}.$ We recall that Dehn twists along curves with positive geometric intersection never commute in the mapping class group, hence $\rho_{1,k}$ is not faithful.
\end{proof}

\section{The several-points case}
\label{sec:several_pts}
\subsection{Kernel in the intersection form $\mod \sigma^{2k},$ where $k<g$}\label{sec:kernel_sev_pts}

In this section, we show that the twisted intersection form on $\calH_{n,2k}$ is not definite for any $n\geq 1$ and any $k<g.$ Indeed, we have:
\begin{proposition}\label{prop:kernel_inter_form_npt}
	Let $\alpha$ and $\beta$ be two arcs in $\Sigma_{g,1}$ with endpoints in $\partial \Sigma_{g,1}^-$ and $\partial \Sigma_{g,1}^+$ respectively. Let $\alpha^{(n)} \in \calH_{n,2k}$ and $\beta^{[n]} \in \calH^{\dagger}_{n,2k}$ be the associated homology classes.
	
	Assume that $\alpha$ and $\beta$ meet in exactly two points, and that there are arcs $a$ of $\alpha$ and $b$ of $\beta$ such that the bigon $a\cup b$ is a separating simple closed curve of genus $d,$ where $d$ divides $k.$ Then $\langle \alpha^{(n)},\beta^{[n]}\rangle_{\Heis_{g,2k}}=0,$ although $\alpha$ and $\beta$ have non-zero geometric intersection.
\end{proposition}
\begin{proof}
	The fact that $\alpha$ and $\beta$ have non-zero geometric intersection follows from the bigon lemma \cite[Proposition 1.7]{FarbMargalit}. To show that $\langle \alpha^{(n)},\beta^{[n]}\rangle_{\Heis_{g,2k}}=0,$ notice that by Equation \ref{E:formula_for_the_pairing}, the intersection $\langle \alpha^{(n)},\beta^{[n]}\rangle_{\Heis_{g,2k}}$ is the same as the intersection $\langle \alpha^{(n)},\beta'^{[n]}\rangle_{\Heis_{g,2k}}$ where $\beta'$ is an arc obtained from $\beta$ by replacing the subarc $b$ by an arc parallel to $a$ (while keeping the same two intersection points between $\alpha$ and $\beta'$ as those between $\alpha$ and $\beta$). Indeed, for any $\underline{x},$ a $n$-uple of intersection points between $\alpha$ and the $n$-th parallelization of $\beta$ or $\beta',$ the loop $\delta_{\underline{x}}'$ that appears in Equation \ref{E:formula_for_the_pairing} for $\langle \alpha^{(n)},\beta'^{[n]}\rangle_{\Heis_{g,2k}}$ is obtained from $\delta_{\underline{x}}$ by inserting loops that have one point going along $a\cup b$ and all other configuration points constant. These loops are sent to $\sigma^{\pm 2d}$ by the same argument as Proposition \ref{prop:HeisenbergSimpleCurve}. Since $\sigma$ is central in $\Heis_g$ and we are working modulo $\sigma^{2k}-1,$ we get:
	$$\varphi_n^{\Heis_{g,2k}}(\delta_{\underline{x}})=\varphi_n^{\Heis_{g,2k}}(\delta_{\underline{x}}'),$$
	for any such $\underline{x},$ and therefore:
	$$\langle \alpha^{(n)},\beta'^{[n]}\rangle_{\Heis_{g,2k}}=\langle \alpha^{(n)},\beta'^{[n]}\rangle_{\Heis_{g,2k}}=0,$$
	since $\langle \alpha^{(n)},\beta'^{[n]}\rangle_{\Heis_{g,2k}}$ depends only on the isotopy classes (rel. boundary) of $\alpha$ and $\beta'$ and $\alpha$ and $\beta'$ can be isotoped to be disjoint. 
\end{proof}
\begin{remark}\label{remark:more_curves_in_kernel_n_pts}
	The proof of Proposition \ref{prop:kernel_inter_form_npt} shows that many more examples of geometrically intersecting pairs of arcs $(\alpha,\beta)$ with vanishing twisted intersection $\langle \alpha^{(n)},\beta^{[n]}\rangle_{\Heis_{g,2k}}$ can be constructed. For instance, one could assume $\alpha$ and $\beta$ to intersect so that they form $m$ "bigons" with genus $d_1,\ldots,d_m,$ where all $d_i$ divide $k.$
\end{remark}
\subsection{A simplified formula for the $n$ points intersection form}
\label{sec:formula_sev_pts}
Next, we give a formula that simplifies the computation of the twisted intersection form on $\calH_n$ for $n>2,$ for some specific types of arcs. We will fix a $g$-holed sphere $S$ embedded in $\Sigma_{g,1}$ so that $\Sigma_{g,1}$ is the union of $S$ and $g$ one-holed tori. We will compute $\langle \alpha^{(n)},\beta^{[n]} \rangle$ when $\alpha^{(n)}\in \calH_n$ and $\beta^{[n]} \in \calH_n^{\dagger}$ are homology classes associated to two oriented separating arcs $\alpha$ and $\beta,$ with the added assumption that $\alpha$ and $\beta$ belong to the $g$-holed sphere $S.$ Because of this assumption, any loop in $\Sigma_{g,1}$ consisting of a subarc of $\alpha$ and a subarc of $\beta$ will be separating.

Let  $x_0,x_1,\ldots ,x_{k-1}$ be the points of intersection of $\alpha$ and $\beta,$ where the label corresponds to their order on $\alpha.$  We will assume that the sign of $x_0$ as an intersection point $\alpha$ and $\beta$ is $+1.$ We will also consider intersection points of $\alpha$ with a $2$-parallel cable $\beta^{[2]}$ of $\beta,$ in which case we will write $x_0,\ldots,x_{k-1}$ for the intersection points of $\alpha$ with the left component of $\beta^{[2]}$ and $\overline{x_0},\ldots,\overline{x_{k-1}}$ for the intersection points with the right component of $\beta^{[2]}.$ We recall that a basepoint $\xi=\{\xi_1,\ldots,\xi_n\}$ of $\Conf_n(\Sigma_{g,1})$ is fixed, where $\xi_1,\ldots,\xi_n$ are points on the interval $\partial^+ \Sigma_{g,1}.$

For $j\in \lbrace 0,\ldots,k-1\rbrace,$ let $\eta_j$ be the oriented loop which consists of connecting $\xi_1$ to $x_j$ along $\alpha$ then $x_j$ to $\xi_1$ in $\beta.$ We can write $\varphi_1^{\Heis_g}(\eta_j)=\sigma^{2n_j}$ for some integer $n_j\in \Z,$ since $\eta_j$ is a loop in $S.$ This is due to the loop being separating and Prop.~\ref{prop:HeisenbergSimpleCurve}. 

Next, for $j<j' \in \lbrace 0,\ldots,k-1 \rbrace,$ we consider the $2$-braid $\beta_{j,j'}$ in $\Sigma_{g,1}$ which consists of connecting $( \xi_1,\xi_2)$ along $\alpha$ to $(x_j,\overline{x_{j'}}),$ then $(x_j,\overline{x_{j'}})$ to $(\xi_1,\xi_2)$ or $(\xi_2,\xi_1)$ along $\beta^{[2]}.$ 

Let us write $\beta_{j,j'}(t)=\lbrace y_1(t),y_2(t)\rbrace,$ where $y_i(t)$ are continuous path in $\Sigma_{g,1}.$  Since the $2$-braid $\beta_{j,j'}$ can be considered a $2$-braid in the $g$-holed sphere $S,$ we can consider its index $A_{j,j'},$ which can be defined as the number of half turns that the vector $\overset{\longrightarrow}{y_1(t)y_2(t)}$ does along $\beta_{j,j'}.$ Alternatively, $A_{j,j'}$ is defined by the formula:
$$\beta_{j,j'}=(\eta_j,\xi_2)\sigma_1^{A_{j,j'}}(\eta_{j'},\xi_2),$$
considered as an identity of $2$-braids in $\Sigma_{g,1}.$ Here $\sigma_1$ is the element braid which exchanges $\xi_1$ and $\xi_2,$ introduced in Equation \ref{eq_localsystem_of_Heisenberg}. Note that we have
$$\varphi_2^{\Heis_g}(\beta_{j,j'})=\sigma^{2n_j+2n_{j'}+A_{j,j'}}.$$
We claim that the data $\lbrace n_j, 0\leq j \leq k-1 \rbrace$ and $\lbrace A_{j,j'}, 0\leq j < j' \leq k-1 \rbrace$ is sufficient to compute the $n$-point twisted intersection $\langle \alpha^{(n)},\beta^{[n]} \rangle_{\Heis_g}$ for any $n\geq 2:$
\begin{proposition}
	Let $\alpha$ and $\beta$ as above, and $n\geq 2.$ Then
	$$\langle \alpha^{(n)},\beta^{[n]}\rangle_{\Heis_g}=\underset{0\leq i_1,\ldots,i_n \leq k-1}{\sum}(-1)^{\underset{j=1}{\overset{n}{\sum}} i_j}\sigma^{2\underset{j=1}{\overset{n}{\sum}} n_j}(-\sigma)^{\underset{1\leq l<j\leq n}{\sum}A_{i_l,i_j}}.$$
\end{proposition}
\begin{proof} We apply the pairing formula given in Equation \ref{E:formula_for_the_pairing}. The intersection points $\underline{x}$ between $\alpha^{(n)}$ and $\beta^{[n]}$ are in correspondance with $n$-tuples $(i_1,\ldots,i_n)$ with $0\leq i_j \leq k-1,$ the index $i_j$ indicating that the intersection point chosen for the $j$-th parallel cable of $\beta$ and $\alpha$ is $x_j.$ The sign $\varepsilon_j$ of that intersection point is $(-1)^{i_j}.$ Indeed, we have chosen the sign of $x_0$ to be $+1,$ and $\alpha$ being separating, the sign of intersection points $x_0,x_1,\ldots,x_{k-1}$ alternate. It remains to show that for the loop $\delta_{\underline{x}}$ appearing in Equation \ref{E:formula_for_the_pairing}, we have
	$$\varphi_n^{\Heis_g}(\delta_{\underline{x}})=\sigma^{2\underset{j=1}{\overset{n}{\sum}} n_j}\sigma^{\underset{1\leq l<j\leq n}{\sum}A_{i_l,i_j}}.$$
	Note that the loop $\delta_{\underline{x}}$ is actually a loop in $\Conf_n(S),$ and that the image of $\pi_1(\Conf_n(S))$ by $\varphi_n^{\Heis_g}$ is $\langle \sigma \rangle$ (thanks to Proposition \ref{prop:HeisenbergSimpleCurve}) which is abelian. So it is sufficient to know to which element of the abelianization $\pi_1(\Conf_n(S))=\pi_1(\Conf_n(D_g))$ the loop $\delta_{\underline{x}}$ corresponds. However, it is well-known that this abelianization is $\Z^{g+1},$ where each of the first $g$ factors correspond to looping one point around one hole, and the last factor corresponds to the total winding number.
	The total winding number of $\delta_{\underline{x}}$ is exactly $ \underset{1\leq l<j\leq n}{\sum}A_{i_l,i_j},$ and it contributes $\sigma^{\underset{1\leq l<j\leq n}{\sum}A_{i_l,i_j}}$ to $\varphi_n^{\Heis_g}(\delta_{\underline{x}}),$ while looping around the holes is accounted by the factor $\sigma^{2\underset{j=1}{\overset{n}{\sum}} n_j},$ thanks to Proposition \ref{prop:HeisenbergSimpleCurve}.
	
\end{proof}

\section{Recovering some Lawrence representations of braid groups}
\label{sec:Lawrence}

In this section we are interested in recovering \textit{Lawrence representations} of braid groups as subrepresentations of $\rho_n^{\Heis_g}$. We first recall what we mean by Lawrence representations and what we need to know. The $n$-th (colored) Lawrence representation is the homological representation of the pure braid group $\mathcal{PB}_k$:
\begin{equation}\label{E:colored_Lawrence}
\mathcal{L}^c_{k,n} : \mathcal{PB}_k \to \Aut_{\Z[s_1^{\pm 1},\ldots,s_k^{\pm 1}, \sigma^{\pm 1}]} \left( H_n^{BM}(\mathrm{Conf}_n(D_k),\mathrm{Conf}_n(D_k)^-;\varphi) \right).
\end{equation}
In the above, the local system on $\mathrm{Conf}_n(D_k)$ is such that $\varphi((\gamma,\xi_2,\ldots,\xi_n))=s_i$ if $\gamma$ is a loop based at $\xi_1$ going once counterclockwise around the puncture $w_i$, and $\varphi(\sigma_{i,i+1})=\sigma,$ if $\sigma_{i,i+1}$ is the standard braid generator of $D_{k+n}$ that exchanges points $\xi_i$ and $\xi_{i+1}.$ The local system is described in \cite[Def.~2.5]{JulesVerma} (where $\sigma$ is $t$). Now these precise relative versions of Lawrence representations are defined in \cite[Lemma~6.35]{JulesVerma}. By relative version we mean that the homology involved is relative to part of the boundary while original Lawrence representations were defined on absolute homologies. The exponent $c$ refers to the fact that these representations are colored and hence must be restricted to the pure braid group. Under the morphism:
\begin{equation}\label{E:uncolored}
\Z[s_1^{\pm 1},\ldots,s_k^{\pm 1}, \sigma^{\pm 1}] \to \Z[s^{\pm 1}, \sigma^{\pm 1}], s_i \mapsto s,
\end{equation} 
one obtains the Lawrence representations of braid groups:
\begin{equation}\label{E:Lawrence_rep}
\mathcal{L}_{k,n} : \mathcal{B}_k \to \Aut_{\Z[s^{\pm 1}]} \left( H_n^{BM}(\mathrm{Conf}_n(D_k),\mathrm{Conf}_n(D_k)^-;\varphi') \right),
\end{equation}
where $\varphi'$ is $\varphi$ postcomposed with \eqref{E:uncolored}. They are well defined on the whole braid group, see \cite[Lemma~6.33]{JulesVerma}. Let $D_k^\circ$ denote the disk with $k$-holes that is with $k$ disks removed (rather than punctures). Notice that it has the same fundamental group as $D_k$ and their configuration spaces are also homotopy equivalent.

\begin{equation*}
\vcenter{\hbox{\begin{tikzpicture}[scale=0.5, every node/.style={scale=0.75},decoration={
    markings,
    mark=at position 0.5 with {\arrow{>}}}
    ] 
    
    \filldraw[color=red!60, fill=red!5,thick](-4,0) circle (0.5);  
    \filldraw[color=red!60, fill=red!5,thick](-1,0) circle (0.5);  
    \node[red] at (0.5,0) {$\cdots$}; 
    \filldraw[color=red!60, fill=red!5,thick](2,0) circle (0.5);
    \filldraw[color=red!60, fill=red!5,thick](5,0) circle (0.5);
    \draw[dashed,green!70!black,postaction=decorate,thick] (-4,-3)--(-4,0) node[midway,left] {$i_1$};
    \node[red] (w1) at (-4,0) {$\bullet$};
    \draw[dashed,green!70!black,postaction=decorate,thick] (-1,-3)--(-1,0) node[midway,left] {$i_2$};
    \node[red] (w2) at (-1,0) {$\bullet$};
    \draw[dashed,green!70!black,postaction=decorate,thick] (2,-3)--(2,0) node[midway,left] {$i_{k-1}$};
    \node[red] (wk1) at (2,0) {$\bullet$};
    \draw[dashed,green!70!black,postaction=decorate,thick] (5,-3)--(5,0) node[midway,left] {$i_{k}$};
    \node[red] (wk1) at (5,0) {$\bullet$};

    \draw[dashed,blue!70!black,postaction=decorate,thick] (-5,-3)--(-5,0.5) .. controls (-4.5,1.5) and (-3.5,1.5) .. (-3,0.5) node[midway,above] {$j_1$} --(-3,-3);
    \draw[dashed,blue!70!black,postaction=decorate,thick] (-2,-3)--(-2,0.5) .. controls (-1.5,1.5) and (-0.5,1.5) .. (-0,0.5)  node[midway,above] {$j_2$} --(-0,-3);
    \draw[dashed,blue!70!black,postaction=decorate,thick] (1,-3)--(1,0.5) .. controls (1.5,1.5) and (2.5,1.5) .. (3,0.5)  node[midway,above] {$j_{k-1}$} --(3,-3);
    \draw[dashed,blue!70!black,postaction=decorate,thick] (4,-3)--(4,0.5)  .. controls (4.5,1.5) and (5.5,1.5) .. (6,0.5) node[midway,above] {$j_k$} --(6,-3);

\draw[gray, thick] (-6,-3) -- (-6,3);
\draw[gray, thick] (7,-3) -- (7,3);
\draw[magenta, thick] (-6,-3) -- (7,-3);
\draw[gray, thick] (-6,3) -- (7,3);
\end{tikzpicture}}},
\end{equation*}

We denote by $A_{G}(i_1,\ldots,i_k)$ the diagram in green in the above picture and by $A_{B}(j_1,\ldots,j_k)$ the one in blue. By adding to them handles namely paths rejoining base points, they define homology classes and Theorem~\ref{thm_structure_twisted_homology} adapts to our situations so that we have:
\begin{equation}\label{E:structure_punctured_disk}
\calH_n(D_k) = \Span_{\Z[G]} \langle A_{G}(i_1,\ldots,i_k), i_1 + \cdots + i_k = n \rangle, 
\end{equation}
\begin{equation}\label{E:structure_holed_disk}
\calH_n(D_k^\circ) = \Span_{\Z[G]} \langle A_{B}(j_1,\ldots,j_k), j_1 + \cdots + j_k = n \rangle 
\end{equation}
Hence these modules are free on $\Z[G]$ for any $\varphi: \pi_1(\Conf_n(D_k)) \to G$. It is true up to any convenient choice of threads so that we do not draw them (two different choices lead to a diagonal change of basis given by multiplication by an invertible element). In the Lawrence case, when $\Z[G]$ is the ring of Laurent polynomials, one has that $\calH_n(D_k^\circ)$ injects in $\calH_n(D_k)$, and the diagonal change of basis between the $A_G$ and the $A_B$ is precisely (up to multiplication by a monomial) that given in \cite[Prop.~7.4]{JulesVerma} (and one should remove the quantum binomials in the formula). It is a diagonal change of basis but not invertible when working over the ring $\Z[G].$

We recall the new approach to these representations provided by Theorem~\ref{P:Universal_construction_of_homol_reps} that reconstructs homological representations from a free module spanned by diagrams after killing the kernel of a pairing. 

\begin{remark}\label{R:universal_construction_of_Lawrence}
We let $V(D_k)$ (resp. $V(D^\circ_k)$) be the free  $\Z[s_1^{\pm 1},\ldots,s_k^{\pm 1}, \sigma^{\pm 1}]$-module of diagrammatic twisted classes in $D_k$ (resp. in $D_k^\circ$). % and $V^\dagger(D_k)$ resp. $V^\dagger(D^\circ_k)$ the one made of dual diagrams (see Sec.~\ref{S:alternative_construction}). 
 Then $\calH_n(D_k)$ (resp. $\calH_n(D_k^\circ)$) is the quotient of $V(D_k)$ (resp. $V(D_k^\circ)$) by the left kernel of the pairing defined by the formula in Section~\ref{sec:inter_form_def} where the local system is the one defined above yielding coefficients in $\Z[s_1^{\pm 1},\ldots,s_k^{\pm 1}, \sigma^{\pm 1}]$. The representation of $\mathcal{L}_{k,n}^c$ is constructed using the action of the mapping class group in the basis and the pairing. We denote by $\mathcal{L}_{k,n}^{c,\circ}$ the representation on $\calH_n(D_k^\circ)$
\end{remark}

\begin{remark}\label{R:Lawrence_faithfulness}
Thanks to \cite{Big00}, the representation $\mathcal{L}^\circ_{k,2}$ (uncolored) is faithful when restricted to the representation of $\calB_k$ on a submodule of $\calH_2(D_k^\circ)$ constructed from the absolute homology. The fact that it is a submodule is a consequence of the short exact sequence in the proof of \cite[Corollary~7.1]{JulesVerma}. Hence $\mathcal{L}_{k,2}^{c,\circ}$ is a faithful representation of $\mathcal{PB}_k$ because it specializes to $\mathcal{L}^\circ_{k,2}$. Actually $\mathcal{L}_{k,n}^c$ is faithful too for all $n > 2,$ as a consequence of the surjective map in the same short exact sequence just mentioned since they all surject to the $n=2$ case. 
\end{remark}

%\subsection{Homological representation in an universal construction way}
\subsection{Evaluated Lawrence representations as a subrepresentation of a pure braid group on $g$ strands made of separating twists}%Evaluated Lawrence representations of the pure braid groups as subrepresentations of the Heisenberg representations}
\label{sec:pure_braid_Lawrence}
Let $S$ be a disk with $g$ holes embedded in $\Sigma_{g,1},$ so that the boundary components of $S$ consist of the boundary component of $\Sigma_{g,1}$ and $g$ non parallel separating closed curves $c_1,\ldots,c_g,$ each of genus $1.$ The surface $\Sigma_{g,1}$ is then the union of $S$ and $g$ one-holed tori. Here is a picture of $S$:

\begin{equation*}
\vcenter{\hbox{\begin{tikzpicture}[scale=0.5, every node/.style={scale=0.75},decoration={
    markings,
    mark=at position 0.5 with {\arrow{>}}}
    ] 
    
    \filldraw[color=red!60, fill=red!5,thick](-4,0) circle (0.5);  
    \filldraw[color=red!60, fill=red!5,thick](-2,0) circle (0.5);  
    \node[red] at (0,0) {$\cdots$}; 
    \filldraw[color=red!60, fill=red!5,thick](2,0) circle (0.5);
    \filldraw[color=red!60, fill=red!5,thick](4,0) circle (0.5);
    \node[red,above] at (-4,0.5) {$c_1$};
    \node[red,above] at (-2,0.5) {$c_2$};
    \node[red,above] at (2,0.5) {$c_{g-1}$};
    \node[red,above] at (4,0.5) {$c_g$};

\draw[gray, thick] (-5,-3) -- (-5,3);
\draw[gray, thick] (5,-3) -- (5,3);
\draw[magenta, thick] (-5,-3) -- (5,-3);
\draw[gray, thick] (-5,3) -- (5,3);
\end{tikzpicture}}},
\end{equation*}
where to embed it into $\Sigma_{g,1}$, one has to connect sum a torus along each red disk. 

Note that $\mathrm{Mod}(S)$ is isomorphic to the framed pure braid group on $g$ strands denoted by $\mathcal{PB}_g^f.$ Moreover, we have that the quotient
$$\mathrm{Mod}(S)/\langle \tau_{c_1},\ldots,\tau_{c_g} \rangle \simeq \mathcal{PB}_g$$ 
is a pure braid group on $g$ strands.

Note that $\mathrm{Mod}(S)$ is generated by Dehn-twists along simple closed curves in $S,$ however, any simple closed curve in $S$ is separating in $\Sigma_{g,1}.$ Therefore, $\mathrm{Mod}(S)\subset K_{\Heis_g},$ since any separating twist is in $K_{\Heis_g},$ and the representation $\rho_n^{\Heis_g}$ is defined on $\mathrm{Mod}(S).$

Restricted to this subgroup, the representation $\rho_n^{\Heis_g}$ admits a subrepresentation. Indeed, let 
$$\calH^S_{n}=H_n^{BM}(\mathrm{Conf}_n(S),\mathrm{Conf}_n(S)^-;(\varphi_n^{\Heis_g})_{|\pi_1(\Conf_n(S))})$$
where $\mathrm{Conf}_n(S)^-$ denotes configurations of $n$ points in $S$ such that at least one point belongs to the interval $\partial^-\Sigma_{g,1}.$ It is clear that $\mathrm{Mod}(S)$ stabilizes this subspace, since $\mathrm{Mod}(S)$ stabilizes $S.$ Let $L$ be the image of $(\varphi_n^{\Heis_g})_{|\pi_1(\Conf_n(S))}$, we have recalled in \eqref{E:structure_holed_disk} that the $\Z[L]$-modules $\calH^S_n$ are free. Let us call $\rho_n^S$ this subrepresentation of $\rho_n|_{\mathrm{Mod}(S)}.$ 

%We recall the $n$-th Lawrence representation is the homological representation of $\mathcal{B}_k$:
%$$\mathcal{L}_{k,n} : \mathcal{B}_k \to \Aut_{\Z[s^{\pm 1}, \sigma^{\pm 1}]} \left( H_n^{BM}(\mathrm{Conf}_n(D_k),\mathrm{Conf}_n(D_k)^-;\varphi) \right).$$ In the above, the local system on $\mathrm{Conf}_n(D_k)$ such that $\varphi((\gamma,\xi_2,\ldots,\xi_n))=q$ if $\gamma$ is a loop based at $\xi_1$ going once counterclockwise around a puncture, and $\varphi(\sigma_{i,i+1})=\sigma,$ if $\sigma_{i,i+1}$ is the standard braid generator of $D_{k+n}$ that exchanges point $\xi_i$ and $\xi_{i+1}.$ These precise relative versions of Lawrence representations are defined in \cite[Lemma~6.33]{JulesVerma}.
Our claim is that the representation $\rho_n^S$ recovers an evaluated Lawrence representation of a pure braid group. 

%{\color{red} TO DO : ATTENTION, en fait dans cette base, et à cette spécialisation, la représentation tronque ! Elle donne pile la 2ème Jones. Je vais bien écrire ça.} 
\begin{theorem}
	\label{thm:pure_braid_group} The representation $\rho_n^S$ of $\mathrm{Mod}(S)$ induces a representation of the pure braid group $\mathcal{PB}_g$ which is isomorphic to the $n$-th Lawrence representation $\mathcal{L}_{g,n}$ from \eqref{E:Lawrence_rep} restricted to $\mathcal{PB}_g,$ and evaluated at $s=\sigma^{-2}.$
\end{theorem}
\begin{proof}
	First, notice that $\tau_{c_1},\ldots,\tau_{c_n}$ act trivially on $\calH_n^S,$ since $S$ lies in the exterior of each of the separating curves $c_1,\ldots,c_n.$
	So $\rho_n^S$ can be considered as a representation of 
	$$\mathcal{PB}_g\simeq \mathrm{Mod}(S)/\langle \tau_{c_1},\ldots,\tau_{c_g} \rangle.$$
	By Proposition \ref{P:Universal_construction_of_homol_reps} for Lawrence representations (summarized in Rem.~\ref{R:universal_construction_of_Lawrence}), it suffices to check that the twisted intersection forms of $\rho_n^S$ is the twisted intersection form for the Lawrence representation evaluated at $q=\sigma^{-2}.$
	By the definition of the twisted intersection given in Section \ref{sec:inter_form_def}, we only need to check that the local system
	$$\varphi_n^{\Heis_g}|_{\pi_1(\Conf_n(S))}: \pi_1(\Conf_n(S))\longrightarrow L$$
	coincides with the one for the evaluated Lawrence representation. However, it follows by Proposition \ref{prop:HeisenbergSimpleCurve} that this is the case, since the boundary components $c_1,\ldots,c_g$ each bound a one-holed torus. 
\end{proof}

\begin{remark}
It is not known to the authors whether the Lawrence representation involved in the above theorem is still faithful after the evaluation.
\end{remark}

\subsection{A pure braid subgroup on $g$ strands of $\Mod(\Sigma_{g,1})$ acts faithfully}\label{S:Pg_faithful}

We draw the surface $\Sigma_{g,1}$ from above:
\[
\vcenter{\hbox{\begin{tikzpicture}[scale=0.5, every node/.style={scale=0.75},decoration={
    markings,
    mark=at position 0.5 with {\arrow{>}}}
    ] 
    \coordinate (u) at (0,5);
    \coordinate (d) at (0,0);
    \node[above] at (u) {$U$};
    \node at (u) {$\bullet$};
    \node at (d) {$\bullet$};
    \node[below] at (d) {$D$};
    \coordinate (r) at (15.5,2.5);
    \node at (r) {$\bullet$};
    \node[right] at (r) {$R$};
    \draw[thick] (0,0)--(14,0) .. controls (16,0) and (16,5) .. (14,5) -- (0,5);
    \draw[thick](2.5,2.5) circle (0.5);  
    \draw[thick](5,2.5) circle (0.5);  
    \node[thick] at (7.5,2.5) {$\cdots$}; 
    \draw[thick](10,2.5) circle (0.5);
    \draw[thick](12.5,2.5) circle (0.5);
    \draw[thick,gray!60!white] (u) to[bend right=20] (d);
    \node[above] at (2.5,3) {$c_1$};
    \node[above] at (5,3) {$c_2$};
    \node[above] at (10,3) {$c_{g-1}$};
    \node[above] at (12.5,3) {$c_g$};
    \draw[thick,magenta] (u) to[bend left=20] (d);
    
        \draw[thick] (2.5,2) to[bend left=20] node[midway,right] {$\alpha_1$} (2.5,0);
    \draw[thick,dashed,gray!60!white] (2.5,2) to[bend right=20] (2.5,0);
        \draw[thick] (5,2) to[bend left=20] node[midway,right] {$\alpha_2$} (5,0);
    \draw[thick,dashed,gray!60!white] (5,2) to[bend right=20] (5,0);
        \draw[thick] (10,2) to[bend left=20] node[midway,right] {$\alpha_{g-1}$} (10,0);
    \draw[thick,dashed,gray!60!white] (10,2) to[bend right=20] (10,0);
        \draw[thick] (12.5,2) to[bend left=20] node[midway,right] {$\alpha_g$} (12.5,0);
    \draw[thick,dashed,gray!60!white] (12.5,2) to[bend right=20] (12.5,0);
%    \filldraw[color=red!60, fill=red!5,thick](-4,0) circle (0.5);  
%    \filldraw[color=red!60, fill=red!5,thick](-2,0) circle (0.5);  
%    \node[red] at (0,0) {$\cdots$}; 
%    \filldraw[color=red!60, fill=red!5,thick](2,0) circle (0.5);
%    \filldraw[color=red!60, fill=red!5,thick](4,0) circle (0.5);
%    \node[red,above] at (-4,0.5) {$c_1$};
%    \node[red,above] at (-2,0.5) {$c_2$};
%    \node[red,above] at (2,0.5) {$c_{g-1}$};
%    \node[red,above] at (4,0.5) {$c_g$};
%
%\draw[gray, thick] (-5,-3) -- (-5,3);
%\draw[gray, thick] (5,-3) -- (5,3);
%\draw[magenta, thick] (-5,-3) -- (5,-3);
%\draw[gray, thick] (-5,3) -- (5,3);
\end{tikzpicture}}}
\]
Namely we assume a coordinate system $(xyz)$ such that points $(UDR)$ form the $(xy)$ plane of equation $z=0$ which also contains curves $(c_1,\ldots,c_g)$, hence the $z$ axis arrives orthogonally to the eye of the reader. Let $$S_g=\Sigma_{g,1} \cap \{ z \ge  0 \},$$ then $S_g$ is a disk with $g$ holes. It is thus a sphere with $g+1$ holes but we give a particular status to the boundary component containing points $U,D,R$, hence made of the magenta arc and the arc that goes from $U$ to $D$ passing by $R$, we call this circle the boundary of the disk. In the figure we drew in light gray arcs that are not in $S_g$. 

\begin{definition}\label{R:Mod(S)_is_braids}
For $g>1$ notice that no boundary component of $S_g$ bounds a punctured disk, nor an annulus in $\Sigma_{g,1}$, so that \cite[Theorem~3.18]{FarbMargalit} claims that $G_g:=\Mod(S_g)$ injects in $\Mod(\Sigma_{g,1})$. The latter subgroup $G_g$ is isomorphic to a framed pure braid group on $g$ strands. We let:
\[
P_g := G_g \big/ \langle \tau_{c_1} , \ldots, \tau_{c_g} \rangle
\]
be the quotient of $G_g$ by the group generated by Dehn twists along curves $c_i$'s, and we notice that $P_g$ is isomorphic to a pure braid group on $g$ strands. The above quotient has a natural section, so that one can think of $P_g$ as a subgroup of $G_g$ or of $\Mod(\Sigma_{g,1})$. 
\end{definition}

We let $p_{i,j}$ be a path in $S_g$ joining $c_i$ and $c_j$ avoiding all the $\alpha_j$'s and $t_{i,j}$ be a simple closed curve of $S_g$ bounding a tubular neighborhood of $c_i \cup p_{i,j} \cup c_j$. It is well known that $G_g$ is generated by the Dehn twists along the $t_{i,j}$'s and along the $c_i$'s. 

We study the representation $\rho_n^{\Heis_g}$ restricted to $G_g$, we recall the morphism:
\begin{equation}\label{E:mcg_on_Heisg_recall}
\Mod(\Sigma_{g,1}) \to \Aut(\Heis_g), [f] \mapsto f_*. 
\end{equation}
\begin{proposition}\label{P:action_of_Gg_on_Heis}
Let $[f] \in G_g$, for any $i \in \{1,\ldots,g\}$, we have: %, let $B:=\langle b_1 , \ldots , b_g \rangle \triangleleft \Heis_g$,
\begin{align*}
& f_*(b_i) = b_i, \\
& f_*(a_i) = a_i b^{\boldsymbol{m}^i(f)},% sigma^{l_i(f)},
\end{align*}
%where $l_i \in \Z$ and:
where:
\[
b^{\boldsymbol{m}^i(f)} = b_1^{m^i_1(f)} \cdots b_g^{m^i_g(f)},
\]
so that ${\boldsymbol{m}^i(f)}$ is a notation for $(m^i_1(f),\ldots,m^i_g(f)) \in \Z^g$. 

Let $M(f)$ be the matrix with $i$-th column being ${\boldsymbol{m}^i(f)}$. We furthermore have that $M(f)$ is symmetric and that for $f_1,f_2 \in G$: 
\begin{align*}
%& {\boldsymbol{m}^i(f_1 \circ f_2)}=\boldsymbol{m}^i(f_1) + \boldsymbol{m}^i(f_2) \\
& M(f_1 \circ f_2) = M(f_1) + M(f_2), %\\
%& l_i(f_1 \circ f_2) = l_i(f_1) + l_i(f_2).
\end{align*} 
\end{proposition}
\begin{proof}
We recall that the action on $\Heis_g/\langle \sigma \rangle$ induced by $\mathrm{Mod}(\Sigma_{g,1})$ is exactly the homological action of $\mathrm{Mod}(\Sigma_{g,1})$ on $H_1(\Sigma_{g,1},\Z).$

Let $[f]\in G_g,$, thus $f$ is supported on $S_g$ and in particular it stabilizes the subsurface $S_g.$  Moreover, the fundamental group of $S_g$ is a free group of rank $g,$ generated by the loops $\beta_1,\ldots,\beta_g,$ introduced in Notation \ref{not_Heisenberg_local_systems}. Hence the image of the loops $\underline{\beta_i} := \{ \beta_i, \xi_2, \ldots, \xi_n \}$ in $\mathrm{Conf}_n(\Sigma_{g,1})$ are compositions of loops $\underline{\beta_j}^{\pm 1},$ and therefore mapped to an element of $B=\langle b_1,\ldots,b_g \rangle.$ Note that the latter subgroup $B$ of $\Heis_g$ is isomorphic to $\Z^g.$ Moreover $f$ fixes the loops $c_i$ pointwise so that $f$ stabilizes the homology class of $\beta_i$ for any $i\in \lbrace 1, \ldots,g \rbrace.$ Hence $f_*(b_i)=b_i.$ 

Finally, note that the homology class $[a_i]$ must be sent by $f$ to another homology class with algebraic intersection $0$ with all $[b_j], i\neq j,$ and algebraic intersection $1$ with $[b_i].$ Therefore, we must have 
$$f_*(a_i)=a_ib^{\mathbf{m}^i(f)}\sigma^{l_i(f)},$$
for some $\mathbf{m}^i(f)\in \Z^g$ and some $l_i(f) \in \Z.$	
	
The fact that the matrix $M(f)$ is symmetric is a consequence of the fact that the homological action of $f$ on $H_1(\Sigma_{g,1})$ respects the intersection form. Indeed, one has
$$0=\omega([f(a_i)],[f(a_j)])=\omega([a_i]+\underset{k}{\sum}m^i_k(f)[b_k],[a_j]+\underset{k}{\sum}m^j_k(f)[b_k])=m^j_i(f)-m^i_j(f).$$

Finally, it is clear from the formula for $f_*$ that $f\longrightarrow M(f)$ is a morphism. %and $f\longrightarrow l_i(f)$ are morphisms.
%{\color{red} NB: JE PENSE QUE $l_i$ est zero ? Une tentative pour $l_i(f)=0$... \\

Now let $A_i^- := \alpha_i \cap \{z\le 0\}$. Notice that $A_i^-$ is fixed by $f$. Let $P_i$ and $Q_i$ be the ends of $A_i^-$. There is an ambient isotopy of $\Sigma_{g,1}$ supported on $S_g$ that rejoins $P_i$ and $Q_i$. This shows that $f(\alpha_i)$ is isotopic to $\alpha_i$ followed by a loop in $S_g$. Then if $\alpha'_i$ is the conjugation of $\alpha_i$ (traveled with negative $z$ first) by a path from $D$ to $\alpha_i$ in the plane $\{z=0\}$, we recall that $\underline{\alpha'_i} := \{\alpha'_i, \xi_2, \ldots, \xi_n \}$ is a generator sent to $a_i$ by the local system $\varphi_n^{\Heisg}$. Its image is thus a loop with image $a_i$ composed with a loop with image in $B$ thanks to the ambient isotopy just discussed. This implies that $l_i=0,$ and concludes the proof. 
\end{proof}

Let $h=a_1^{x_1} \cdots a_g^{x_g} b_1^{y_1} \cdots b_g^{y_g} \sigma^{s} \in \Heis_g$, according to the proposition, we have:
\[
f_*(h) = a_1^{x_1} \cdots a_g^{x_g} \prod_i b_i^{y_i + \sum_j x_j m^j_i(f)} \sigma^{s}     %+\sum_j l_j(f)}
\]
and as a consequence we refine the restriction of \eqref{E:mcg_on_Heisg_recall} to:
\[
\begin{array}{ccc}
G_g & \to & \Z^{g(g+1)/2} \\ % \oplus \Z^{g} \\
f & \mapsto & f_* = (m^j_i(f))_{1 \le i \le j \le g}.%  \oplus (l_i(f))_{i \in \{1, \ldots , g \}}.
\end{array}
\]
It is easy to check that this is surjective, namely that $M(\tau_{c_i}) = E_{i,i}$ and $M(\tau_{t_{i,j}}) = E_{i,j} + E_{j,i}+E_{i,i}+E_{j,j},$ which recovers generators of symmetric matrices (with a standard notation used for the canonical generators of spaces of matrices).
\begin{lemma}
The group $[G_g,G_g]$ is in $K_{\Heis_g}$.
\end{lemma} 
\begin{proof}
Since the image of $G_g$ in $\Aut(\Heis_g)$ is an abelian group, $[G_g,G_g]$ acts trivially on $\Heis_g$ which is the definition of $K_{\Heis_g}$. 
\end{proof}

Restricted to this subgroup, the representation $\rho_n^{\Heis_g}$ admits a subrepresentation. Indeed, let 
$$\calH^{S_g}_{n}=H_n^{BM}(\mathrm{Conf}_n(S_g),\mathrm{Conf}_n(S_g)^-;(\varphi_n^{\Heis_g})_{|\pi_1(\Conf_n(S_g))})$$
where $\mathrm{Conf}_n(S_g)^-$ denotes configurations of $n$ points in $S_g$ such that at least one point belongs to the interval $\partial^-\Sigma_{g,1}.$ It is clear that $\mathrm{Mod}(S_g)$ stabilizes this subspace, since $\mathrm{Mod}(S_g)$ stabilizes $S_g.$ Let $L$ be the image of $(\varphi_n^{\Heis_g})_{|\pi_1(\Conf_n(S_g))}$, we note that following \eqref{E:structure_holed_disk}, the $\Z[L]$-modules $\calH^{S_g}_n$ are free. Let us call $\rho_n^{S_g}$ this subrepresentation of $\rho^{\Heis_g}_n|_{\mathrm{Mod}(S_g)}.$ 

%We recall the $n$-th Lawrence representation is the homological representation of $\mathcal{B}_k$:
%$$\mathcal{L}_{k,n} : \mathcal{B}_k \to \Aut_{\Z[s^{\pm 1}, \sigma^{\pm 1}]} \left( H_n^{BM}(\mathrm{Conf}_n(D_k),\mathrm{Conf}_n(D_k)^-;\varphi) \right).$$ In the above, the local system on $\mathrm{Conf}_n(D_k)$ such that $\varphi((\gamma,\xi_2,\ldots,\xi_n))=q$ if $\gamma$ is a loop based at $\xi_1$ going once counterclockwise around a puncture, and $\varphi(\sigma_{i,i+1})=\sigma,$ if $\sigma_{i,i+1}$ is the standard braid generator of $D_{k+n}$ that exchanges point $\xi_i$ and $\xi_{i+1}.$ These precise relative versions of Lawrence representations are defined in \cite[Lemma~6.33]{JulesVerma}.
% Our claim is that the representation $\rho_n^S$ recovers an evaluated Lawrence representation of a pure braid group. 
\begin{theorem}\label{T:Pg=full_Lawrence}
	\label{thm:pure_braid_group2} The representation $\rho_n^{S_g}$ of $\mathrm{Mod}(S_g)$ induces a representation of the pure braid group $P_g$ with coefficients in a ring of Laurent polynomials $\Z[b_1^{\pm 1}, \ldots, b_g^{\pm 1},\sigma^{\pm 1}]$. It is isomorphic to the $n$-th colored Lawrence representation $\mathcal{L}^{c,\circ}_{g,n}$ of $\mathcal{PB}_g$ (from \eqref{E:colored_Lawrence}, see Rem.~\ref{R:universal_construction_of_Lawrence} for the case with holes rather than punctures) under the obvious isomorphism of rings $$\Z[b_1^{\pm 1}, \ldots, b_g^{\pm 1},\sigma^{\pm 1}]\simeq \Z[s_1^{\pm 1}, \ldots, s_g^{\pm 1},\sigma^{\pm 1}].$$.
\end{theorem}
\begin{proof}
First notice that $L$ is here the subgroup $\langle b_1, \ldots , b_g , \sigma \rangle$ of $\Heis_g$ since $\pi_1(\Conf_n(S_g))$ is generated by classes of loops $\underline{\beta_i}$'s described above and braids $\sigma_{i,i+1}$ described in Not.~\ref{not_Heisenberg_local_systems}. The subgroup $L$ is abelian, which justifies that $\Z[L]=\Z[b_1^{\pm 1}, \ldots, b_g^{\pm 1},\sigma^{\pm 1}]$. 

According to Prop.~\ref{P:action_of_Gg_on_Heis}, one notices that the action of $G_g$ on $L$ is trivial, so that it results in a (non-crossed) $\Z[L]$-linear representation $\rho_n^{S_g}$ of $G_g$ on $\calH^{S_g}_{n}$. 

We also note that $\tau_{c_1},\ldots,\tau_{c_n}$ act trivially on $\calH^{S_g}_{n},$ which means that acts of $\rho_n^{S_g}$ can be considered a representation of $P_g$ rather than of $G_g.$

	%First, notice that $\tau_{c_1},\ldots,\tau_{c_n}$ act trivially on $\calH_n^S,$ since $S$ lies in the exterior of each of the separating curves $c_1,\ldots,c_n.$
%	So $\rho_n^S$ can be considered a representation of 
%	$$\mathcal{PB}_g\simeq \mathrm{Mod}(S)/\langle \tau_{c_1},\ldots,\tau_{c_g} \rangle.$$
	Following the construction described in Remark~\ref{R:universal_construction_of_Lawrence}, it is sufficient to check that the twisted intersection form of $\rho_n^{S_g}$ is the twisted intersection form for the Lawrence representation.
	By the definition of the twisted intersection form given in Section \ref{sec:inter_form_def}, we only need to check that the local system
	$$\varphi_n^{\Heis_g}|_{\pi_1(\Conf_n(S_g))}: \pi_1(\Conf_n(S_g))\longrightarrow L$$
	coincides with the one for the Lawrence representation (denoted $\varphi$ in \eqref{E:colored_Lawrence}). It is obvious after identifying $b_i$'s with $s_i$'s. 
\end{proof}
As a consequence, we found a subgroup of $K_{\Heis_g}$ that acts faithfully under $\rho_n^{\Heis_g}$ for $n>1$. 
\begin{coro}\label{C:Pg_is_faithful}
The subgroup $[P_g,P_g]$ of $\Mod(\Sigma_{g,1})$ acts faithfully under $\rho_n^{\Heis_g}$ for $n>1$. Hence the crossed action $\tilde{\rho}_n^{\Heis_g}$ is faithful on $P_g$.
\end{coro}
\begin{proof}
We claim that the representations $\mathcal{L}^{c,\circ}_{g,n}$ are faithful for $n>1$ (Remark~\ref{R:Lawrence_faithfulness}). 

%, since the uncolored one $\mathcal{L}_{g,n}$ is that is a specialization of $\mathcal{L}^c_{g,n}$. The case $n=2$ is faithful on the subrepresentation restricted to the absolute homology, this is the result of Bigelow \cite{Big00}. The fact that the absolute Lawrence representations are subrepresentations of the $\mathcal{L}_{g,n}$ defined here using a relative homology is a consequence of the short exact sequence in the proof of \cite[Corollary~7.1]{JulesVerma}. Notice that $\calH_n^{S_g}$ surjects on $\calH_{n-1}^{S_g}$ which is also a consequence of the mentioned short exact sequence. It concludes the proof for $n>2$. 

Notice that $\rho_n^{\Heis_g}$ is not defined on $P_g$ since $P_g \not\subset K_{\Heis_g}$, only its restriction to the submodule involved in Theorem~\ref{thm:pure_braid_group2} is. Nevertheless, the crossed action $\tilde{\rho}_n^{\Heis_g}$ from Proposition~\ref{prop_uncross_in_general} is defined on $P_g$, its kernel is in $K_{\Heis_g}$, and $[P_g,P_g]=P_g \cap K_{\Heis_g}$ acts faithfully which concludes the proof.

\end{proof}

We note that $\tilde{\rho}_n^{\Heisg}$ being faithful on such a sugroup $P_g$ is a very positive sign in the direction of $\rho_n^{\Heisg}$ being faithful. 

\subsection{Evaluated Lawrence representations as subrepresentations of a pure braid group on $2g$ strands in $\Mod(\Sigma_{g,1})$}
\label{sec:P2g}

We let $\mathcal{V}_{2g}$ be the subsurface of $\Sigma_{g,1}$ depicted in green in the following picture:
\[
\vcenter{\hbox{\begin{tikzpicture}[scale=0.35, every node/.style={scale=0.75},decoration={
    markings,
    mark=at position 0.5 with {\arrow{>}}}
    ]
    \fill[green!80!black] (-10,-5) rectangle ++(1.5,10);
    \fill[green!80!black] (8.5,-5) rectangle ++(1.5,10);
    \fill[green!80!black] (-5.5,-5) rectangle ++(1,10);
    \fill[green!80!black] (4.5,-5) rectangle ++(1,10);
    \fill[green!80!black] (-10,3.5) rectangle ++(20,1.5);
    \fill[green!80!black] (-10,-5) rectangle ++(20,1.5);
    \draw[thick] (-10,-5) rectangle ++(20,10);
    \coordinate (a1g) at (-7,0);
    \coordinate (a1d) at (-3,0);
    \coordinate (agg) at (3,0);
    \coordinate (agd) at (7,0);
    \draw[thick,fill=gray!50!white] (a1g) circle (1cm);
    \draw[thick,fill=gray!50!white] (a1d) circle (1cm);
    \draw[thick,fill=gray!50!white] (agg) circle (1cm);
    \draw[thick,fill=gray!50!white] (agd) circle (1cm);
    \node at (a1g) {$1$};
    \node at (a1d) {\reflectbox{$1$}};
    \node at (agg) {$g$};
    \node at (agd) {\reflectbox{$g$}};
    \node at (0,0) {$\cdots$};
    \draw[white,very thick] (-10,-5)--(10,-5);
    \draw[magenta,very thick] (-10,-5)--(10,-5);
    \draw[thick,orange] (-8.5,-3.5) rectangle ++(3,7);
    \node[orange,below] at (-7,3.5) {$c_1$};
    \draw[thick,orange] (-4.5,-3.5) rectangle ++(3,7);
    \node[orange,below] at (-3,3.5) {$d_1$};
    \draw[thick,orange] (5.5,-3.5) rectangle ++(3,7);
    \node[orange,below] at (7,3.5) {$d_g$};
    \draw[thick,orange] (1.5,-3.5) rectangle ++(3,7);
    \node[orange,below] at (3,3.5) {$c_g$};
\end{tikzpicture}}}
\]
The surface $\mathcal{V}_{2g}$ can also be described as a regular neighborhood of the union of $\partial \Sigma_{g,1}$ and $2g-1$ arcs that separate the disks onto which we glued the $g$ handles. It is homeomorphic to a disk with $2g$ holes. We denote by $c_1,d_1,\ldots,c_g,d_g$ the boundary curves of $\mathcal{V}_{2g}$ (other than $\partial\Sigma_{g,1}$). Notice that none of the boundary components of $\mathcal{V}_g$ bounds a disk in $\Sigma_{g,1}$. Nevertheless the boundary components $c_i$ and $d_i$ cobound a cylinder in $\Sigma_{g,1}$ for all $i\in\{1,\ldots,g\}$. According to \cite[Theorem~3.18]{FarbMargalit}, denoting $\mathcal{G}_{2g}:=\Mod(\mathcal{V}_{2g})$, there is a map:
\[
\eta : \mathcal{G}_{2g} \to \Mod(\Sigma_{g,1}),
\]
and its kernel is the free abelian group generated by the products of Dehn twists of the form $\tau_{c_i} \tau_{d_i}^{-1}$, for all $i\in\{1,\ldots ,g\}$. Now notice that $\mathcal{G}_{2g}$ is isomorphic to a framed pure braid group on $2g$ strands, so that:
\[
\mathcal{P}_{2g} := \mathcal{G}_{2g} \big/ \langle \tau_{c_1} , \tau_{d_1}, \ldots, \tau_{c_g},\tau_{d_g} \rangle
\]
is thus isomorphic to a pure braid group on $2g$ strands. There is a section to the exact sequence $1\rightarrow \langle \tau_{c_1} , \tau_{d_1}, \ldots, \tau_{c_g},\tau_{d_g} \rangle  \rightarrow G_{2g} \rightarrow P_{2g} \rightarrow 1,$ so $\mathcal{P}_{2g}$ can be considered a subgroup of  $\mathcal{G}_{2g},$ and even of  $\Mod(\Sigma_{g,1})$ since  $\ker \eta \cap P_{2g}$ is trivial. 

\begin{proposition}\label{P:action_of_Imeta_on_Heis}
Let $[f] \in \im(\eta)$, for any $i \in \{1,\ldots,g\}$, we have: %, let $B:=\langle b_1 , \ldots , b_g \rangle \triangleleft \Heis_g$,
\begin{align*}
& f_*(a_i) = a_i \\ 
& f_*(b_i) = b_i a^{\boldsymbol{m}^i(f)} \sigma^{l_i(f)},
\end{align*}
where $l_i \in \Z$ and:
\[
a^{\boldsymbol{m}^i(f)} = a_1^{m^i_1(f)} \cdots a_g^{m^i_g(f)},
\]
so that ${\boldsymbol{m}^i(f)}$ is a notation for $(m^i_1(f),\ldots,m^i_g(f)) \in \Z^g$. 
Let $M(f)$ be the matrix with $i$-th column being ${\boldsymbol{m}^i(f)}$. We furthermore have that $M(f)$ is symmetric and that for $f_1,f_2 \in G$: 
\begin{align*}
%& {\boldsymbol{m}^i(f_1 \circ f_2)}=\boldsymbol{m}^i(f_1) + \boldsymbol{m}^i(f_2) \\
& M(f_1 \circ f_2) = M(f_1) + M(f_2), %\\
& l_i(f_1 \circ f_2) = l_i(f_1) + l_i(f_2).
\end{align*} 
\end{proposition}
\begin{proof}
The proof is similar to the proof of Proposition \ref{P:action_of_Gg_on_Heis}.
Again, the action on $\Heis_g/\langle \sigma \rangle$ induced by $\mathrm{Mod}(\Sigma_{g,1})$ is the homological action of $\mathrm{Mod}(\Sigma_{g,1})$ on $H_1(\Sigma_{g,1},\Z).$ Since $f$ is supported on $\mathcal{V}_{2g}$, it stabilizes it. Moreover, the fundamental group of $\mathcal{V}_{2g}$ is a free group of rank $2g,$ generated by the loops $c'_1,d'_1,\ldots,d'_g,$ obtained from the corresponding curve conjugated by a path from the leftmost boundary side in the picture to the curve that is supported in $\mathcal{V}_{2g}$. In these loops we suppose that the curves $c_i$'s and $d_i$ are traveled counterclockwise regarding the planar picture orientation. If one defines $\underline{c_i} = \{ c'_i, \xi_2,\ldots,\xi_n\}$ and $\underline{d_i} = \{ d'_i, \xi_2,\ldots,\xi_n\}$ to be loops in $\Conf_n(\Sigma_{g,1})$ notice that:
\[
\varphi_n^{\Heisg} ( \underline{c_i} ) = a_i, \text{  and } \varphi_n^{\Heisg} ( \underline{d}_i ) = a_i^{-1} \sigma^{-2}. 
\]
see Notation \ref{not_Heisenberg_local_systems}. Hence the image of the loops $\underline{c_i}$ by $f$ are sent to compositions of loops $\underline{c_i}^{\pm 1}$ and $\underline{d_i}^{\pm 1}$, but stabilizes $c_i$'s (and $d_i$'s) pointwise so that $f_*(a_i)=a_i$. As in the proof of Proposition \ref{P:action_of_Gg_on_Heis}, the homology class $[b_i]$ must be sent by $f$ to another homology class with algebraic intersection $0$ with all $[a_j], i\neq j,$ and algebraic intersection $1$ with $[a_i].$ Therefore, 
$$f_*(b_i)=b_ia^{\mathbf{m}^i(f)}\sigma^{l_i(f)},$$
for some $\mathbf{m}^i(f)\in \Z^g$ and some $l_i(f) \in \Z.$	
The fact that the matrix $M(f)$ is symmetric, and that $f_*$ that $f\longrightarrow M(f)$ and $f\longrightarrow l_i(f)$ are morphisms, is proved as in the proof of Proposition \ref{P:action_of_Gg_on_Heis}.
\end{proof}

Let $h=a_1^{x_1} \cdots a_g^{x_g} b_1^{x_1} \cdots b_g^{x_g} \sigma^{s} \in \Heis_g$, according to the proposition, we have:
\[
f_*(h) = a_1^{x_1} \cdots a_g^{x_g} \prod_i b_i^{y_i + \sum_j m^j_i(f)} \sigma^{s+\sum_j l_j(f)}
\]
and as a consequence we refine the restriction of \eqref{E:mcg_on_Heisg_recall} to:
\[
\begin{array}{ccc}
\im(\eta) & \to & \Z^{g(g+3)/2} \\ % \oplus \Z^{g} \\
f & \mapsto & f_* = (m^j_i(f))_{1 \le i \le j \le g} \oplus (l_i(f))_{i \in \{1, \ldots , g \}}.
\end{array}
\]

\begin{lemma}
The group $[\im(\eta),\im(\eta)]$ is in $K_{\Heis_g}$.
\end{lemma} 
\begin{proof}
Since the image of $\im(\eta)$ in $\Aut(\Heis_g)$ is an abelian group, its derived subgroup acts trivially on $\Heis_g$ which is the definition of $K_{\Heis_g}$. 
\end{proof}

Restricted to $[\mathcal{P}_{2g},\mathcal{P}_{2g}]$, the representation $\rho_n^{\Heis_g}$ admits a subrepresentation. Indeed, let 
$$\calH^{\mathcal{V}_g}_{n}=H_n^{BM}(\mathrm{Conf}_n(\mathcal{V}_{2g}),\mathrm{Conf}_n(\mathcal{V}_{2g})^-;(\varphi_n^{\Heis_g})_{|\pi_1(\Conf_n(\mathcal{V}_{2g}))})$$
where $\mathrm{Conf}_n(\mathcal{V}_{2g})^-$ denotes configurations of $n$ points in $\mathcal{V}_{2g}$ such that at least one point belongs to the interval $\partial^-\Sigma_{g,1}.$ It is clear that $\im(\eta)$ stabilizes this subspace. Let $L$ be the image of $(\varphi_n^{\Heis_g})_{|\pi_1(\Conf_n(\mathcal{V}_{2g}))}$, we note that following \eqref{E:structure_holed_disk}, the $\Z[L]$-modules $\calH^{\mathcal{V}_{2g}}_n$ are free. Let us call $\rho_n^{\mathcal{V}_{2g}}$ this subrepresentation of $\rho^{\Heis_g}_n|_{\im(\eta)}.$ 

%We recall the $n$-th Lawrence representation is the homological representation of $\mathcal{B}_k$:
%$$\mathcal{L}_{k,n} : \mathcal{B}_k \to \Aut_{\Z[s^{\pm 1}, \sigma^{\pm 1}]} \left( H_n^{BM}(\mathrm{Conf}_n(D_k),\mathrm{Conf}_n(D_k)^-;\varphi) \right).$$ In the above, the local system on $\mathrm{Conf}_n(D_k)$ such that $\varphi((\gamma,\xi_2,\ldots,\xi_n))=q$ if $\gamma$ is a loop based at $\xi_1$ going once counterclockwise around a puncture, and $\varphi(\sigma_{i,i+1})=\sigma,$ if $\sigma_{i,i+1}$ is the standard braid generator of $D_{k+n}$ that exchanges point $\xi_i$ and $\xi_{i+1}.$ These precise relative versions of Lawrence representations are defined in \cite[Lemma~6.33]{JulesVerma}.
% Our claim is that the representation $\rho_n^S$ recovers an evaluated Lawrence representation of a pure braid group. 
\begin{theorem}\label{T:P2g=evaluated_Lawrence}
	\label{thm:pure_braid_group3} The representation $\rho_n^{\mathcal{V}_{2g}}$ of $\im(\eta)$ induces a representation of the pure braid group $\mathcal{P}_{2g}$ with coefficients in a ring of Laurent polynomials $\Z[a_1^{\pm 1}, \ldots, a_g^{\pm 1},\sigma^{\pm 1}]$. It is isomorphic to the $n$-th colored Lawrence representation $\mathcal{L}^{c,\circ}_{2g,n}$ of $\mathcal{PB}_{2g}$ (from \eqref{E:colored_Lawrence}, see Rem.~\ref{R:universal_construction_of_Lawrence} for the case with holes rather than punctures) evaluated by the following morphism of rings: 
$$
\begin{array}{ccc}
\Z[s_1^{\pm 1}, \ldots, s_{2g}^{\pm 1},\sigma^{\pm 1}] & \to &  \Z[a_1^{\pm 1}, \ldots, a_g^{\pm 1},\sigma^{\pm 1}] \\
s_{2k-1} & \mapsto & a_k, \\
s_{2k} & \mapsto & a_k^{-1}\sigma^{-2}, \\
\sigma & \mapsto & \sigma
\end{array}
$$
where $k$ is any integer in $\{1,\ldots,g\}$.
\end{theorem}
\begin{proof}
First notice that $\Z[L]$ is here $\Z[a_1^{\pm 1}, \ldots, a_g^{\pm 1},\sigma^{\pm 1}]$ since $\pi_1(\Conf_n(\mathcal{V}_{2g}))$ is generated by classes of loops $\underline{c_i}$'s and $\underline{d_i}$'s described above and braids $\sigma_{i,i+1}$ described in Not.~\ref{not_Heisenberg_local_systems}. %Latter subgroups is abelian, which justifies that $\Z[L]=\Z[b_1^{\pm 1}, \ldots, b_g^{\pm 1},\sigma^{\pm 1}]$. 

According to Prop.~\ref{P:action_of_Imeta_on_Heis}, one notices that the action of $\im(\eta)$ on $L$ is trivial, so that it results in a (non-crossed) $\Z[L]$-linear representation $\rho_n^{\mathcal{V}_{2g}}$ of $\im(\eta)$ on $\calH^{\mathcal{V}_{2g}}_{n}$. Following the construction described in Remark~\ref{R:universal_construction_of_Lawrence}, it is sufficient to check that the twisted intersection form of $\rho_n^{\mathcal{V}_{2g}}$ is the twisted intersection form for the Lawrence representation. By the definition of the form given in Section \ref{sec:inter_form_def}, we only need to check that the local system
	$$\varphi_n^{\Heis_g}|_{\pi_1(\Conf_n(\mathcal{V}_{2g}))}: \pi_1(\Conf_n(\mathcal{V}_{2g}))\longrightarrow L$$
coincides with the one for the Lawrence representation (denoted $\varphi$ in \eqref{E:colored_Lawrence}). It is clear that it is true at the evaluation mentioned in the statement. 
\end{proof}


\begin{thebibliography}{99}

\bibitem[B-N]{BarNatanGassner}
D.Bar-Natan, \textit{A Note on the Unitarity Property of the Gassner Invariant}, Bulletin of Chelyabinsk State University
(Mathematics, Mechanics, Informatics) 3-358-17 (2015), 22–25.

 \bibitem[BGG11]{BGG11}
 P.~Bellingeri, E.~Godelle, J.~Guaschi,
 \textit{Exact Sequences, Lower Central Series and Representations of Surface Braid Groups},
 \arxiv{1106.4982}{[math.GT]}.
 
 \bibitem[BMW22]{Bianchietal}
 A. Bianchi, J. Miller, J. Wilson, {\em Mapping class group actions on configuration spaces and the Johnson filtration}, Transactions of the American Mathematical Society 375:8 (2022) 5461-5489.

 \bibitem[Big99]{Big99}
 S.~Bigelow,
 \textit{The Burau Representation Is Not Faithful for $n = 5$}, 
 \doi{10.2140/gt.1999.3.397}{Geom. Topol. \textbf{3} (1999), no.~1, 397--404};
 \arxiv{math/9904100}{[math.GT]}.

 \bibitem[Big00]{Big00}
 S.~Bigelow,
 \textit{Braid Groups Are Linear},
 \doi{10.1090/S0894-0347-00-00361-1}{J. Amer. Math. Soc. \textbf{14} (2001), no.~2, 471--486};
 \href{http://arXiv.org/abs/math/0005038}{\texttt{arXiv:math/0005038} [math.GR]}.
 
 \bibitem[Big04]{Big04}
 S. Bigelow, 
 {\em Homological representations of the Iwahori-Hecke algebra}, 
 \doi{}{Geom. Topol. Monographs {\bf 7}, (2004), 493--507.}
 
 \bibitem[BB01]{BB01}
 S.~Bigelow, R.~Budney, 
 \textit{The Mapping Class Group of a Genus Two Surface Is Linear},
 %\doi{10.2140/agt.2001.1.699}
 {Algebr. Geom. Topol. \textbf{1} (2001), no.~2, 699--708};
 %\arxiv{math/0010310}{[math.GT]}.



\bibitem[Bir]{Bir}
J. Birman, \textit{Braids, Links and Mapping Class Groups}, Annals of Mathematics Studies 82 Princeton University Press,
Princeton, New Jersey, (1975).

 \bibitem[BCGP14]{BCGP14}
 C.~Blanchet, F.~Costantino, N.~Geer, B.~Patureau-Mirand,  
 \textit{Non-Semisimple TQFTs, Reidemeister Torsion and Kashaev's Invariants}, 
 \doi{10.1016/j.aim.2016.06.003}{Adv. Math. \textbf{301} (2016), 1--78};
 \arxiv{1404.7289}{[math.GT]}.

 \bibitem[BHMV95]{BHMV}
 C. Blanchet, N. Habegger, G. Masbaum, P. Vogel,
 \textit{Topological Quantum Field Theo-ries Derived from the Kauman Bracket},
 Topology {\bf 34} (1995), No. 4, 883--927.
 
 \bibitem[BPS21]{BPS21}
 C.~Blanchet, M.~Palmer, A.~Shaukat,
 \textit{Heisenberg homology on surface configurations},
 \arxiv{2109.00515}{[math.GT]}.
 
 \bibitem[DPS]{DPS}
 J. Darné, M. Palmer, A. Soulié, \textit{When the lower central series stops} accepted for publication at the memoirs of the AMS.
 
  \bibitem[DGGPR20]{DGGPR20}
 M.~De~Renzi, A.~Gainutdinov, N.~Geer, B.~Patureau-Mirand, I.~Runkel,
 \textit{Mapping Class Group Representations from Non-Semisimple TQFTs}, 
 \doi{10.1142/S0219199721500917}{Commun. Contemp. Math. (2021), 2150091};
 \arxiv{2010.14852}{[math.GT]}.
 
 \bibitem[DrMa]{JulesMarco}
 M.~De Renzi, J. Martel,
 \textit{Homological Construction of Quantum Representations of Mapping Class Groups},
 \arxiv{2212.10940}{[math.GT]}.
 
 \bibitem[DMW]{JulesMarcoBangxin}
 M. De Renzi, J. Martel, B. Wang, \textit{Hennings TQFTs for cobordisms decorated with cohomology classes}, preprint, arXiv:2402.05103. 
 
 \bibitem[D22]{BurauKernel}
 E. Dlugie, \textit{The Burau Representation and Shapes of Polyhedra}, preprint, arXiv:2210.06561 [math.GT].

 \bibitem[FM]{FarbMargalit}
 B. Farb, D. Margalit, 
 \textit{A Primer on Mapping Class Groups}, 
 Princeton Mathematical Series 49, Princeton University Press, Princeton, NJ, 2012.
 
 \bibitem[G24]{Pierre}
P. Godfard, \textit{Construction of Hodge structures on the $\mathrm{SO}(3)$ modular functors}, preprint, arXiv:2402.16804 [math.GT]. 
 
 \bibitem[Har81]{Harvey}
 J.W. Harvey, \textit{Boundary structure of the modular group.} Ann. of Math. Stud., 97, Princeton Univ. Press, Princeton, N.J., 1981.
 
 \bibitem[Knu]{Knu} 
 K. Knudson, \textit{On the kernel of the Gassner representation}, Archiv der Mathematik 85, (2005), 108–117.
 
 \bibitem[Kr02]{K02}
 D.~Krammer, 
 \textit{Braid Groups Are Linear}, 
 %\doi{10.2307/3062152}
 {Ann. of Math. (2) \textbf{155} (2002), no.~1, 131--156};
 %\arxiv{math/0405198}{[math.GR]}. 
 
 \bibitem[La90]{L90}
 R.~Lawrence, 
 \textit{Homological Representations of the Hecke Algebra}, 
 \doi{10.1007/BF02097660}{Comm. Math. Phys. \textbf{135} (1990), no.~1, 141--191}.
  
 
 \bibitem[L-P]{LongPaton}
 D. Long, M. Paton, \textit{The Burau representation is not faithful for $n \ge 6$}, Topology, Vol. 32 No. 2 (1993), 439--447. 
 
 \bibitem[Ly94]{L94} 
 V.~Lyubashenko,
 \textit{Invariants of $3$-Manifolds and Projective Representations of Mapping Class Groups via Quantum Groups at Roots of Unity},
 \doi{10.1007/BF02101805}{Comm. Math. Phys. \textbf{172} (1995), no.~3, 467--516};
 \href{http://arXiv.org/abs/hep-th/9405167}{\texttt{arXiv:hep-th/9405167}}. 
 
 \bibitem[Ma39]{Magnus}
 W.~Magnus, 
 \textit{On a Theorem of Marshall Hall},
 %\doi{10.2307/1968892}
 {Ann. of Math. (2) \textbf{40} (1939), no.~4, 764--768}. 
 
 \bibitem[Ma18]{Margalit}
 D.~Margalit, 
 \textit{Problems, Questions, and Conjectures About Mapping Class Groups}, 
 {Proceedings of Symposia in Pure Mathematics \textbf{102}, American Mathematical Society, Providence, RI, 2019, 157--186}.
 %\arxiv{1806.08773}{[math.GT]}.
 

 \bibitem[Ma39]{M39}
 W.~Magnus, 
 \textit{On a Theorem of Marshall Hall},
 \doi{10.2307/1968892}{Ann. of Math. (2) \textbf{40} (1939), no.~4, 764--768}. 
 
 \bibitem[Ma20]{JulesVerma}
 J.~Martel,
 \textit{A Homological Model for $\Uq$ Verma-Modules and Their Braid Representations},
 \doi{10.2140/gt.2022.26.1225}{Geom. \& Topol. \textbf{26-3} (2022), 1225--1289}
 \arxiv{2002.08785}{[math.GT]}.
 
  
 \bibitem[Ma20b]{Jules2}
 J. Martel,
 \textit{Colored version for Lawrence representations}, to appear at NYJ Maths.% , ArXiv preprint (April 2020).

  \bibitem[Mo07]{Mo07}
 T.~Moriyama,
 \textit{The Mapping Class Group Action on the Homology of the Configuration Spaces of Surfaces},
 %\doi{10.1112/jlms/jdm077}{J. Lond. Math. Soc. (2) \textbf{76} (2007), no.~2, 451--466}.
 Journal of the London Mathematical Society, Volume 76, Issue 2, October 2007, pages 451–466. 

 \bibitem[Put]{Put}
 A. Putman, \textit{The commutator subgroups of free and surface groups},
 Enseign. Math.68(2022), no.3-4, 389–408

\bibitem[RS11]{RafiSchleimer} K. Rafi and S. Schleimer \textit{Curve complexes are rigid}, Duke Math. J. 158 (2011), no. 2, 225–246.

 \bibitem[Su02]{S02}
 M.~Suzuki,
 \textit{The Magnus Representation of the Torelli Group $\mathcal{I}_{g,1}$ Is Not Faithful for $g \ge 2$},
 \doi{10.1090/S0002-9939-01-06128-7}{Proc. Amer. Math. Soc. \textbf{130} (2002), no.~3, 909--914}. 
 
 \bibitem[Su05a]{S05a}
 M.~Suzuki,
 \textit{Geometric Interpretation of the Magnus Representation of the Mapping Class Group},
 Kobe J. Math. \textbf{22} (2005), 39--47.
  
 \bibitem[Su05b]{S05b}
  M.~Suzuki,
 \textit{On the Kernel of the Magnus Representation of the Torelli Group},
 \doi{10.1090/S0002-9939-04-07766-4}{Proc. Amer. Math. Soc. \textbf{133} (2005), no.~6, 1865--1872}.


\end{thebibliography}
\end{document}